\documentclass[10pt]{article}
\usepackage[utf8]{inputenc}
\usepackage[toc,page]{appendix}
\usepackage{geometry}
\usepackage{bbm}
\usepackage{amsmath,amsthm,amssymb,amsfonts}
\usepackage{mathrsfs} 
\usepackage[colorlinks, linkcolor=blue]{hyperref}
\newtheorem{theorem}{Theorem}
\newtheorem{lemma}[theorem]{Lemma}
\newtheorem{remark}[theorem]{Remark}

\newtheorem{proposition}[theorem]{Proposition}
\newtheorem{corollary}[theorem]{Corollary}

\numberwithin{equation}{section}
\numberwithin{theorem}{section}
\usepackage{graphicx}
\usepackage{xcolor}
\usepackage{academicons}
\usepackage{orcidlink}
\usepackage{enumitem,lipsum}
\providecommand{\keywords}[1]
{
  \small	
  \textbf{Keywords } #1
}
\providecommand{\thank}[1]
{
  \small	
  \textbf{Acknowledgment } #1
}

\title{\textbf{Scattering of the defocusing Calogero--Moser derivative nonlinear Schrödinger equation}
}
\date{}
\author{Xi Chen\thanks{Department of Mathematics and Computer Science, University of Basel. \\ \hspace*{5.5mm}  Email: \texttt{xi01.chen@unibas.ch}}}

\begin{document}

\maketitle
\begin{abstract}
We consider the defocusing Calogero–Moser derivative nonlinear Schrödinger equation (CM-DNLS)
\begin{align*}
i \partial_{t} u+\partial_{x}^2 u-2\Pi D\left(|u|^{2}\right)u=0, \quad (t,x ) \in \mathbb{R} \times \mathbb{R}.
\end{align*}
Using the Gérard-type explicit formula, we prove the scattering for solutions to this equation with initial data in $L_{+}^2(\mathbb{R}):=\left\{f \in L^2(\mathbb{R}): \operatorname{supp}(\widehat{f}) \subset[0,+\infty)\right\}$. We also characterize the scattering term using the distorted Fourier transform associated with the Lax operator. This is one of the first works to apply the Gérard-type explicit formula to study the long-time behavior of an integrable equation for a broad class of initial data, beyond the previously studied rational cases.
\end{abstract}
\\

\tableofcontents{}
\section{Introduction}
The long-time behavior of a dispersive equation is shaped by the competition between nonlinear interaction and dispersion. Scattering asks whether a global nonlinear solution approaches a solution of the corresponding linear equation as time tends to infinity. At critical regularity the problem is especially delicate, because the natural norm supplies no extra smoothness, pointwise decay, or spatial localization. Integrable equations provide additional algebraic and spectral structure, but even an explicit formula need not make the large-time dynamics transparent for rough initial data.

In this paper we study this question for the defocusing Calogero--Moser derivative nonlinear Schr\"odinger equation (CM-DNLS)
\begin{equation}
\label{3.1}
i \partial_t u+\partial_x^2 u-2 \Pi D\left(|u|^2\right) u=0, \quad(t, x) \in \mathbb{R} \times \mathbb{R}.
\end{equation}
Here $D=\frac1i\partial_x$, while $\Pi$ discards the negative Fourier frequencies. More precisely, $\Pi$ is the Riesz--Szeg\H{o} orthogonal projection of $L^2(\mathbb R)$ onto the subspace
\begin{align*}
L_{+}^2(\mathbb{R}):=\left\{f \in L^2(\mathbb{R}): \operatorname{supp}(\widehat{f}) \subset[0,+\infty)\right\}.
\end{align*}
It acts by
\begin{equation}
\widehat{\Pi(f)}(\xi):= 1_{\xi>0} \widehat{f}(\xi), \quad \forall f \in L^2(\mathbb{R}).
\end{equation}
The restriction to nonnegative frequencies is often called the chiral condition. For $\lambda>0$, the scaling $u(t,x)\mapsto u_\lambda(t,x):=\lambda^{1/2}u(\lambda^2t,\lambda x)$ preserves \eqref{3.1}, the $L^2$ norm, and the chiral Fourier support. Thus $L_+^2(\mathbb R)$ is the scale-critical chiral phase space considered here.

Our main theorem states that, for every $u_0\in L_+^2(\mathbb R)$, the corresponding global solution of \eqref{3.1} scatters in $L^2$ in both time directions. In other words, the nonlinear solution approaches a free Schr\"odinger evolution as $t\to\pm\infty$. No additional Sobolev regularity, spatial decay, or moment condition is imposed. We characterize the negative-time scattering state through the distorted Fourier transform associated with $L_{u_0}$; the positive-time state is then obtained by time reversal.

The integrable structure enters through a G\'erard-type explicit formula, which expresses the nonlinear flow through the resolvent of an operator built from the initial Lax operator $L_{u_0}$. We construct a distorted Fourier transform that diagonalizes $L_{u_0}$ and use it to rewrite this formula in spectral variables. In that representation, the difference between the solution and its candidate free evolution becomes an oscillatory defect term. We first estimate this term for smooth data and then pass to arbitrary data in $L_+^2(\mathbb R)$ by approximation. 

The defocusing equation was introduced in \cite{60} under the name intermediate nonlinear Schr\"odinger equation. Its focusing counterpart appeared in \cite{22} as a formal continuum limit of the classical Calogero--Moser systems \cite{37,38}. In \cite{5}, G\'erard and Lenzmann studied the focusing CM-DNLS equation
\begin{equation}
\label{1.5}
i \partial_t u+\partial_x^2 u+2\Pi D\left(|u|^2\right) u=0,\quad(t, x) \in \mathbb{R} \times \mathbb{R}.
\end{equation}
For $u\in H_+^s(\mathbb R):=H^s(\mathbb R)\cap L_+^2(\mathbb R)$ with $s$ sufficiently large, they established a Lax pair and infinitely many conservation laws. These conservation laws yielded global well-posedness in $H_+^s(\mathbb R)$ for every $s>\tfrac12$ under the small-mass condition $\|u_0\|_{L^2}^2<2\pi$. Killip, Laurens and Vi\c{s}an subsequently combined the explicit formulae for \eqref{3.1} and \eqref{1.5} with commuting-flow methods to prove global well-posedness for both equations in $H_+^s(\mathbb R)$, $s\geq0$, with the same small-mass restriction in the focusing case \cite{2}. The author proved global well-posedness for \eqref{3.1} under a non-vanishing condition at infinity and derived G\'erard-type explicit formulae for the corresponding solutions \cite{70}.

Beyond the small-mass theory, the focusing dynamics can be singular. Hogan and Kowalski constructed a class of solutions that blow up in finite or infinite time \cite{9}. Kim, Kim and Kwon subsequently constructed smooth finite-time blow-up solutions to \eqref{1.5}. Further work on \eqref{1.5} established quantized blow-up dynamics \cite{56}, a sharp classification of finite-time single-bubble blow-up dynamics \cite{59}, and a further class of finite-time blow-up solutions \cite{67}.

Several complementary spectral and explicit-formula approaches have been developed. Matsuno formulated an inverse scattering transform for the defocusing equation \eqref{3.1} with a nonzero constant boundary condition at infinity \cite{54}, while Frank and Read proposed an inverse scattering transform for the focusing equation \eqref{1.5} based on the distorted Fourier transform \cite{78}. Using G\'erard-type explicit formulae, Badreddine characterized the zero-dispersion limit for both equations with initial data in $L_+^2(\mathbb R)\cap L^\infty(\mathbb R)$ and described the limit through the branches of the multivalued solution of the inviscid Burgers--Hopf equation \cite{13}. On the circle, she proved global well-posedness for both equations in $H_+^s(\mathbb T)$, $s\geq0$, with the small-mass restriction in the focusing case \cite{11}, and classified their traveling waves and finite-gap potentials \cite{12}. Related periodic developments include explicit-formula-based numerical schemes for both equations \cite{90,89} and finite-time blow-up solutions for the periodic focusing equation \cite{66}. Further work on the intermediate nonlinear Schr\"odinger equation may be found in \cite{52,60,53,55,42,57}.

More recently, several asymptotic descriptions of the dynamics have been obtained. Kim and Kwon established soliton resolution for both finite-time blow-up and global-in-time solutions of the focusing equation \eqref{1.5} \cite{23}. In particular, they proved scattering for small-mass solutions in
\[
H^{1,1}(\mathbb{R}):=\{u\in H^1(\mathbb{R}): xu\in L^2(\mathbb{R})\}
\]
with $\|u_0\|_{L^2}^2<2\pi$, and proved soliton resolution for global solutions in the same weighted space. Their method naturally suggests scattering for defocusing solutions in $H^{1,1}(\mathbb R)$. For the Benjamin--Ono equation, Gassot, G\'erard and Miller used the explicit formula to prove soliton resolution for a broad class of initial data \cite{0}; Gassot and G\'erard subsequently treated infinite-order multisolitons \cite{0.01}.

Here we apply the G\'erard-type explicit formula to the defocusing equation for arbitrary data in $L_+^2(\mathbb R)$. To our knowledge, this is among the first applications of such a formula to the long-time behavior of an integrable equation for a broad class of initial data not restricted to rational profiles. The distorted Fourier transform associated with $L_{u_0}$ not only drives the negative-time argument but also gives an explicit spectral description of the negative-time scattering state; the positive-time statement then follows by time reversal.

Rather than analyzing the explicit formula directly, as in \cite{0}, we first rewrite it through the distorted Fourier transform. This produces a spectral representation of the defect term, to which oscillatory integral estimates can be applied. A direct adaptation of the method in \cite{0} to \eqref{3.1} might require the additional assumption $u_0\in L^1(\mathbb R)$, whereas the spectral formulation used here does not require this assumption. Thus the result reaches the scale-critical space within the chiral class covered by the presently available explicit formulae.

\subsection{The Lax pair}
\label{subsection 2.3}
We next collect the operator-theoretic ingredients needed to state the explicit formula. For $1\leq p\leq\infty$, we write $L_+^p(\mathbb R)$ for the $L^p$ functions whose distributional Fourier transforms are supported in $\xi\geq0$. In particular, $L_+^2(\mathbb R)$ is a Hardy space and may be identified with the holomorphic functions on the upper half-plane $\mathbb C_+:=\{z\in\mathbb C:\operatorname{Im}z>0\}$ satisfying
\begin{align*}
\sup _{y>0} \int_{\mathbb{R}}|f(x+i y)|^{2} d x<+\infty.
\end{align*}
The Riesz--Szeg\H{o} projection $\Pi:L^2(\mathbb R)\to L_+^2(\mathbb R)$ has both the Fourier and Cauchy representations
\begin{equation}
\label{1.10}
\forall f \in L^2(\mathbb{R}), \quad \forall z \in \mathbb{C}_{+}, \quad \Pi f (z) = \frac{1}{2 \pi} \int_{0}^{\infty} \mathrm{e}^{i z \xi} \widehat{f}(\xi) d \xi=   \frac{1}{2i\pi} \int_{\mathbb{R}} \frac{f(y)}{y-z} dy.
\end{equation}
For $b\in L^\infty(\mathbb R)$, the associated Toeplitz operator is
\begin{align*}
T_{b} f :=\Pi(b f), \quad \forall f \in L_{+}^2(\mathbb{R}).
\end{align*}
For $u\in H_+^1(\mathbb R)$, the Lax operator associated with \eqref{3.1} is
\begin{align*}
\forall f \in \operatorname{Dom}\left(L_{u}\right)=H_{+}^{1}(\mathbb{R}), \quad  L_{u} f :=D f+uT_{\bar{u}} f \text{ with } D:=\frac{1}{i} \frac{d}{d x}.
\end{align*}
For a rough potential $u\in L_+^2(\mathbb R)$, the same operator is defined through the quadratic form (see also \cite[Appendix A]{5})
\[
\mathcal{Q}_u(f,g)=\left\langle Df, g \right\rangle + \left\langle {T_{\overline u}f}, {T_{\overline u}g} \right\rangle.
\]
More precisely, $L_u$ is determined by
\[
\left\langle L_u(f), g\right\rangle :=\mathcal{Q}_u(f, g) \quad \text { for } f \in \operatorname{dom}\left(L_u\right) \text { and } g \in H_{+}^{1 / 2}(\mathbb{R}) ,
\]
with
\[
\operatorname{dom}\left(L_u\right)=\left\{f \in H_{+}^{1 / 2}(\mathbb{R}): \exists C>0 \text { s.t. }\left|\mathcal{Q}_u(f, g)\right| \leq C\|g\|_{L^2} \text { for } g \in H_{+}^{1 / 2}(\mathbb{R})\right\}.
\]
When $u\in H_+^2(\mathbb R)$, the second operator in the Lax pair is
\begin{align*}
B_u:=-u T_{\partial_x \bar{u}}+\partial_x u T_{\bar{u}}+i\left(u T_{\bar{u}}\right)^2.
\end{align*}
We also recall from \cite{1} the operator $X^*$, whose domain is
\begin{equation}
\label{1.51}
\operatorname{Dom}\left(X^{*}\right):=\left\{f \in L_{+}^2(\mathbb{R}): \exists \lambda_f \in \mathbb{C}: x f+\lambda_f \in L^2(\mathbb{R})\right\}= \left\{f \in L_{+}^{2}(\mathbb{R}):  \widehat{f} \in H^{1}(0, \infty)\right\}, 
\end{equation}
with
\[\widehat{X^{*} f}(\xi) :=i \frac{d}{d \xi}[\widehat{f}(\xi)] \mathbf{1}_{\xi>0}.\]
The operator $X^*$ is the adjoint of multiplication by $x$ on $L_+^2(\mathbb R)$, and $\left(\operatorname{Dom}(X^*),-iX^*\right)$ is maximally dissipative. Moreover, $-iX^*$ generates the adjoint contraction semigroup $\left(S(\eta)^*\right)_{\eta\geq0}$:
\begin{align*}
\forall \eta \geq 0, \quad S(\eta)^{*}=\mathrm{e}^{-i \eta X^{*}}.
\end{align*}
The boundary trace at zero satisfies
\begin{align*}
\forall f \in \operatorname{Dom}\left(X^{*}\right), \left|\widehat{f}\left(0_{+}\right)\right|^2 = -4\pi\operatorname{Im}\left\langle X^{*} f , f\right\rangle_{L^2(\mathbb{R})} \leq 4\pi \|X^{*}f\|_{L^2} \|f\|_{L^2}.
\end{align*}
Following \cite[Definition~4.3]{2}, define the boundary functional $I_+$ by
\[
\begin{aligned}
\operatorname{Dom}(I_+)
&:=\left\{f\in L^2_+(\mathbb R):
\lim_{y\to+\infty}2\pi y f(iy)\ \text{exists in }\mathbb C\right\},\\
I_+(f)&:=\lim_{y\to+\infty}2\pi y f(iy).
\end{aligned}
\]
Equivalently, \eqref{1.10} gives
\[
I_+(f)=\lim_{y\to+\infty}
y\int_0^\infty e^{-y\xi}\widehat f(\xi)\,d\xi.
\]
If $f\in\operatorname{Dom}(X^*)$, we have
\begin{equation}
\label{1.121}
I_+(f)=\widehat f(0_+).
\end{equation}
Consequently,
\begin{equation}
\label{1.9}
\forall f \in \operatorname{Dom}\left(X^*\right), \quad X^* f=x f+\frac{1}{2 i \pi} I_{+}(f).
\end{equation}

\subsection{Global well-posedness}
The critical well-posedness theory for the defocusing equation may be summarized as follows.
\begin{theorem}[\cite{5,2}]
Given $u_0 \in H_{+}^s(\mathbb{R})$ with $s\geq 0$, there exists a solution $u \in C(\mathbb{R}; H_{+}^s(\mathbb{R}))$ to the defocusing (CM-DNLS) \eqref{3.1} with $u(0,x) = u_0(x)$.
\end{theorem}
\begin{remark}
The analogous global well-posedness result holds for the focusing equation \eqref{1.5} in $H_+^s(\mathbb R)$, $s\geq0$, when $\|u_0\|_{L^2}^2<2\pi$; it also holds at the threshold $\|u_0\|_{L^2}^2=2\pi$ for $u_0\in H_+^1(\mathbb R)$ \cite{5,2}. Finite-time blow-up phenomena have been exhibited in the large-mass regime $\|u_0\|_{L^2}^2>2\pi$ \cite{10,56,59,67}.
\end{remark}
\subsection{Gérard-type explicit formula}
Following G\'erard's approach for the Benjamin--Ono equation \cite{1}, Killip, Laurens and Vi\c{s}an derived in \cite{2} an explicit formula for solutions of \eqref{3.1} in $L_+^2(\mathbb R)$.
\begin{theorem}[{\cite[Theorem~1.7]{2}}]
\label{theorem 1.3}
Let $u \in C\left(\mathbb{R}, L_{+}^2(\mathbb{R})\right)$ be the solution to \eqref{3.1} with $u(0)=u_0$. Then for any $t \in \mathbb{R}$,
\begin{equation}
\label{1.8}
\forall z \in \mathbb{C}_{+}, \quad  u(t, z)=\frac{1}{2 i \pi} I_{+}\left[\left(X^{*}+2 t L_{u_0}-z \mathrm{Id}\right)^{-1} u_0\right] .
\end{equation}
Here $X^*+2tL_{u_0}$ denotes the closed realization constructed in
\cite[Proposition~4.10]{2}.
\end{theorem}
\noindent G\'erard first obtained this type of formula for the Benjamin--Ono equation on the line \cite{1}; see \cite{27} for its extension. Analogues are now available for the cubic Szeg\H{o} equation on the line \cite{80}, CM-DNLS \cite{2,70}, and the half-wave maps equation \cite{30,30.1}. Further developments and applications may be found in \cite{80,20,2,27,13,29,25,30,90,0,89,30.1,0.01,20.1,66,0.1,67,4.44,14.1}.

For later use, we record an additional structural consequence of \eqref{1.8}, obtained by the argument of \cite[Lemma 3.1]{27}.
\begin{lemma}
\label{lemma:lax-resolvent-hardy-bound}
Let $u_0\in H_{+}^2(\mathbb{R})$. For
$v\in L^2_+(\mathbb R)$ and 
$t\in\mathbb R$, define, for \(z\in\mathbb C_+\),
\begin{equation}
\label{eq:lax-resolvent-hardy-function}
H_v(t,z)
:=
\frac{1}{2i\pi}
I_+\left[
\left(X^*+2tL_{u_0}-z\operatorname{Id}\right)^{-1}v
\right].
\end{equation}
Then there exists a unique function $h_v(t)\in L^2_+(\mathbb R)$
such that \(H_v(t,\cdot)\) is the Hardy extension of \(h_v(t)\), namely
\begin{equation}
\label{eq:lax-resolvent-hardy-representation}
H_v(t,z)
=
\frac1{2\pi}
\int_0^\infty e^{iz\xi}\widehat{h_v(t)}(\xi)\,d\xi,
\qquad z\in\mathbb C_+
\end{equation}
and
\begin{equation}
\label{eq:lax-resolvent-hardy-bound}
\|h_v(t)\|_{L^2(\mathbb R)}
=
\|v\|_{L^2(\mathbb R)}.
\end{equation}
\end{lemma}
\begin{proof}
For $u_0\in H_{+}^2(\mathbb{R})$, the proof of the explicit formula in \cite{1,2} gives
\[h_v (t) =U(t)\mathrm{e}^{-i tL_{u_0}^2} v,\]
where $U(t)$ is the unitary solution of the following linear ODE on the space of bounded operators on $L^2_+(\mathbb R)$:
\[ U'(t)=B_{u(t)}U(t),\quad U(0)=\mathrm{Id}. \]
Here $B_u:=-u T_{\partial_x \bar{u}}+\partial_x u T_{\bar{u}}+i\left(u T_{\bar{u}}\right)^2$ and $u(t)$ is the solution to \eqref{3.1} with initial datum $u_0$.
\end{proof}
\subsection{Main result}
We can now state the scattering theorem announced above. It applies to every initial datum in the critical space $L_+^2(\mathbb R)$ and also identifies the scattering states.
\begin{theorem}
\label{theorem 1.5}
Given $u_0\in L_+^{2}(\mathbb R)$, let \(u\in C(\mathbb R;L_+^2(\mathbb R))\) be the solution to the
defocusing (CM-DNLS) \eqref{3.1} with initial data \(u(0)=u_0\). Then there exist
\(v_0^-,v_0^+\in L_+^2(\mathbb R)\) such that
\begin{equation}
\label{3.13}
\lim_{t\to\pm\infty}
\left\|
u(t)-e^{it\partial_x^2}v_0^\pm
\right\|_{L^2(\mathbb R)}
=0.
\end{equation}
The negative-time scattering state is characterized by
\[
\widehat{v_0^-}(\xi)
=
\widetilde u_0(\xi)\mathbf 1_{\xi>0},
\qquad
\widetilde u_0:=\mathcal F_{d,u_0}u_0.
\]
For the positive-time scattering state, define
\[
u_0^\sharp(x):=\overline{u_0(-x)}
\]
and let
\[
\widetilde{u_0^\sharp}
:=
\mathcal F_{d,u_0^\sharp}u_0^\sharp .
\]
If \(v_{0,\sharp}^-\in L_+^2(\mathbb R)\) is defined by
\[
\widehat{v_{0,\sharp}^-}(\xi)
=
\widetilde{u_0^\sharp}(\xi)\mathbf 1_{\xi>0},
\]
then one defines
\[
v_0^+(x):=\overline{v_{0,\sharp}^-(-x)}.
\]
Here \(\mathcal F_{d,u_0}\) denotes the distorted Fourier transform associated with
\(L_{u_0}\) defined in Section \ref{sec 2.2}.
\end{theorem}
\begin{remark}
\label{remark 1.101}
Let
\[
\mathcal R f(x):=\overline{f(-x)}.
\]
If \(u\) is a solution to \eqref{3.1} with initial data \(u_0\), then
\[
u^\sharp(t,x):=\mathcal R u(-t,x)=\overline{u(-t,-x)}
\]
is also a solution to \eqref{3.1}, with initial data
\[
u_0^\sharp=\mathcal R u_0.
\]
Moreover, \(\mathcal R\) preserves \(L_{+}^{2}(\mathbb R)\).

Therefore, once the negative-time scattering statement has been proved for arbitrary
initial data in \(L_{+}^{2}(\mathbb R)\), applying it to \(u_0^\sharp\) gives
\[
u^\sharp(t)-e^{it\partial_x^2}v_{0,\sharp}^-
\to0
\qquad\text{in }L^2(\mathbb R)
\]
as \(t\to-\infty\). Replacing \(t\) by \(-t\) and applying \(\mathcal R\), we obtain
\[
u(t)-e^{it\partial_x^2}\mathcal R v_{0,\sharp}^-
\to0
\qquad\text{in }L^2(\mathbb R)
\]
as \(t\to+\infty\). Thus it is enough to prove the negative-time scattering statement
\[
\lim_{t\to-\infty}
\left\|
u(t)-e^{it\partial_x^2}v_0^-
\right\|_{L^2(\mathbb R)}
=0.
\]
\end{remark}
\begin{remark}
Kim and Kwon proved scattering for the focusing equation \eqref{1.5} in $H^{1,1}(\mathbb R)$ below the mass threshold \cite{23}. Their argument used the pseudo-conformal symmetry and, in particular, the assumption $xu_0\in L^2(\mathbb R)$. Theorem \ref{theorem 1.5} removes this weighted condition in the defocusing problem and treats every chiral datum in $L_+^2(\mathbb R)$. It also characterizes the negative-time scattering state through the distorted Fourier transform associated with $L_{u_0}$, while the positive-time state follows from the time-reversal symmetry above. Using the corresponding distorted Fourier theory for the focusing equation \cite[Section~4]{78}, the method can be adapted to establish scattering for focusing data in $L_+^2(\mathbb R)\cap\{\|u_0\|_{L^2}^2<2\pi\}$, extending the chiral scattering result of \cite{23}. Related results for the Benjamin--Ono equation may be found in \cite{29,0,0.01,0.1}, and rational-data results for half-wave maps in \cite{30}.
\end{remark}
\begin{remark}
The distinction between $L_+^2(\mathbb R)$ and $H^{1,1}(\mathbb R)$ is not merely technical. The ground state  solution in the focusing case $\frac{\sqrt{2}}{x+i}$ belongs to $L_+^2(\mathbb R)$ but not to $H^{1,1}(\mathbb R)$; see \cite[Lemma~4.1]{5}. Thus removing the weighted assumption genuinely enlarges the admissible class of data.
\end{remark}
\subsection*{Structure of the paper}

The paper is organized as follows. In Section \ref{Section 2}, we develop the spectral preliminaries for
the Lax operator \(L_u\). We construct the normalized generalized eigenfunctions, prove
their uniform bounds, and establish the distorted Fourier transform, including its
unitarity, distributional extension, and stability with respect to the potential.

In Section \ref{section 2.3}, we rewrite the explicit formula in distorted Fourier variables. The main
step is to identify the action of \(X^*\) under the distorted Fourier transform. This
leads to a spectral representation of the defect term
\[
u(t,z)-e^{it\partial_x^2}v_0^-(z),
\]
first for smooth data and then, by approximation, for data in
\(L_{+}^{2}(\mathbb R)\).

Section \ref{section 3} is devoted to the proof of Theorem \ref{theorem 1.5}. We first reduce the
scattering statement to a weak radiation limit. We then prove this limit by
combining the distorted Fourier representation of the defect term with oscillatory estimates and approximation arguments.

Finally, the Appendix \ref{appendix} collects the auxiliary estimates used in the proof. These include
oscillatory integral bounds, resolvent estimates and the
construction of the rough boundary functional.
\vskip 0.5cm
\thank{The author is currently a postdoctoral researcher at the Department of Mathematics and Computer Science, University of Basel. Part of this work was carried out during his PhD studies at the Laboratoire de Mathématiques d’Orsay (UMR 8628), Université Paris-Saclay. The author is deeply grateful to Patrick Gérard for his insightful comments throughout this work. The author also wishes to thank Rowan Killip, Thierry Laurens and Monica Vi\c{s}an for their helpful suggestions. Finally, the author gratefully acknowledges financial support by the Swiss National Science Foundation (SNSF) under Grant No.~204121.}\\

\textbf{Declaration on the use of AI}.
The overall conceptual framework and strategy of this article originated from ideas developed by the author in October 2025, without the assistance of any artificial intelligence tools. At a later stage, the author used ChatGPT to help explore ideas for completing technically intricate parts of the proofs and to pursue the strongest result attainable within this framework, namely scattering for initial data in \(L^2_+(\mathbb R)\). ChatGPT was also used to refine the prose and to improve the presentation of certain arguments and proofs. All mathematical arguments and proofs in the final manuscript were carefully and rigorously checked multiple times by the author. The author takes full responsibility for the accuracy, integrity, and conclusions of the manuscript.
\section{Preliminaries}
\label{Section 2}
In this section we develop the spectral tools associated with the Lax operator \(L_u\).
We first construct the bounded generalized eigenfunctions by a Fredholm argument. We then use these eigenfunctions to define the distorted Fourier
transform, prove its \(L^2\)-unitarity and surjectivity, and show the mapping and
continuity properties that will be used in the scattering argument. The corresponding
distorted Fourier theory for the focusing (CM-DNLS) equation was developed in
\cite[Section 4]{78}.

\subsection{Generalized eigenfunctions}\label{sec 2.1}
First we show that if $u \in L_{+}^2(\mathbb{R})$, the Lax operator $L_u$ has no eigenvalue.
\begin{proposition}
\label{prop:no_eigenvalues_L2_defocusing}
If $u \in L_{+}^2(\mathbb{R})$, then $L_u$ has no eigenvalues.
\end{proposition}
\begin{proof}
Let $\psi \in H_{+}^{1/2}(\mathbb{R})$ be an eigenvector of $L_u$. From the proof of \cite[Proposition 5.1]{5}, we can infer
\begin{align*}
\|\psi\|_{L^2}^2+\frac{1}{2 \pi}|\langle\psi, u\rangle|^2 = 0.
\end{align*}
Thus $\psi = 0$. 
\end{proof}

\begin{remark}
Using Weyl's theorem, we can easily conclude that $\sigma_{ess} (L_u) = \sigma_{ess}(D) = [0,\infty)$. Since $L_u$ has no eigenvalues, we can conclude that $\sigma(L_u)=  \sigma_{ac}(L_u) \cup \sigma_{sc}(L_u)= [0,\infty)$.
\end{remark}
\begin{remark}
For the Lax operator of the focusing (CM-DNLS) \eqref{1.5}, if $\|u\|_{L^2}^2<2\pi$, we can also infer that this Lax operator has no eigenvalues.
\end{remark}
\noindent We now give a fundamental result concerning the generalized eigenfunctions of $L_u$, which constitute the foundation of the distorted Fourier transform.
\begin{proposition}[Generalized eigenfunctions]
\label{prop:generalized_eigenfunctions}
Let $u\in L_{+}^{2}(\mathbb R)$. For every $\lambda >0$, the space
\[
\mathscr{M}_u(\lambda)
=\bigl\{\,m\in L^\infty_+(\mathbb{R}):\;(L_u-\lambda)m=0\ \text{in the distributional sense}\bigr\}
\]
is one–dimensional.  Every element $m\in\mathscr{M}_u(\lambda)$ satisfies
\begin{align*}
\mathrm{e}^{-i \lambda x} m(x).
\end{align*}
has a limit $\ell_m( \pm \infty) \text { as } x \rightarrow \pm \infty$. Furthermore, $\ell_m(-\infty)$ characterizes $m$ and
\begin{align*}
\left|\ell_m(+\infty)\right|=\left|\ell_m(-\infty)\right|.
\end{align*}
\end{proposition}
\begin{proof}
Notice that, if $m \in L^{\infty}(\mathbb{R})$,
\begin{align*}
\Pi(\bar{u} m) \in L_{+}^2(\mathbb{R}),
\end{align*}
so that $u T_{\bar{u}} m \in L^1(\mathbb{R})$. Equation $(L_u - \lambda) m = 0$ is equivalent to
\begin{align*}
D\left(\mathrm{e}^{-i \lambda x} m(x)\right)=-\mathrm{e}^{-i \lambda x} u(x) \Pi(\bar{u} m)(x),
\end{align*}
which implies that
\begin{align*}
D\left(\mathrm{e}^{-i \lambda x} m(x)\right) \in L^1(\mathbb{R}),
\end{align*}
hence $\mathrm{e}^{-i \lambda x} m(x)$ has a limit $\ell_m( \pm \infty)$ as $x \rightarrow \pm \infty$. Furthermore,
\begin{align*}
i\left(\left|\ell_m(+\infty)\right|^2-\left|\ell_m(-\infty)\right|^2\right) & =\lim _{R \rightarrow+\infty} \int_{-R}^R[m(x) \overline{D m(x)}-D m(x) \overline{m(x)}] d x \\
& =\int_{-\infty}^{\infty}\left[-m(x) \overline{u T_{\bar{u}} m(x)}+u T_{\bar{u}} m(x) \overline{m(x)}\right] d x \\
& =-\left\|T_{\bar{u}} m\right\|^2+\left\|T_{\bar{u}} m\right\|^2 \\
& =0 .
\end{align*}
This proves $\left|\ell_m(+\infty)\right|=\left|\ell_m(-\infty)\right|$. It remains to prove that the space of such functions $m$ is one-dimensional.\\\\
Given $\ell \in \mathbb{C}$, the problem $\left(L_u-\lambda\right) m=0, \ell_m(-\infty)=\ell$ is equivalent to 
\begin{equation}
\label{3.2}
a=a_\lambda(x) = a(x,\lambda):= \mathrm{e}^{-i \lambda x} m,\quad a-K_\lambda(a)=\ell,
\end{equation}
where $K_\lambda: L^{\infty}(\mathbb{R}) \rightarrow L^{\infty}(\mathbb{R})$ is defined as 
\begin{align*}
K_\lambda(a)(x):=-i\int_{-\infty}^x \mathrm{e}^{-i \lambda y} u(y) \Pi\left(\bar{u} \mathrm{e}^{i \lambda \cdot} a \right)(y) d y.
\end{align*}
We claim that $1-K_\lambda: L^{\infty} \rightarrow L^{\infty}$ is bijective. We shall use the Fredholm alternative by showing that $K_\lambda$ is compact and $1-K_{\lambda}$ is injective.
\begin{lemma}
\label{lemma 2.3}
The operator $K_\lambda: L^{\infty}(\mathbb{R}) \rightarrow L^{\infty}(\mathbb{R})$ is compact.
\end{lemma}
\begin{proof}
If $\|a\|_{L^{\infty}} \leq 1$, $\mathrm{e}^{-i \lambda y} u(y) \Pi\left(\bar{u} \mathrm{e}^{i \lambda y} a\right)$ belongs to a bounded subset of $L^1$. Furthermore,
\begin{equation}
\label{2.2}
\begin{aligned}
& \left|K_\lambda(a)(x)-K_\lambda(a)\left(x^{\prime}\right)\right| \leq\|u\|_{L^2}\left(\int_x^{x^{\prime}}|u(y)|^2 d y\right)^{\frac{1}{2}}, \quad x \leq x^{\prime} \\
& \left|K_\lambda(a)(x)\right| \leq\|u\|_{L^2}\left(\int_{-\infty}^x|u(y)|^2 d y\right)^{\frac{1}{2}} \\
& \left|K_\lambda(a)(x)-K_\lambda(a)(+\infty)\right| \leq\|u\|_{L^2}\left(\int_x^{+\infty}|u(y)|^2 d y\right)^{\frac{1}{2}}
\end{aligned}
\end{equation}
For all these reasons, the set
\begin{align*}
\left\{K_\lambda(a),\|a\|_{L^{\infty}} \leq 1\right\}
\end{align*}
is relatively compact in $L^{\infty}$.
\end{proof}
\begin{lemma}
\label{lemma 2.6}
Let
\[
u\in L^2_+(\mathbb R).
\]
For every \(\lambda>0\), the operator
\[
I-K_\lambda:L^\infty(\mathbb R)\to L^\infty(\mathbb R)
\]
is injective.
\end{lemma}

\begin{proof}
Let \(f\in L^\infty(\mathbb R)\) satisfy
\[
f=K_\lambda f.
\]
Set
\[
m(x):=e^{i\lambda x}f(x),
\qquad
h:=T_{\bar u}m=\Pi(\bar u m),
\qquad
q:=uh .
\]
Then
\[
h\in L^2_+(\mathbb R),
\qquad
q\in L^1(\mathbb R),
\]
and \(f\) is locally absolutely continuous with
\[
f'(x)=-i e^{-i\lambda x}q(x).
\]
Equivalently,
\[
(D-\lambda)m=-q=-uT_{\bar u}m
\]
in the sense of distributions. Moreover, the formula gives
\[
f(x)
=
-i\int_{-\infty}^x e^{-i\lambda y}q(y)\,dy,
\]
and therefore
\[
f(-\infty)=0.
\]
From the previous derivation, we know
\[
f(+\infty)= f(-\infty) =0.
\]
Thus
\[
m(x)=e^{i\lambda x}f(x)\to0
\qquad\text{as }|x|\to\infty .
\]

We next prove that \(m\in L^2(\mathbb R)\). The argument below is adapted from the tail bootstrap of \cite{78}, where bounded solutions with two-sided decay are shown to belong to \(L^2\) under only an \(L^2\) assumption on the potential. For \(R>0\), define
\[
H(R):=\int_{|y|>R}|u(y)h(y)|\,dy,
\qquad
Q(R):=\|u\|_{L^2(|x|>R)},
\]
\[
F(R):=\|m\|_{L^2(|x|<R)},
\qquad
A(R):=\int_{|x|>R}\frac{|u(x)|}{|x|}\,dx .
\]
Since \(q=uh\in L^1\) and \(f(\pm\infty)=0\), we have, for \(|x|>R\),
\begin{equation}
\label{eq:FR_tail_bound_m}
|m(x)|=|f(x)|
\le H(|x|).
\end{equation}
We now estimate \(H(R)\). Decompose
\[
h=h_{0,R}+h_{1,R},
\]
where
\[
h_{0,R}:=\Pi\bigl(\mathbf 1_{\{|y|>R/2\}}\bar u m\bigr),
\qquad
h_{1,R}:=\Pi\bigl(\mathbf 1_{\{|y|<R/2\}}\bar u m\bigr).
\]
For the first term, by the \(L^2\)-boundedness of \(\Pi\) and
\eqref{eq:FR_tail_bound_m},
\[
\begin{aligned}
\int_{|x|>R}|u(x)h_{0,R}(x)|\,dx
&\le
Q(R)\|h_{0,R}\|_2  \\
&\le
Q(R)\|\mathbf 1_{\{|y|>R/2\}}\bar u m\|_2  \\
&\le
Q(R)Q(R/2)H(R/2)  \\
&\le
Q(R/2)^2H(R/2).
\end{aligned}
\]
Now we denote the Hilbert transform on the line by $H_{\mathbb R}$. For the second term, the supports of
\(\mathbf 1_{\{|y|<R/2\}}\bar u m\) and \(\mathbf 1_{\{|x|>R\}}\) are separated.
Hence the identity $\Pi = \frac{1}{2} (I + iH_{\mathbb R})$ gives
\[
|h_{1,R}(x)|
\le
\frac{C}{|x|}
\int_{|y|<R/2}|\bar u(y)m(y)|\,dy,
\qquad |x|>R .
\]
Therefore
\[
\begin{aligned}
\int_{|x|>R}|u(x)h_{1,R}(x)|\,dx
&\le
C A(R)
\int_{|y|<R/2}|\bar u(y)m(y)|\,dy  \le
C\|u\|_2 F(R/2)A(R).
\end{aligned}
\]
Combining the two estimates, we get
\begin{equation}
\label{eq:FR_H_ineq}
H(R)
\le
Q(R/2)^2H(R/2)
+
C\|u\|_2F(R/2)A(R),
\qquad R>0.
\end{equation}

We claim that \(H\in L^2(R_0,\infty)\) for some large \(R_0\). First observe that
\(Q(R)\to0\) as \(R\to\infty\). Also, by the dual Hardy's inequality,
\[
A\in L^2(R_0,\infty)
\]
for every \(R_0>0\), and moreover
\[
\|A\|_{L^2(R_0,\infty)}\to0
\qquad\text{as }R_0\to\infty.
\]
Choose \(R_0\) so large that
\[
Q(R_0/2)^2\le \delta
\]
with \(\delta>0\) sufficiently small.

For \(S>R_0\), set
\[
Y(S):=\int_{R_0}^S H(R)^2\,dR .
\]
Squaring \eqref{eq:FR_H_ineq}, integrating over \(R\in(R_0,S)\), and absorbing
the finite contribution from \(R\in(R_0/2,R_0)\), we obtain
\[
Y(S)
\le
C_{R_0}
+
C\delta^2Y(S)
+
C\int_{R_0}^S F(R/2)^2A(R)^2\,dR .
\]
By \eqref{eq:FR_tail_bound_m},
\[
F(R/2)^2
\le
C_{R_0}
+
2\int_{R_0}^{R}H(s)^2\,ds
\le
C_{R_0}+2Y(R).
\]
Hence
\[
Y(S)
\le
C_{R_0}
+
C\delta^2Y(S)
+
C\int_{R_0}^S Y(R)A(R)^2\,dR .
\]
Taking \(\delta\) sufficiently small and absorbing the second term on the
right-hand side, we get
\[
Y(S)
\le
C_{R_0}
+
C\int_{R_0}^S Y(R)A(R)^2\,dR .
\]
Since \(A^2\in L^1(R_0,\infty)\), Gronwall's inequality implies
\[
\sup_{S>R_0}Y(S)<+\infty .
\]
Thus
\[
H\in L^2(R_0,\infty).
\]
Using \eqref{eq:FR_tail_bound_m} again, we conclude that
\[
m\in L^2(\mathbb R).
\]

We now prove that \(m\in L^2_+(\mathbb R)\). Since
\[
h=T_{\bar u}m\in L^2_+(\mathbb R)
\]
and \(u\in L^2_+(\mathbb R)\), we have
\[
q=uh\in L^1(\mathbb R),
\qquad
\operatorname{supp}\widehat q\subset[0,\infty)
\]
in the sense of distributions. The equation
\[
(D-\lambda)m=-q
\]
gives, in Fourier variables,
\[
(\xi-\lambda)\widehat m(\xi)=-\widehat q(\xi).
\]
Let \(\psi\in C_c^\infty((-\infty,0))\). Since \(\lambda\ge0\), the function
\[
\phi(\xi):=\frac{\psi(\xi)}{\xi-\lambda}
\]
also belongs to \(C_c^\infty((-\infty,0))\). Therefore
\[
\langle \widehat m,\psi\rangle
=
\langle (\xi-\lambda)\widehat m,\phi\rangle
=
-\langle \widehat q,\phi\rangle
=
0.
\]
Thus \(\widehat m\) vanishes on \((-\infty,0)\). Since \(m\in L^2(\mathbb R)\),
this proves
\[
m\in L^2_+(\mathbb R).
\]

It remains to justify that \(m\) is an eigenfunction of the form-defined operator
\(L_u\). We first show that
\[
m\in H^{1/2}_+(\mathbb R).
\]
Let
\[
P_R=\chi(D/R),
\]
where \(\chi\in C_c^\infty(\mathbb R)\), \(0\le\chi\le1\), \(\chi=1\) near
\(0\), and \(\chi(\xi/R)\uparrow1\) for every \(\xi\ge0\). Testing
\[
D m+u h=\lambda m
\]
against \(P_Rm\), and taking real parts, gives
\[
\int_0^\infty
\xi\,\chi(\xi/R)|\widehat m(\xi)|^2\,\frac{d\xi}{2\pi}
+
\operatorname{Re}\int_{\mathbb R}u h\,\overline{P_Rm}\,dx
=
\lambda\,\operatorname{Re}\langle m,P_Rm\rangle .
\]
The multipliers \(P_R\) are uniformly bounded on \(L^\infty\), and
\(P_Rm\to m\) in \(L^2_{\mathrm{loc}}\). By truncating \(u\) and using the
uniform \(L^\infty\)-boundedness of \(P_Rm\), we get
\[
\bar u\,P_Rm\to \bar u\,m
\qquad\text{in }L^2.
\]
Hence
\[
T_{\bar u}(P_Rm)\to T_{\bar u}m=h
\qquad\text{in }L^2.
\]
Therefore
\[
\int_{\mathbb R}u h\,\overline{P_Rm}\,dx
=
\langle h,T_{\bar u}(P_Rm)\rangle
\longrightarrow
\|h\|_2^2,
\]
and
\[
\langle m,P_Rm\rangle\to\|m\|_2^2.
\]
Letting \(R\to+\infty\) and using monotone convergence, we obtain
\[
\int_0^\infty
\xi|\widehat m(\xi)|^2\,\frac{d\xi}{2\pi}<+\infty.
\]
Thus
\[
m\in H^{1/2}_+(\mathbb R).
\]

Finally, the distributional identity
\[
D m+uT_{\bar u}m=\lambda m
\]
implies, first for \(\varphi\in\mathcal S_+(\mathbb R)\) and then by density for
all \(\varphi\in H^{1/2}_+(\mathbb R)\), that
\[
\mathcal Q_u(m,\varphi)
=
\lambda\langle m,\varphi\rangle .
\]
Hence
\[
m\in\operatorname{Dom}(L_u),
\qquad
L_um=\lambda m .
\]
By Proposition \(\ref{prop:no_eigenvalues_L2_defocusing}\), the operator \(L_u\)
has no eigenvalues. Therefore
\[
m=0.
\]
Consequently \(f=0\), and the injectivity of
\[
I-K_\lambda:L^\infty(\mathbb R)\to L^\infty(\mathbb R)
\]
follows.
\end{proof}
\noindent Now we can conclude the proof of Proposition \ref{prop:generalized_eigenfunctions}. By the Fredholm alternative, the operator $\operatorname{Id}-K_\lambda: L^{\infty}(\mathbb{R}) \rightarrow L^{\infty}(\mathbb{R})$ is bijective. Coming back to \eqref{3.2}, the linear form 
\begin{align*}
m \in \mathscr{M}_u(\lambda) \mapsto \ell_m(-\infty) \in \mathbb{C}
\end{align*}
is bijective. Therefore $\mathscr{M}_u(\lambda)$ is one-dimensional.  Furthermore, one can verify \(m(\cdot,\lambda)\in L^\infty_+(\mathbb R)\).
\end{proof}
The preceding proposition treats strictly positive energies. We shall also need the
limiting equation at \(\lambda=0\), in order to control the behavior of \(m_-(x,\lambda)\) as \(\lambda\to0_+\).
\begin{remark}
\label{remark 2.7}
We shall also use the limiting case \(\lambda=0\) of \eqref{3.2}. Define
\[
K_0a(x)
:=
-i\int_{-\infty}^x
u(y)\Pi(\overline u\,a)(y)\,dy .
\]
Then the same compactness argument as in Lemma \ref{lemma 2.3} shows that
\[
K_0:L^\infty(\mathbb R)\to L^\infty(\mathbb R)
\]
is compact. Moreover, the injectivity argument in Lemma \ref{lemma 2.6} also applies at
\(\lambda=0\). Hence
\[
\ker(I-K_0)=\{0\}.
\]
By the Fredholm alternative,
\[
I-K_0:L^\infty(\mathbb R)\to L^\infty(\mathbb R)
\]
is invertible.

Consequently, for every \(\ell\in\mathbb C\), the equation
\[
a-K_0a=\ell
\]
has a unique solution in \(L^\infty(\mathbb R)\). In particular, when \(\ell=1\), the
solution is
\[
a=1.
\]
Indeed, since \(u\in L^2_+(\mathbb R)\), we have
\[
\Pi(\overline u)=0
\]
in \(L^2(\mathbb R)\), and therefore
\[
K_0 1=0.
\]
Thus \(a=1\), and hence \(m=a=1\), is the unique solution of \eqref{3.2} with
\(\lambda=0\) and \(\ell=1\).
\end{remark}
\noindent
The next estimate is the main uniform bound for the generalized eigenfunctions. It is
essential that the bound be uniform both near \(\lambda=0\) and as
\(\lambda\to+\infty\), since later arguments use compactness and approximation in the
spectral variable.
\begin{proposition}[Uniform \(L^\infty_x\) bound in \(\lambda\)]
\label{prop:uniform_m}
Let
\[
u\in L^2_+(\mathbb R).
\]
For \(\lambda\ge0\), define
\[
K_\lambda a(x)
:=
-i\int_{-\infty}^x
e^{-i\lambda y}u(y)
\Pi\bigl(\overline u\,e^{i\lambda\cdot}a\bigr)(y)\,dy ,
\qquad a\in L^\infty(\mathbb R).
\]
Then, there exists a constant \(C=C(u)>0\) such that
\begin{equation}
\label{2.3}
\sup_{\lambda\ge0}
\|(I-K_\lambda)^{-1}\|_{\mathcal L(L^\infty,L^\infty)}
\le C .
\end{equation}
Consequently, if
\[
a_\lambda:=a(\cdot,\lambda):=(I-K_\lambda)^{-1}1,
\]
then
\[
\sup_{\lambda\ge0}\|a(\cdot,\lambda)\|_{L^\infty(\mathbb R)}\le C.
\]
For \(\lambda>0\), set
\[
m_-(x,\lambda):=e^{i\lambda x}a(x,\lambda).
\]
Then
\[
\sup_{\lambda>0}\|m_-(\cdot,\lambda)\|_{L^\infty(\mathbb R)}\le C.
\]
\end{proposition}

\begin{proof}
We split the proof into the low-frequency regime and the high-frequency regime.

\medskip
\noindent\textbf{Step 1: elementary bounds, continuity, and compactness.}
For \(a\in L^\infty(\mathbb R)\), by Cauchy-Schwarz inequality and the \(L^2\)-boundedness of
\(\Pi\),
\begin{equation}
\label{2.8.1}
\begin{aligned}
\|K_\lambda a\|_{L^\infty}
&\le
\int_{\mathbb R}|u(y)|
\left|
\Pi\bigl(\overline u\,e^{i\lambda\cdot}a\bigr)(y)
\right|\,dy  \\
&\le
\|u\|_{L^2}
\left\|
\Pi\bigl(\overline u\,e^{i\lambda\cdot}a\bigr)
\right\|_{L^2}
\le
\|u\|_{L^2}^2\|a\|_{L^\infty}.
\end{aligned}
\end{equation}
Hence \(K_\lambda\) is bounded on \(L^\infty\), uniformly in \(\lambda\ge0\).

We next prove that
\[
\lambda\mapsto K_\lambda
\]
is continuous in the operator norm of
\(\mathcal L(L^\infty,L^\infty)\). Let \(\lambda,\mu\ge0\) and
\(\|a\|_{L^\infty}\le1\). Write
\[
(K_\lambda-K_\mu)a=A_1+A_2,
\]
where
\[
A_1(x)
=
-i\int_{-\infty}^x
\bigl(e^{-i\lambda y}-e^{-i\mu y}\bigr)u(y)
\Pi\bigl(\overline u\,e^{i\lambda\cdot}a\bigr)(y)\,dy,
\]
and
\[
A_2(x)
=
-i\int_{-\infty}^x
e^{-i\mu y}u(y)
\Pi\bigl(\overline u\,(e^{i\lambda\cdot}-e^{i\mu\cdot})a\bigr)(y)\,dy.
\]
Again by Cauchy-Schwarz inequality and the \(L^2\)-boundedness of \(\Pi\),
\[
\|A_1\|_{L^\infty}
\le
\|(e^{-i\lambda\cdot}-e^{-i\mu\cdot})u\|_{L^2}\|u\|_{L^2},
\]
and
\[
\|A_2\|_{L^\infty}
\le
\|u\|_{L^2}\|(e^{i\lambda\cdot}-e^{i\mu\cdot})u\|_{L^2}.
\]
Since
\[
e^{\pm i\lambda x}u(x)\to e^{\pm i\mu x}u(x)
\quad\text{in }L^2(\mathbb R)
\]
as \(\lambda\to\mu\), we obtain
\begin{equation}
\label{2.8.2}
\|K_\lambda-K_\mu\|_{\mathcal L(L^\infty,L^\infty)}
\to0
\qquad\text{as }\lambda\to\mu.
\end{equation}
We now prove the collective compactness needed below. Let
\[
B:=\{a\in L^\infty(\mathbb R):\|a\|_{L^\infty}\le1\}.
\]
We claim that
\[
\mathcal C:=\bigcup_{\lambda\ge0}K_\lambda(B)
\]
is relatively compact in \(L^\infty(\mathbb R)\). In fact, from \eqref{2.2} we know that the functions in \(\mathcal C\) can be extended
continuously to the two-point compactification \([-\infty,+\infty]\), with a uniform
modulus of continuity and a uniformly bounded set of values at \(+\infty\). By the
Arzelà--Ascoli theorem on \([-\infty,+\infty]\), \(\mathcal C\) is relatively compact in
\(L^\infty(\mathbb R)\).

\medskip
\noindent\textbf{Step 2: uniform invertibility on bounded \(\lambda\)-intervals.}
We claim that for every \(\Lambda>0\),
\begin{equation}
\label{2.8.7}
\sup_{0\le\lambda\le\Lambda}
\|(I-K_\lambda)^{-1}\|_{\mathcal L(L^\infty,L^\infty)}
<\infty.
\end{equation}
For each fixed \(\lambda\in[0,\Lambda]\), the operator \(K_\lambda\) is compact on
\(L^\infty\). By Lemma \ref{lemma 2.6} for \(\lambda>0\), and by Remark
\ref{remark 2.7} for \(\lambda=0\), we have
\[
\ker(I-K_\lambda)=\{0\}.
\]
Hence \(I-K_\lambda\) is invertible for each \(\lambda\in[0,\Lambda]\).

It remains to prove that the inverse norms are uniformly bounded. Suppose not. Then
there exist \(\lambda_n\in[0,\Lambda]\) and \(b_n\in L^\infty(\mathbb R)\) such that
\begin{equation}
\label{2.8.8}
\|b_n\|_{L^\infty}=1,
\qquad
\|(I-K_{\lambda_n})b_n\|_{L^\infty}\to0.
\end{equation}
By relative compactness of $\mathcal C$, the sequence \(\{K_{\lambda_n}b_n\}\) has a convergent
subsequence in \(L^\infty\). Since
\[
b_n=K_{\lambda_n}b_n+(I-K_{\lambda_n})b_n,
\]
the sequence \(\{b_n\}\) also has a convergent subsequence. Passing to this subsequence,
we may assume
\[
b_n\to b\quad\text{in }L^\infty,
\qquad
\lambda_n\to\lambda_*\in[0,\Lambda].
\]
Using the operator-norm continuity \eqref{2.8.2}, we obtain
\[
(I-K_{\lambda_*})b
=
\lim_{n\to\infty}(I-K_{\lambda_n})b_n
=0.
\]
Thus \(b=0\), by injectivity of \(I-K_{\lambda_*}\). This contradicts
\[
\|b\|_{L^\infty}
=
\lim_{n\to\infty}\|b_n\|_{L^\infty}
=1.
\]
Therefore \eqref{2.8.7} holds.

In particular,
\begin{equation}
\label{2.8.9}
\sup_{0\le\lambda\le\Lambda}
\|a(\cdot,\lambda)\|_{L^\infty}
\le
\sup_{0\le\lambda\le\Lambda}
\|(I-K_\lambda)^{-1}\|_{\mathcal L(L^\infty,L^\infty)}
<\infty.
\end{equation}

\medskip
\noindent\textbf{Step 3: the high-frequency limit operator.}
Define the Volterra operator
\begin{equation}
\label{2.8.10}
(K_\infty a)(x)
:=
-i\int_{-\infty}^x |u(y)|^2a(y)\,dy .
\end{equation}
Then \(K_\infty\) is bounded on \(L^\infty\). Moreover,
\begin{equation}
\label{2.8.11}
\|K_\infty^n\|_{\mathcal L(L^\infty,L^\infty)}
\le
\frac{\|u\|_{L^2}^{2n}}{n!},
\qquad n\ge0.
\end{equation}
Indeed, this follows by iterating the Volterra integral and estimating the integral over
the simplex
\[
-\infty<y_n<\cdots<y_1<x.
\]
Consequently the Neumann series
\begin{equation}
\label{2.8.12}
(I-K_\infty)^{-1}
=
\sum_{n=0}^\infty K_\infty^n
\end{equation}
converges in \(\mathcal L(L^\infty,L^\infty)\). In particular, \(I-K_\infty\) is
invertible.

We now prove that
\[
K_\lambda\to K_\infty
\]
uniformly on compact subsets of \(L^\infty(\mathbb R)\) as \(\lambda\to+\infty\). Let
\(C\subset L^\infty(\mathbb R)\) be compact, and set
\[
\mathcal H:=\{\overline u\,a:\ a\in C\}\subset L^2(\mathbb R).
\]
The map
\[
a\mapsto \overline u\,a
\]
is continuous from \(L^\infty\) to \(L^2\). Hence \(\mathcal H\) is compact in
\(L^2(\mathbb R)\). For \(h\in L^2\), we have
\[
\widehat{\Pi(h e^{i\lambda\cdot})}(\xi)
=
\mathbf 1_{\xi\ge0}\widehat h(\xi-\lambda).
\]
Therefore
\begin{equation}
\label{2.8.13}
\begin{aligned}
\|\Pi(h e^{i\lambda\cdot})-h e^{i\lambda\cdot}\|_{L^2}^2
&=
\int_{\xi<0}|\widehat h(\xi-\lambda)|^2\,d\xi  \\
&=
\int_{\eta<-\lambda}|\widehat h(\eta)|^2\,d\eta .
\end{aligned}
\end{equation}
Since \(\mathcal H\) is compact in \(L^2\), its Fourier tails are uniformly small:
\begin{equation}
\label{2.8.14}
\sup_{h\in\mathcal H}
\int_{\eta<-\lambda}|\widehat h(\eta)|^2\,d\eta
\longrightarrow0
\qquad\text{as }\lambda\to+\infty.
\end{equation}
Hence
\begin{equation}
\label{2.8.15}
\sup_{a\in C}
\left\|
\Pi(\overline u\,e^{i\lambda\cdot}a)
-
\overline u\,e^{i\lambda\cdot}a
\right\|_{L^2}
\longrightarrow0.
\end{equation}
Using Cauchy-Schwarz inequality,
\[
\begin{aligned}
\|(K_\lambda-K_\infty)a\|_{L^\infty}
&\le
\int_{\mathbb R}
|u(y)|
\left|
\Pi(\overline u\,e^{i\lambda\cdot}a)(y)
-
\overline u(y)e^{i\lambda y}a(y)
\right|\,dy  \\
&\le
\|u\|_{L^2}
\left\|
\Pi(\overline u\,e^{i\lambda\cdot}a)
-
\overline u\,e^{i\lambda\cdot}a
\right\|_{L^2}.
\end{aligned}
\]
Taking the supremum over \(a\in C\), and using \eqref{2.8.15}, we get
\begin{equation}
\label{2.8.16}
\sup_{a\in C}
\|(K_\lambda-K_\infty)a\|_{L^\infty}
\longrightarrow0
\qquad\text{as }\lambda\to+\infty.
\end{equation}

\medskip
\noindent\textbf{Step 4: uniform invertibility for large \(\lambda\).}
We claim that there exist \(\Lambda_0>0\) and \(C_\infty<\infty\) such that
\begin{equation}
\label{2.8.17}
\sup_{\lambda\ge\Lambda_0}
\|(I-K_\lambda)^{-1}\|_{\mathcal L(L^\infty,L^\infty)}
\le C_\infty .
\end{equation}
Suppose not. Then there exist \(\lambda_n\to+\infty\) and \(b_n\in L^\infty\) such that
\begin{equation}
\label{2.8.18}
\|b_n\|_{L^\infty}=1,
\qquad
\|(I-K_{\lambda_n})b_n\|_{L^\infty}\to0.
\end{equation}
By collective compactness, the sequence \(\{K_{\lambda_n}b_n\}\) is relatively compact in
\(L^\infty\). Since
\[
b_n=K_{\lambda_n}b_n+(I-K_{\lambda_n})b_n,
\]
the sequence \(\{b_n\}\) is also relatively compact. Passing to a subsequence, we may
assume
\begin{equation}
\label{2.8.19}
b_n\to b
\quad\text{in }L^\infty.
\end{equation}
Set
\[
C:=\{b\}\cup\{b_n:n\ge1\}.
\]
Then \(C\) is compact in \(L^\infty\). By \eqref{2.8.16}, applied to this compact set,
\begin{equation}
\label{2.8.20}
\|(K_{\lambda_n}-K_\infty)b_n\|_{L^\infty}\to0.
\end{equation}
Therefore
\[
\begin{aligned}
(I-K_\infty)b_n
&=
(I-K_{\lambda_n})b_n
+
(K_{\lambda_n}-K_\infty)b_n
\longrightarrow0
\quad\text{in }L^\infty.
\end{aligned}
\]
Passing to the limit in \(L^\infty\), we get
\[
(I-K_\infty)b=0.
\]
Since \(I-K_\infty\) is invertible, \(b=0\). This contradicts
\[
\|b\|_{L^\infty}
=
\lim_{n\to\infty}\|b_n\|_{L^\infty}
=1.
\]
Thus \eqref{2.8.17} holds.

\medskip
\noindent\textbf{Step 5: conclusion.}
Combining the low-frequency bound \eqref{2.8.7} with the high-frequency bound
\eqref{2.8.17}, we obtain
\[
\sup_{\lambda\ge0}
\|(I-K_\lambda)^{-1}\|_{\mathcal L(L^\infty,L^\infty)}
<\infty.
\]
Since
\[
a_\lambda=(I-K_\lambda)^{-1}1,
\]
we conclude that
\[
\sup_{\lambda\ge0}
\|a(\cdot,\lambda)\|_{L^\infty}
<\infty.
\]
Finally, for \(\lambda>0\),
\[
\|m_-(\cdot,\lambda)\|_{L^\infty}
=
\|e^{i\lambda\cdot}a(\cdot,\lambda)\|_{L^\infty}
=
\|a(\cdot,\lambda)\|_{L^\infty}.
\]
Therefore
\[
\sup_{\lambda>0}
\|m_-(\cdot,\lambda)\|_{L^\infty}
<\infty.
\]
This completes the proof.
\end{proof}
This completes the construction and uniform control of the normalized generalized
eigenfunctions. We now use them to build the distorted Fourier transform associated with
\(L_u\).
\subsection{The distorted Fourier transform}\label{sec 2.2}
From Proposition \ref{prop:generalized_eigenfunctions}, for $u \in L_{+}^{2}(\mathbb R)$, there exists a
unique $m_{-} \in \mathscr{M}_u(\lambda)$ such that $\ell_{m_{-}}(-\infty)=1$. Given $f \in L_{+}^1(\mathbb{R})$, we can define
\begin{align*}
\mathcal{F}_{d,u} f (\lambda) := \widetilde{f}(\lambda):=\int_{\mathbb{R}} f(x) \overline{m_{-}(x, \lambda)} d x, \quad \lambda > 0.
\end{align*}
If moreover $f \in L^2(\mathbb{R})$, then we can show that $\widetilde{f} \in L^2((0,+\infty))$. Denote by $\mathcal{H}_c\left(L_u\right)$ the continuous space of $L_u$, namely the space of those $f \in L_{+}^2(\mathbb{R})$ having a spectral measure with no pure point part. Denote by $P_c(L_u)$ the orthogonal projector from $L_{+}^2$ to $\mathcal{H}_c\left(L_u\right)$.
\begin{proposition}[Plancherel's identity]
\label{prop 2.6}
For $u \in L_{+}^{2}(\mathbb R)$, the operator $f \mapsto \mathcal{F}_{d,u} f:= \widetilde{f}$ extends to $L_{+}^2(\mathbb{R})$ with the identity
\begin{equation}
\label{3.3}
\int_{\mathbb{R}}\left|f(x)\right|^2 d x= \int_{\mathbb{R}}\left|P_c(L_u) f(x)\right|^2 d x=\frac{1}{2 \pi} \int_0^{\infty}|\widetilde{f}(\lambda)|^2 d \lambda.
\end{equation}
\end{proposition}
\begin{proof}
We first prove the identity for
\[
f\in L^1(\mathbb R)\cap L^2_+(\mathbb R).
\]
The general case will then follow by density. Throughout the proof the scalar
product is linear in the first variable.

For \(\lambda>0\) and \(\varepsilon>0\), set
\[
g_+^\varepsilon(\cdot,\lambda)
:=
(L_u-\lambda-i\varepsilon)^{-1}f,
\qquad
g_-^\varepsilon(\cdot,\lambda)
:=
(L_u-\lambda+i\varepsilon)^{-1}f .
\]
We shall need the following local limiting absorption statement.

Let \(I\Subset(0,\infty)\). For \(h\in L^1(\mathbb R)\), define
\[
(R_{0,\varepsilon}^+(\lambda)h)(x)
:=
i\int_{-\infty}^x
e^{i(\lambda+i\varepsilon)(x-y)}h(y)\,dy,
\]
and
\[
(R_{0,\varepsilon}^-(\lambda)h)(x)
:=
-i\int_x^{+\infty}
e^{i(\lambda-i\varepsilon)(x-y)}h(y)\,dy.
\]
For \(\varepsilon=0\) we write \(R_{0,0}^\pm(\lambda)\) for the corresponding
oscillatory boundary operators. Then
\[
(D-\lambda\mp i\varepsilon)R_{0,\varepsilon}^{\pm}(\lambda)h=h
\]
for \(\varepsilon>0\), and
\[
(D-\lambda)R_{0,0}^{\pm}(\lambda)h=h
\]
in the sense of distributions. Moreover,
\[
\|R_{0,\varepsilon}^{\pm}(\lambda)h\|_{L^\infty}
\leq
\|h\|_{L^1},
\qquad
\lambda\in I,\quad 0\leq\varepsilon\leq1.
\]
For every fixed \(h\in L^1(\mathbb R)\),
\[
R_{0,\varepsilon}^{\pm}(\lambda)h
\longrightarrow
R_{0,0}^{\pm}(\lambda)h
\]
locally uniformly in \(x\), uniformly for \(\lambda\in I\), as
\(\varepsilon\to0_+\).

Now put
\[
A_{\varepsilon}^{\pm}(\lambda)
:=
uT_{\bar u}R_{0,\varepsilon}^{\pm}(\lambda)
=
u\Pi\bigl(\bar u\,R_{0,\varepsilon}^{\pm}(\lambda)\,\cdot\bigr),
\qquad
0\leq\varepsilon\leq1 .
\]
Then \(A_{\varepsilon}^{\pm}(\lambda)\) is bounded on \(L^1(\mathbb R)\), with
\[
\|A_{\varepsilon}^{\pm}(\lambda)h\|_{L^1}
\leq
\|u\|_{L^2}^2\|h\|_{L^1}.
\]
Furthermore, the family
\[
\{A_{\varepsilon}^{\pm}(\lambda):\lambda\in I,\ 0\leq\varepsilon\leq1\}
\]
is collectively compact on \(L^1(\mathbb R)\). Indeed, approximate \(u\) in
\(L^2\) by a compactly supported smooth function \(v\). The error is small in
the operator norm \(L^1\to L^1\), since
\[
\begin{aligned}
&\left\|
u\Pi\bigl(\bar u\,R_{0,\varepsilon}^{\pm}(\lambda)h\bigr)
-
v\Pi\bigl(\bar v\,R_{0,\varepsilon}^{\pm}(\lambda)h\bigr)
\right\|_{L^1}  \\
&\leq
\|u-v\|_{L^2}\|u\|_{L^2}
\|R_{0,\varepsilon}^{\pm}(\lambda)h\|_{L^\infty}
+
\|v\|_{L^2}\|u-v\|_{L^2}
\|R_{0,\varepsilon}^{\pm}(\lambda)h\|_{L^\infty}  \\
&\leq
C_u\|u-v\|_{L^2}\|h\|_{L^1}.
\end{aligned}
\]
For compactly supported \(v\), the operator
\[
h\mapsto \bar v\,R_{0,\varepsilon}^{\pm}(\lambda)h
\]
is compact from \(L^1\) to \(L^2\), uniformly for
\(\lambda\in I\) and \(0\leq\varepsilon\leq1\). This follows by approximating
its kernel in \(L^2_x\), uniformly in the \(y\)-variable and in the parameters
\((\lambda,\varepsilon)\). Since \(v\Pi:L^2\to L^1\) is bounded, the collective
compactness follows.

We shall also use that, for every fixed \(h\in L^1(\mathbb R)\),
\[
A_{\varepsilon}^{\pm}(\lambda)h
\longrightarrow
A_{0}^{\pm}(\lambda)h
\quad\text{in }L^1,
\]
uniformly for \(\lambda\in I\), as \(\varepsilon\to0_+\). Indeed,
\[
\begin{aligned}
\|A_{\varepsilon}^{\pm}(\lambda)h
-
A_0^\pm(\lambda)h\|_{L^1}
&\leq
\|u\|_{L^2}
\left\|
\bar u
\left(
R_{0,\varepsilon}^{\pm}(\lambda)h
-
R_{0,0}^{\pm}(\lambda)h
\right)
\right\|_{L^2},
\end{aligned}
\]
and the right-hand side tends to zero by local uniform convergence of the free
boundary resolvents and by the \(L^2\)-tail smallness of \(u\). The same argument
gives strong continuity of \(A_{\varepsilon}^{\pm}(\lambda)\) in the parameters
\((\lambda,\varepsilon)\).

We next prove that
\[
I+A_{\varepsilon}^{\pm}(\lambda)
\]
is injective on \(L^1(\mathbb R)\), for every
\(\lambda\in I\) and \(0\leq\varepsilon\leq1\). Suppose first that
\[
h+A_{\varepsilon}^{+}(\lambda)h=0,
\qquad
g:=R_{0,\varepsilon}^{+}(\lambda)h.
\]
Then
\[
h=-uT_{\bar u}g,
\]
and hence \(h\) has non-negative Fourier support. Therefore \(g\) also has
non-negative Fourier support. For \(\varepsilon>0\) this follows immediately
from the Fourier multiplier representation of \(R_{0,\varepsilon}^{+}(\lambda)\);
for \(\varepsilon=0\), one tests
\[
(D-\lambda)g=h
\]
against functions whose Fourier supports are compactly contained in
\((-\infty,0)\).

If \(\varepsilon=0\), then \(g\in L^\infty_+(\mathbb R)\) and
\[
(L_u-\lambda)g=0
\]
in the sense of distributions. Moreover,
\[
e^{-i\lambda x}g(x)
=
i\int_{-\infty}^x e^{-i\lambda y}h(y)\,dy,
\]
so that
\[
\ell_g(-\infty)=0.
\]
By Proposition \(\ref{prop:generalized_eigenfunctions}\), the element of
\(\mathscr M_u(\lambda)\) with \(\ell_g(-\infty)=0\) is zero. Thus \(g=0\), and
then \(h=(D-\lambda)g=0\).

If \(0<\varepsilon\leq1\), then \(g\in L^2_+(\mathbb R)\), since
\(R_{0,\varepsilon}^{+}(\lambda)\) has an exponentially decaying kernel. Let
\[
z_+:=\lambda+i\varepsilon .
\]
The equation becomes
\[
(D-z_+)g=-uT_{\bar u}g.
\]
We claim that \(g\in\operatorname{Dom}(L_u)\). Let
\(P_R=\chi(D/R)\), where \(\chi\in C_c^\infty([0,\infty))\),
\(0\leq\chi\leq1\), and \(\chi(\xi/R)\uparrow1\) for \(\xi\ge0\). Testing
\[
Dg+uT_{\bar u}g=z_+g
\]
against \(P_Rg\), taking real parts, and letting \(R\to+\infty\), gives
\[
\int_0^\infty
\xi|\widehat g(\xi)|^2\,\frac{d\xi}{2\pi}
+
\|T_{\bar u}g\|_{L^2}^2
=
\lambda\|g\|_{L^2}^2
<+\infty.
\]
Hence \(g\in H^{1/2}_+(\mathbb R)\). The distributional equation then gives,
by density in \(H^{1/2}_+\),
\[
\mathcal Q_u(g,\varphi)
=
z_+\langle g,\varphi\rangle,
\qquad
\varphi\in H^{1/2}_+(\mathbb R).
\]
Thus \(g\in\operatorname{Dom}(L_u)\) and \(L_ug=z_+g\). Since \(L_u\) is
self-adjoint and \(z_+\notin\mathbb R\), this implies \(g=0\), and hence
\(h=0\).

The proof for \(I+A_{\varepsilon}^{-}(\lambda)\) is identical, with
\(z_-:=\lambda-i\varepsilon\). In the boundary case \(\varepsilon=0\), one has
\[
g:=R_{0,0}^{-}(\lambda)h,
\qquad
\ell_g(+\infty)=0.
\]
Since \(g\in\mathscr M_u(\lambda)\) and every
\(m\in\mathscr M_u(\lambda)\) satisfies
\[
|\ell_m(+\infty)|=|\ell_m(-\infty)|,
\]
we also get \(\ell_g(-\infty)=0\), and therefore \(g=0\).

Thus \(I+A_{\varepsilon}^{\pm}(\lambda)\) is injective for all
\(\lambda\in I\) and \(0\leq\varepsilon\leq1\). Since
\(A_{\varepsilon}^{\pm}(\lambda)\) is compact on \(L^1\), Fredholm's alternative
gives invertibility. The collective compactness and the strong continuity in the
parameters imply the local uniform bound
\[
\sup_{\lambda\in I,\ 0\leq\varepsilon\leq1}
\left\|
(I+A_{\varepsilon}^{\pm}(\lambda))^{-1}
\right\|_{\mathcal L(L^1,L^1)}
<\infty .
\]
Moreover, for every fixed \(r\in L^1(\mathbb R)\),
\[
(I+A_{\varepsilon}^{\pm}(\lambda))^{-1}r
\longrightarrow
(I+A_0^\pm(\lambda))^{-1}r
\quad\text{in }L^1,
\]
locally uniformly for \(\lambda\in I\), as \(\varepsilon\to0_+\).

Applying this to \(r=f\), define
\[
h_\pm^\varepsilon(\cdot,\lambda)
:=
(I+A_\varepsilon^\pm(\lambda))^{-1}f,
\qquad
h_\pm(\cdot,\lambda)
:=
(I+A_0^\pm(\lambda))^{-1}f.
\]
Then
\[
h_\pm^\varepsilon(\cdot,\lambda)
\to
h_\pm(\cdot,\lambda)
\quad\text{in }L^1,
\]
locally uniformly for \(\lambda\in I\). Put
\[
\Gamma_\pm^\varepsilon(\cdot,\lambda)
:=
R_{0,\varepsilon}^{\pm}(\lambda)h_\pm^\varepsilon(\cdot,\lambda).
\]
Since
\[
h_\pm^\varepsilon
+
uT_{\bar u}\Gamma_\pm^\varepsilon
=
f,
\]
we have
\[
(L_u-\lambda\mp i\varepsilon)\Gamma_\pm^\varepsilon=f
\]
in the sense of distributions. The same support and cutoff argument as above
shows that, for \(0<\varepsilon\leq1\),
\[
\Gamma_\pm^\varepsilon\in\operatorname{Dom}(L_u).
\]
Hence, by uniqueness of the resolvent of the self-adjoint operator \(L_u\),
\[
\Gamma_+^\varepsilon
=
(L_u-\lambda-i\varepsilon)^{-1}f
=
g_+^\varepsilon,
\]
and
\[
\Gamma_-^\varepsilon
=
(L_u-\lambda+i\varepsilon)^{-1}f
=
g_-^\varepsilon .
\]
Let
\[
g_\pm(\cdot,\lambda)
:=
R_{0,0}^{\pm}(\lambda)h_\pm(\cdot,\lambda).
\]
Then, for every \(\lambda\in I\),
\[
g_\pm^\varepsilon(\cdot,\lambda)
\longrightarrow
g_\pm(\cdot,\lambda)
\]
locally uniformly in \(x\), and also in duality against \(L^1\)-functions. In
particular,
\[
\left\langle g_\pm^\varepsilon(\cdot,\lambda),f\right\rangle
\longrightarrow
\left\langle g_\pm(\cdot,\lambda),f\right\rangle .
\]
The convergence is dominated on \(I\), since
\[
\|g_\pm^\varepsilon(\cdot,\lambda)\|_{L^\infty}
\leq
C_I\|f\|_{L^1}.
\]
The limiting functions satisfy
\[
h_\pm=f-uT_{\bar u}g_\pm.
\]
Since \(f\in L^2_+\) and \(uT_{\bar u}g_\pm\) has non-negative Fourier support,
the functions \(h_\pm\) have non-negative Fourier support. Hence
\[
g_\pm=R_{0,0}^{\pm}(\lambda)h_\pm
\in L^\infty_+(\mathbb R).
\]
Moreover,
\[
(L_u-\lambda)g_\pm=f
\]
in the sense of distributions, and therefore
\[
(L_u-\lambda)(g_+-g_-)=0.
\]
Thus
\[
g_+-g_-\in\mathscr M_u(\lambda).
\]
From the asymptotics of the free boundary resolvents,
\[
e^{-i\lambda x}g_+(x,\lambda)
=
i\int_{-\infty}^x e^{-i\lambda y}h_+(y,\lambda)\,dy,
\]
and
\[
e^{-i\lambda x}g_-(x,\lambda)
=
-i\int_x^{+\infty}e^{-i\lambda y}h_-(y,\lambda)\,dy,
\]
we get
\[
\ell_{g_+}(-\infty)=0,
\qquad
\ell_{g_-}(+\infty)=0.
\]
Since \(m_-(\cdot,\lambda)\) is normalized by
\[
\ell_{m_-}(-\infty)=1,
\]
and since \(g_+-g_-\in\mathscr M_u(\lambda)\), we obtain
\[
g_+-g_-=-\ell_{g_-}(-\infty)m_-(\cdot,\lambda).
\]
It remains to compute the coefficient. Since
\[
(L_u-\lambda)g_-=f,
\qquad
(L_u-\lambda)m_-=0,
\]
we have, after integrating by parts on \((-R,R)\),
\[
\begin{aligned}
&\int_{-R}^R
\left(
g_-\overline{Dm_-}
-
Dg_-\,\overline{m_-}
\right)\,dx  \\
&=
-\int_{-R}^R f\,\overline{m_-}\,dx
+
\int_{-R}^R
\left[
uT_{\bar u}g_-\,\overline{m_-}
-
g_-\,\overline{uT_{\bar u}m_-}
\right]\,dx .
\end{aligned}
\]
Letting \(R\to+\infty\), the last two terms cancel. Indeed,
\[
\int_{\mathbb R}
uT_{\bar u}g_-\,\overline{m_-}\,dx
=
\int_{\mathbb R}
T_{\bar u}g_-\,\overline{T_{\bar u}m_-}\,dx
=
\int_{\mathbb R}
g_-\,\overline{uT_{\bar u}m_-}\,dx .
\]
On the other hand,
\[
g_-\overline{Dm_-}-Dg_-\,\overline{m_-}
=
i\partial_x(g_-\overline{m_-}).
\]
Using
\[
\ell_{g_-}(+\infty)=0,
\qquad
\ell_{m_-}(-\infty)=1,
\]
we obtain
\[
\lim_{R\to+\infty}
\int_{-R}^R
\left(
g_-\overline{Dm_-}
-
Dg_-\,\overline{m_-}
\right)\,dx
=
-i\ell_{g_-}(-\infty).
\]
Therefore
\[
-i\ell_{g_-}(-\infty)
=
-\int_{\mathbb R} f(x)\overline{m_-(x,\lambda)}\,dx
=
-\widetilde f(\lambda).
\]
Hence
\[
g_+(x,\lambda)-g_-(x,\lambda)
=
i\widetilde f(\lambda)m_-(x,\lambda).
\]
Consequently,
\[
\left\langle g_+(\cdot,\lambda)-g_-(\cdot,\lambda),f\right\rangle
=
i|\widetilde f(\lambda)|^2 .
\]

Now let \(\chi\in C_c^\infty(0,\infty)\), and choose \(I\Subset(0,\infty)\)
such that \(\operatorname{supp}\chi\subset I\). By Stone's formula,
\[
\left\langle \chi(L_u)f,f\right\rangle
=
\frac{1}{2\pi i}
\lim_{\varepsilon\to0_+}
\int_0^\infty
\chi(\lambda)
\left\langle
g_+^\varepsilon(\cdot,\lambda)
-
g_-^\varepsilon(\cdot,\lambda),
f
\right\rangle
\,d\lambda .
\]
Using the convergence in duality proved above and dominated convergence on the
support of \(\chi\), we get
\[
\left\langle \chi(L_u)f,f\right\rangle
=
\frac{1}{2\pi}
\int_0^\infty
\chi(\lambda)|\widetilde f(\lambda)|^2\,d\lambda .
\]
By applying this identity to \(f+\zeta g\) and polarizing, we also obtain, for
all \(f,g\in L^1(\mathbb R)\cap L^2_+(\mathbb R)\),
\begin{equation}
\label{eq:local_spectral_identity}
\left\langle \chi(L_u)f,g\right\rangle
=
\frac{1}{2\pi}
\int_0^\infty
\chi(\lambda)\widetilde f(\lambda)
\overline{\widetilde g(\lambda)}\,d\lambda ,
\qquad
\chi\in C_c^\infty(0,\infty).
\end{equation}

Choose a sequence
\[
0\leq \chi_n\in C_c^\infty(0,\infty),
\qquad
\chi_n(\lambda)\uparrow 1
\quad\text{for every }\lambda>0.
\]
Since \(L_u\) has no eigenvalues, the spectral measure of \(f\) has no atom at
\(\lambda=0\). Hence
\[
\chi_n(L_u)f\to f
\quad\text{in }L^2.
\]
Therefore
\[
\|f\|_{L^2}^2
=
\lim_{n\to\infty}
\left\langle \chi_n(L_u)f,f\right\rangle
=
\frac{1}{2\pi}
\lim_{n\to\infty}
\int_0^\infty
\chi_n(\lambda)|\widetilde f(\lambda)|^2\,d\lambda .
\]
By monotone convergence,
\[
\|f\|_{L^2}^2
=
\frac{1}{2\pi}
\int_0^\infty
|\widetilde f(\lambda)|^2\,d\lambda .
\]
This proves \eqref{3.3} for every
\(f\in L^1(\mathbb R)\cap L^2_+(\mathbb R)\).

Finally, \(L^1(\mathbb R)\cap L^2_+(\mathbb R)\) is dense in
\(L^2_+(\mathbb R)\). The estimate just proved gives
\[
\|\mathcal F_{d,u}f\|_{L^2(0,\infty)}
=
(2\pi)^{1/2}\|f\|_{L^2}
\]
on this dense subspace. Thus \(\mathcal F_{d,u}\) extends uniquely and
continuously to \(L^2_+(\mathbb R)\), and the same identity holds for all
\(f\in L^2_+(\mathbb R)\). By density, the polarized identity
\eqref{eq:local_spectral_identity} also holds for all
\(f,g\in L^2_+(\mathbb R)\).

Since \(L_u\) has no eigenvalues, its pure point spectral subspace is trivial,
and therefore
\[
P_c(L_u)=\operatorname{Id}
\quad\text{on }L^2_+(\mathbb R).
\]
Hence, for every \(f\in L^2_+(\mathbb R)\),
\[
\int_{\mathbb R}|f(x)|^2\,dx
=
\int_{\mathbb R}|P_c(L_u)f(x)|^2\,dx
=
\frac{1}{2\pi}
\int_0^\infty|\widetilde f(\lambda)|^2\,d\lambda .
\]
The proof is complete.
\end{proof}
The isometry immediately gives the corresponding inner-product identity. We shall also
use the compatibility of the transform with the Lax operator on its natural domain.
\begin{remark}
By the polarization identity, the isometry property implies preservation of the inner product, i.e., for any $f, g \in L_{+}^2(\mathbb{R})$, we have
\begin{equation}
\label{2.4}
\left\langle f, g \right\rangle_{L^2(\mathbb{R})} = \frac{1}{2\pi}\langle \widetilde{f}, \widetilde{g} \rangle_{L^2(0,\infty)}.
\end{equation}
\end{remark}
\begin{remark}
\label{remark 2.14}
Let \(u\in L^{2}_+(\mathbb R)\). Then, for every
\(f\in \operatorname{Dom}(L_u)\), one has
\[
\mathcal F_{d,u}(L_u f)(\lambda)
=
\lambda \mathcal F_{d,u}f(\lambda)
\]
as an identity in \(L^2(0,\infty)\).

Indeed, set
\[
\mathcal U:=\frac{1}{\sqrt{2\pi}}\mathcal F_{d,u}.
\]
By Proposition \(\ref{prop 2.6}\), \(\mathcal U\) is an isometry from
\(L^2_+(\mathbb R)\) into \(L^2(0,\infty)\). Let
\[
\mathcal K:=\mathcal U L^2_+(\mathbb R)
\]
be its closed range.

We first show that, for every \(\chi\in C_c^\infty(0,\infty)\),
\begin{equation}
\label{eq:compact_functional_calculus}
\mathcal U\chi(L_u)=M_\chi\mathcal U,
\end{equation}
where \(M_\chi\) denotes multiplication by \(\chi(\lambda)\). By
\eqref{eq:local_spectral_identity}, for every \(f,g\in L^2_+(\mathbb R)\),
\[
\left\langle
\mathcal U\chi(L_u)f,\mathcal Ug
\right\rangle_{L^2(0,\infty)}
=
\left\langle
M_\chi\mathcal Uf,\mathcal Ug
\right\rangle_{L^2(0,\infty)}.
\]
Hence
\[
\mathcal U\chi(L_u)f
=
P_{\mathcal K}M_\chi\mathcal Uf,
\]
where \(P_{\mathcal K}\) is the orthogonal projection onto \(\mathcal K\).

Assume first that \(\chi\) is real-valued. Applying
\eqref{eq:local_spectral_identity} with \(\chi^2\), we obtain
\[
\|M_\chi\mathcal Uf\|_{L^2(0,\infty)}^2
=
\left\langle \chi^2(L_u)f,f\right\rangle_{L^2}
=
\|\chi(L_u)f\|_{L^2}^2
=
\|\mathcal U\chi(L_u)f\|_{L^2(0,\infty)}^2 .
\]
Since \(\mathcal U\chi(L_u)f\) is the orthogonal projection of
\(M_\chi\mathcal Uf\) onto \(\mathcal K\), equality of the two norms implies
\[
M_\chi\mathcal Uf\in\mathcal K.
\]
Therefore
\[
\mathcal U\chi(L_u)f
=
M_\chi\mathcal Uf.
\]
By linearity, the same conclusion holds for complex-valued
\(\chi\in C_c^\infty(0,\infty)\). This proves
\eqref{eq:compact_functional_calculus}.

We now pass from compactly supported functions of \(L_u\) to the resolvent.
Choose \(\theta_n\in C_c^\infty(0,\infty)\) such that
\[
0\leq \theta_n\leq1,
\qquad
\theta_n(\lambda)=1\quad\text{for }1/n\leq\lambda\leq n,
\]
and set
\[
\rho_n(\lambda):=\frac{\theta_n(\lambda)}{\lambda-i}.
\]
Then \(\rho_n\in C_c^\infty(0,\infty)\),
\[
\rho_n(\lambda)\to\frac1{\lambda-i}
\quad\text{for every }\lambda>0,
\]
and the sequence \((\rho_n)\) is uniformly bounded. Since \(L_u\ge0\) and
\(0\) is not an eigenvalue of \(L_u\), the spectral theorem gives
\[
\rho_n(L_u)\to (L_u-i)^{-1}
\quad\text{strongly on }L^2_+(\mathbb R).
\]
On the other hand, by dominated convergence,
\[
M_{\rho_n}\mathcal Uh
\to
M_{(\lambda-i)^{-1}}\mathcal Uh
\quad\text{in }L^2(0,\infty)
\]
for every \(h\in L^2_+(\mathbb R)\). Passing to the limit in
\[
\mathcal U\rho_n(L_u)=M_{\rho_n}\mathcal U
\]
yields
\begin{equation}
\label{eq:resolvent_intertwining}
\mathcal U(L_u-i)^{-1}
=
M_{(\lambda-i)^{-1}}\mathcal U.
\end{equation}

Now let \(f\in\operatorname{Dom}(L_u)\) and put
\[
h:=(L_u-i)f\in L^2_+(\mathbb R).
\]
Then
\[
f=(L_u-i)^{-1}h.
\]
Using \eqref{eq:resolvent_intertwining}, we get
\[
\mathcal Uf
=
\mathcal U(L_u-i)^{-1}h
=
M_{(\lambda-i)^{-1}}\mathcal Uh.
\]
Equivalently,
\[
\mathcal Uh=(\lambda-i)\mathcal Uf
\]
in \(L^2(0,\infty)\). Since \(h=L_uf-if\), this gives
\[
\mathcal U(L_uf)-i\mathcal Uf
=
(\lambda-i)\mathcal Uf.
\]
Therefore
\[
\mathcal U(L_uf)=\lambda\mathcal Uf
\]
in \(L^2(0,\infty)\). Multiplying by \(\sqrt{2\pi}\), we obtain
\[
\mathcal F_{d,u}(L_u f)(\lambda)
=
\lambda \mathcal F_{d,u}f(\lambda)
\]
as an identity in \(L^2(0,\infty)\).
\end{remark}
\noindent \noindent
It remains to identify the range of the isometry. The next lemma proves that no spectral
subspace is missing: the distorted Fourier transform is onto \(L^2(0,\infty)\).
\begin{lemma}
\label{lemma 2.15}
For $u \in L_{+}^{2}(\mathbb R)$, the map $f \in L_{+}^2(\mathbb{R}) \mapsto \widetilde{f}:=\mathcal F_{d,u} f \in L^2(0, \infty)$ is onto.
\end{lemma}
\begin{proof}
We use the normalized distorted Fourier transform
\[
\mathcal U f:=\frac{1}{\sqrt{2\pi}}\mathcal F_{d,u} f .
\]
By Proposition \ref{prop 2.6}, \(\mathcal U\) is an isometry from
\(L^2_+(\mathbb R)\) into \(L^2(0,\infty)\). Hence its range
\[
\mathcal K:=\mathcal U L^2_+(\mathbb R)
\]
is a closed subspace of \(L^2(0,\infty)\).

We first prove that \(\mathcal K\) is a reducing subspace for the multiplication operator
\[
M_\lambda\varphi(\lambda):=\lambda\varphi(\lambda)
\]
on \(L^2(0,\infty)\). Let \(L_u\) be considered as a self-adjoint operator on
\(L^2_+(\mathbb R)\), and define
\[
T:=\mathcal U L_u\mathcal U^{-1}
\]
as an operator on the Hilbert space \(\mathcal K\), with domain
\[
\operatorname{Dom}(T):=\mathcal U\operatorname{Dom}(L_u).
\]
Since \(\mathcal U:L^2_+(\mathbb R)\to\mathcal K\) is unitary, \(T\) is self-adjoint on
\(\mathcal K\).

By Remark \ref{remark 2.14}, for every \(f\in\operatorname{Dom}(L_u)\),
\[
\mathcal U(L_uf)=M_\lambda\mathcal U f.
\]
Therefore, for every \(h\in\operatorname{Dom}(T)\),
\begin{equation}
\label{2.15.1}
Th=M_\lambda h.
\end{equation}
In particular,
\[
T\subset M_\lambda
\]
in the sense that \(T\) is the restriction of \(M_\lambda\) to the Hilbert space
\(\mathcal K\).

Let \(z\in\mathbb C\setminus\mathbb R\). Since \(T\) is self-adjoint on \(\mathcal K\),
the resolvent
\[
(T-z)^{-1}:\mathcal K\to\operatorname{Dom}(T)
\]
is well-defined and bounded. We claim that
\begin{equation}
\label{2.15.2}
(M_\lambda-z)^{-1}\mathcal K\subset\mathcal K.
\end{equation}
Indeed, if \(h\in\mathcal K\), set
\[
g:=(T-z)^{-1}h\in\operatorname{Dom}(T)\subset\mathcal K.
\]
Using \eqref{2.15.1}, we have
\[
(M_\lambda-z)g=(T-z)g=h.
\]
By uniqueness of the solution to
\[
(M_\lambda-z)g=h
\]
in \(L^2(0,\infty)\), it follows that
\[
g=(M_\lambda-z)^{-1}h.
\]
Thus \((M_\lambda-z)^{-1}h\in\mathcal K\), proving \eqref{2.15.2}.

Since \eqref{2.15.2} holds for both \(z\) and \(\overline z\), and
\[
\bigl((M_\lambda-z)^{-1}\bigr)^*
=
(M_\lambda-\overline z)^{-1},
\]
the subspace \(\mathcal K\) reduces the resolvent of \(M_\lambda\). Hence \(\mathcal K\)
reduces all spectral projections of \(M_\lambda\). Equivalently, the orthogonal
projection \(P_{\mathcal K}\) commutes with every multiplication operator
\[
\mathbf 1_\Omega(M_\lambda),
\qquad \Omega\subset(0,\infty)\ \text{Borel}.
\]
Since \(M_\lambda\) has multiplicity one, there exists a measurable set
\(E\subset(0,\infty)\) such that
\begin{equation}
\label{2.15.3}
\mathcal K=L^2(E),
\end{equation}
where \(L^2(E)\) is identified with the closed subspace of \(L^2(0,\infty)\) consisting of
functions vanishing a.e. on \((0,\infty)\setminus E\).

It remains to prove that \(E\) has full Lebesgue measure. Suppose, for contradiction, that
\[
|(0,\infty)\setminus E|>0.
\]
Then there exists a compact interval
\[
I\subset(0,\infty)
\]
such that
\begin{equation}
\label{2.15.4}
|I\setminus E|>0.
\end{equation}
Choose \(\psi\in\mathcal S_+(\mathbb R)\) such that
\begin{equation}
\label{2.15.5}
\widehat\psi(\lambda)=1
\qquad\text{for }\lambda\in I.
\end{equation}
For \(R>0\), define the left translate
\[
\psi_R(x):=\psi(x+R).
\]
Then \(\psi_R\in\mathcal S_+(\mathbb R)\), and, with our Fourier convention,
\begin{equation}
\label{2.15.6}
\widehat{\psi_R}(\lambda)=e^{iR\lambda}\widehat\psi(\lambda).
\end{equation}
We compare \(\mathcal F_{d,u}\psi_R\) with the ordinary Fourier transform of \(\psi_R\) on
\(I\). Recall that
\[
m_-(x,\lambda)=e^{i\lambda x}a(x,\lambda).
\]
Since
\[
a_\lambda=1+K_\lambda a_\lambda,
\]
and since Proposition \ref{prop:uniform_m} gives
\[
\sup_{\lambda\ge0}\|a(\cdot,\lambda)\|_{L^\infty}<\infty,
\]
we have, for \(x<0\),
\begin{equation}
\label{2.15.7}
\begin{aligned}
|a(x,\lambda)-1|
&\le
\int_{-\infty}^x
|u(y)|
\left|
\Pi\bigl(\overline u\,e^{i\lambda\cdot}a(\cdot,\lambda)\bigr)(y)
\right|\,dy  \\
&\le
\left(\int_{-\infty}^x |u(y)|^2\,dy\right)^{1/2}
\left\|
\Pi\bigl(\overline u\,e^{i\lambda\cdot}a(\cdot,\lambda)\bigr)
\right\|_{L^2}  \\
&\le
C_u
\left(\int_{-\infty}^x |u(y)|^2\,dy\right)^{1/2}.
\end{aligned}
\end{equation}
The right-hand side tends to \(0\) as \(x\to-\infty\), uniformly in \(\lambda\ge0\).
Hence
\[
a(x,\lambda)\to1
\qquad\text{as }x\to-\infty,
\]
uniformly for \(\lambda\in I\).

For \(\lambda\in I\), we compute
\begin{equation}
\label{2.15.8}
\begin{aligned}
\mathcal F_{d,u}\psi_R(\lambda)-\widehat{\psi_R}(\lambda)
&=
\int_{\mathbb R}
\psi_R(x)e^{-i\lambda x}
\bigl(\overline{a(x,\lambda)}-1\bigr)\,dx  \\
&=
e^{iR\lambda}
\int_{\mathbb R}
\psi(y)e^{-i\lambda y}
\bigl(\overline{a(y-R,\lambda)}-1\bigr)\,dy .
\end{aligned}
\end{equation}
We claim that
\begin{equation}
\label{2.15.9}
\sup_{\lambda\in I}
\left|
\mathcal F_{d,u}\psi_R(\lambda)-\widehat{\psi_R}(\lambda)
\right|
\longrightarrow0
\qquad\text{as }R\to+\infty.
\end{equation}
Indeed, fix \(\varepsilon>0\). Since \(\psi\in\mathcal S(\mathbb R)\) and \(a\) is
uniformly bounded, we can choose \(M>0\) such that the contribution of \(|y|>M\) in
\eqref{2.15.8} is \(<\varepsilon\), uniformly in \(R\) and \(\lambda\in I\). On
\(|y|\le M\), we have \(y-R\to-\infty\) uniformly as \(R\to+\infty\), and therefore
\[
\sup_{\substack{|y|\le M\\ \lambda\in I}}
|a(y-R,\lambda)-1|\to0
\]
by \eqref{2.15.7}. This proves \eqref{2.15.9}. Consequently,
\begin{equation}
\label{2.15.10}
\left\|
\mathcal F_{d,u}\psi_R-e^{iR\lambda}\widehat\psi
\right\|_{L^2(I)}
\longrightarrow0
\qquad\text{as }R\to+\infty.
\end{equation}
On the other hand, since
\[
\mathcal U\psi_R\in\mathcal K=L^2(E),
\]
we have
\begin{equation}
\label{2.15.11}
\mathcal F_{d,u}\psi_R(\lambda)=0
\qquad\text{for a.e. }\lambda\in(0,\infty)\setminus E.
\end{equation}
Restricting \eqref{2.15.10} to \(I\setminus E\), and using \eqref{2.15.5}, we obtain
\[
\left\|
e^{iR\lambda}\widehat\psi(\lambda)
\right\|_{L^2(I\setminus E)}
\longrightarrow0.
\]
But \(\widehat\psi(\lambda)=1\) on \(I\), so the left-hand side is exactly
\[
|I\setminus E|^{1/2},
\]
which is strictly positive by \eqref{2.15.4}. This contradiction shows that
\[
|(0,\infty)\setminus E|=0.
\]
Therefore
\[
\mathcal K=L^2(0,\infty).
\]
Hence \(\mathcal U\), and consequently \(\mathcal F_{d,u}\), is onto.
\end{proof}
The isometry and surjectivity show that the normalized distorted Fourier transform is
unitary. Consequently, \(L_u\) becomes the multiplication operator by \(\lambda\) in the
distorted Fourier variables.
\begin{corollary}
\label{coro 2.14}
For $u \in L_{+}^{2}(\mathbb R)$, the operator \(L_u\) is unitarily equivalent to multiplication by \(\lambda\) on
\(L^2(0,\infty)\). In particular,
\[
\sigma(L_u)=\sigma_{\mathrm{ac}}(L_u)=[0,\infty).
\]
\end{corollary}
\begin{proof}
By Proposition \ref{prop 2.6} and Lemma \ref{lemma 2.15}, the normalized distorted Fourier transform
\[
\mathcal U:=\frac{1}{\sqrt{2\pi}}\mathcal{F}_{d,u}
\]
is a unitary operator
\[
\mathcal U:L^2_+(\mathbb R)\longrightarrow L^2(0,\infty).
\]
Moreover, by Remark \ref{remark 2.14}, for every \(f\in \operatorname{Dom}(L_u)\),
\[
\mathcal U(L_uf)(\lambda)=\lambda\,\mathcal U f(\lambda).
\]
Equivalently,
\[
\mathcal U L_u\mathcal U^{-1}=M_\lambda,
\]
where \(M_\lambda\) denotes the operator of multiplication by \(\lambda\) on
\(L^2(0,\infty)\), with domain
\[
\operatorname{Dom}(M_\lambda)
=
\left\{
\varphi\in L^2(0,\infty):\lambda\varphi(\lambda)\in L^2(0,\infty)
\right\}.
\]
Therefore \(L_u\) is unitarily equivalent to \(M_\lambda\). The spectrum of
\(M_\lambda\) is the essential range of the function \(\lambda\mapsto\lambda\) on
\((0,\infty)\), hence
\[
\sigma(M_\lambda)=[0,\infty).
\]
Furthermore, the spectral measure of \(M_\lambda\) is
\[
E_{M_\lambda}(\Omega)\varphi
=
\mathbf 1_\Omega(\lambda)\varphi(\lambda),
\qquad
\Omega\subset[0,\infty)\ \text{Borel}.
\]
Thus all scalar spectral measures of \(M_\lambda\) are absolutely continuous with respect
to Lebesgue measure:
\[
\langle E_{M_\lambda}(\Omega)\varphi,\varphi\rangle
=
\int_{\Omega\cap(0,\infty)}|\varphi(\lambda)|^2\,d\lambda.
\]
Consequently,
\[
\sigma(M_\lambda)=\sigma_{\mathrm{ac}}(M_\lambda)=[0,\infty),
\]
and \(M_\lambda\) has no point spectrum and no singular continuous spectrum. By unitary
equivalence, the same holds for \(L_u\). Hence
\[
\sigma(L_u)=\sigma_{\mathrm{ac}}(L_u)=[0,\infty).
\]
\end{proof}

\noindent
We next show the Schwartz-space mapping property of the distorted Fourier transform in
the smooth case. This result will allow us to extend the transform to distributions and
to justify the spectral computations in the next section.
\begin{proposition}
\label{proposition 3.6}
Let \(u\in \mathcal{S}_{+}(\mathbb R):=\mathcal{S}(\mathbb R)\cap L^2_+(\mathbb R)\). Then the distorted Fourier transform
\[
\mathcal{F}_{d,u}:\mathcal{S}_{+}(\mathbb R):=\mathcal{S}(\mathbb R)\cap L^2_+(\mathbb R)\longrightarrow
\mathcal{S}_{*},
\]
where
\[
\mathcal{S}_{*}:=\{\varphi=g|_{(0,\infty)}:\ g\in \mathcal{S}(\mathbb R),\
\operatorname{supp}g\subset[0,\infty)\},
\]
is well-defined and is a continuous bijection.
\end{proposition}

\begin{proof}
Write
\[
m_-(x,\lambda)=e^{i\lambda x}a(x,\lambda),
\qquad \lambda>0,
\]
where \(a_\lambda=a(\cdot,\lambda)\) solves
\[
(I-K_\lambda)a_\lambda=1
\]
with
\[
K_\lambda b(x)
:=
-i\int_{-\infty}^x
e^{-i\lambda y}u(y)
\Pi\bigl(\overline u\,e^{i\lambda\cdot}b\bigr)(y)\,dy .
\]
By Remark \ref{remark 2.7}, the same equation is meaningful at \(\lambda=0\), and
\[
a(x,0)=1.
\]
Moreover, by Proposition \ref{prop:uniform_m},
\begin{equation}
\label{3.6.1}
\sup_{\lambda\ge0}
\|(I-K_\lambda)^{-1}\|_{\mathcal L(L^\infty,L^\infty)}<\infty .
\end{equation}

We first record the regularity properties of the amplitude \(a\). Since
\(u\in\mathcal S_+(\mathbb R)\), for every \(j\ge0\) the map
\[
\lambda\mapsto K_\lambda
\]
is \(C^\infty\) from \([0,\infty)\) to
\(\mathcal L(L^\infty,L^\infty)\). Indeed, after differentiating \(K_\lambda\) in
\(\lambda\), one only produces factors \(y^r u(y)\) outside the projector and factors
\(y^s\overline u(y)\) inside the projector. Hence, by the Cauchy--Schwarz inequality and
the \(L^2\)-boundedness of \(\Pi\),
\begin{equation}
\label{3.6.2}
\|\partial_\lambda^jK_\lambda b\|_{L^\infty}
\le C_j\|b\|_{L^\infty},
\qquad \lambda\ge0 .
\end{equation}
Differentiating
\[
(I-K_\lambda)a_\lambda=1
\]
\(k\) times in \(\lambda\), and using \eqref{3.6.1} and \eqref{3.6.2}, gives
\begin{equation}
\label{3.6.3}
\sup_{\lambda\ge0}
\|\partial_\lambda^k a(\cdot,\lambda)\|_{L^\infty}<\infty,
\qquad k\ge0 .
\end{equation}
We shall also use the following mixed \(x\)- and \(\lambda\)-bounds. For every
\(n,k\ge0\), there exists \(C_{n,k}>0\) such that
\begin{equation}
\label{3.6.4}
\sup_{\lambda\ge0}\sup_{x\in\mathbb R}
|\partial_x^n\partial_\lambda^k a(x,\lambda)|
\le C_{n,k}.
\end{equation}
We include the proof. For \(\lambda\ge0\), define
\[
T_\lambda q
:=
e^{-i\lambda x}\Pi(e^{i\lambda\cdot}q)(x).
\]
Then, in Fourier variables,
\[
\widehat{T_\lambda q}(\xi)
=
\mathbf 1_{\xi\ge-\lambda}\widehat q(\xi).
\]
Therefore, for every integer \(m\ge0\),
\[
\|T_\lambda q\|_{H^m}
\le
\|q\|_{H^m},
\qquad \lambda\ge0.
\]
Set
\[
P_\lambda
:=
\Pi\bigl(\overline u\,e^{i\lambda\cdot}a(\cdot,\lambda)\bigr).
\]
Then
\begin{equation}
\label{3.6.5}
\partial_x a(x,\lambda)
=
-i e^{-i\lambda x}u(x)P_\lambda(x).
\end{equation}
After differentiating \eqref{3.6.5} in \(\lambda\), one obtains, for every \(k\ge0\),
\begin{equation}
\label{3.6.6}
\partial_x\partial_\lambda^k a(x,\lambda)
=
\sum_{r+s+j\le k}
c_{r,s,j,k}\,
x^r u(x)\,
T_\lambda\bigl(x^s\overline u\,\partial_\lambda^j a(\cdot,\lambda)\bigr)(x).
\end{equation}
Here the factors \(x^r\) and \(x^s\) come from differentiating
\(e^{-i\lambda x}\) and \(e^{i\lambda y}\) with respect to \(\lambda\), but no positive
power of \(\lambda\) is produced.

We prove by induction on \(m\ge0\) that, for every \(k\ge0\),
\[
\sup_{\lambda\ge0}
\|\partial_x\partial_\lambda^k a(\cdot,\lambda)\|_{H^m}
<\infty .
\]
For \(m=0\), by \eqref{3.6.3},
\[
\sup_{\lambda\ge0}
\|\partial_\lambda^j a(\cdot,\lambda)\|_{L^\infty}<\infty
\]
for every \(j\ge0\). Since \(u\in\mathcal S_+(\mathbb R)\), we have
\[
\sup_{\lambda\ge0}
\|x^s\overline u\,\partial_\lambda^j a(\cdot,\lambda)\|_{L^2}
<\infty .
\]
Using the uniform \(L^2\)-boundedness of \(T_\lambda\), and then \eqref{3.6.6}, gives
\[
\sup_{\lambda\ge0}
\|\partial_x\partial_\lambda^k a(\cdot,\lambda)\|_{L^2}
<\infty .
\]
Thus the claim holds for \(m=0\).

Assume now that the claim holds for some \(m\ge0\), for every \(k\ge0\). Then, for every
\(s,j\),
\[
x^s\overline u\,\partial_\lambda^j a(\cdot,\lambda)
\]
is bounded in \(H^{m+1}(\mathbb R)\), uniformly in \(\lambda\ge0\). Indeed, after applying
at most \(m+1\) derivatives in \(x\), each term is a product of a Schwartz function and
some \(\partial_x^\ell\partial_\lambda^j a\), with \(0\le \ell\le m+1\). The case
\(\ell=0\) is controlled by \eqref{3.6.3}, while the cases \(\ell\ge1\) are controlled by
the induction hypothesis.

Using the uniform \(H^{m+1}\)-boundedness of \(T_\lambda\), we obtain
\[
\sup_{\lambda\ge0}
\left\|
T_\lambda\bigl(x^s\overline u\,\partial_\lambda^j a(\cdot,\lambda)\bigr)
\right\|_{H^{m+1}}
<\infty .
\]
Multiplication by the fixed Schwartz function \(x^r u(x)\) is bounded on
\(H^{m+1}(\mathbb R)\). Therefore \eqref{3.6.6} gives
\[
\sup_{\lambda\ge0}
\|\partial_x\partial_\lambda^k a(\cdot,\lambda)\|_{H^{m+1}}
<\infty .
\]
This closes the induction.

Finally, for \(n=0\), \eqref{3.6.4} is exactly \eqref{3.6.3}. For \(n\ge1\), choose
\(m=n\) in the estimate just proved. Then
\[
\partial_x^n\partial_\lambda^k a
=
\partial_x^{n-1}\bigl(\partial_x\partial_\lambda^k a\bigr)
\]
belongs to \(H^1(\mathbb R)\), uniformly in \(\lambda\ge0\). Since
\(H^1(\mathbb R)\hookrightarrow L^\infty(\mathbb R)\), we obtain \eqref{3.6.4}.

We next prove the behavior at \(\lambda=0\). Since \(u\in\mathcal S_+(\mathbb R)\), its
ordinary Fourier transform belongs to \(\mathcal S_*\). Therefore
\begin{equation}
\label{3.6.7}
\partial_\xi^j\widehat u(0)=0,
\qquad j\ge0 .
\end{equation}
For \(\lambda\ge0\),
\begin{equation}
\label{3.6.8}
\widehat{\Pi(\overline u\,e^{i\lambda\cdot})}(\xi)
=
\mathbf 1_{\xi\ge0}\widehat{\overline u}(\xi-\lambda)
=
\mathbf 1_{0\le\xi\le\lambda}\,
\overline{\widehat u(\lambda-\xi)} .
\end{equation}
Because all derivatives of \(\widehat u\) vanish at \(0\), differentiating
\eqref{3.6.8} in \(\lambda\) produces no boundary contributions. Hence, for every
\(j,N\ge0\),
\begin{equation}
\label{3.6.9}
\left\|
\partial_\lambda^j
\Pi(\overline u\,e^{i\lambda\cdot})
\right\|_{L^2}
\le C_{j,N}\lambda^N
\qquad\text{as }\lambda\to0_{+}.
\end{equation}
Indeed, after differentiating \(j\) times in \(\lambda\), the Fourier transform is
supported in \(0\le\xi\le\lambda\) and is bounded by
\(C_N(\lambda-\xi)^N\) for every \(N\).

Differentiating \(K_\lambda1\) in \(\lambda\), and using the Cauchy--Schwarz inequality,
we get
\begin{equation}
\label{3.6.10}
\|\partial_\lambda^jK_\lambda1\|_{L^\infty}
\le C_{j,N}\lambda^N
\qquad\text{as }\lambda\to0_{+}
\end{equation}
for every \(j,N\ge0\). Since
\[
a_\lambda-1=(I-K_\lambda)^{-1}K_\lambda1
\]
and \(\lambda\mapsto (I-K_\lambda)^{-1}\) is \(C^\infty\) near \(\lambda=0\), it follows
that
\begin{equation}
\label{3.6.11}
\partial_\lambda^k(a(\cdot,\lambda)-1)\big|_{\lambda=0}=0
\quad\text{in }L^\infty(\mathbb R),
\qquad k\ge0.
\end{equation}
We now prove that \(\mathcal F_{d,u}\) maps \(\mathcal S_+(\mathbb R)\) continuously into
\(\mathcal S_*\). Let \(f\in\mathcal S_+(\mathbb R)\). Then
\begin{equation}
\label{3.6.12}
(\mathcal F_{d,u}f)(\lambda)
=
\widetilde f(\lambda)
=
\int_{\mathbb R}
e^{-i\lambda x}\overline{a(x,\lambda)}f(x)\,dx,
\qquad \lambda>0.
\end{equation}
By \eqref{3.6.3}, differentiation under the integral sign is justified. For \(k\ge0\),
\begin{equation}
\label{3.6.13}
\partial_\lambda^k\widetilde f(\lambda)
=
\sum_{\ell=0}^k
c_{k,\ell}
\int_{\mathbb R}
e^{-i\lambda x}
(-ix)^{k-\ell}
\partial_\lambda^\ell\overline{a(x,\lambda)}
f(x)\,dx .
\end{equation}
Consequently, for every \(g\in\mathcal S_+(\mathbb R)\),
\begin{equation}
\label{3.6.14}
\sup_{\lambda\ge0}
|\partial_\lambda^k\mathcal F_{d,u}g(\lambda)|
\le
C_k
\sum_{j+\ell\le k}
\|x^jg\|_{L^1}
\sup_{\lambda\ge0}
\|\partial_\lambda^\ell a(\cdot,\lambda)\|_{L^\infty}.
\end{equation}
Thus all \(\lambda\)-derivatives of \(\mathcal F_{d,u}g\) are bounded.

We next prove rapid decay as \(\lambda\to+\infty\). We first note that
\[
L_u=D+uT_{\overline u}
\]
maps \(\mathcal S_+(\mathbb R)\) continuously into itself. The derivative part is clear.
For the perturbation, if \(h\in\mathcal S_+(\mathbb R)\), then
\(\overline u\,h\in\mathcal S(\mathbb R)\). Although the Riesz projection \(\Pi\) does
not in general map \(\mathcal S(\mathbb R)\) into itself, it maps
\(\mathcal S(\mathbb R)\) continuously into \(C^\infty\) functions whose derivatives have
at most polynomial growth. Multiplication by the Schwartz function \(u\) therefore gives
\[
u\,T_{\overline u}h
=
u\,\Pi(\overline u\,h)
\in\mathcal S(\mathbb R).
\]
Moreover \(T_{\overline u}h\in L^2_+(\mathbb R)\), and the product of two positive-frequency
functions has Fourier support in \([0,\infty)\). Hence
\[
uT_{\overline u}h\in\mathcal S_+(\mathbb R).
\]
Thus
\begin{equation}
\label{3.6.15}
L_u^Nf\in\mathcal S_+(\mathbb R),
\qquad N\ge0,
\end{equation}
and the map \(f\mapsto L_u^Nf\) is continuous on \(\mathcal S_+(\mathbb R)\).

By Remark \ref{remark 2.14}, we can infer
\[
\mathcal F_{d,u}(L_ug)(\lambda)=\lambda\,\mathcal F_{d,u}g(\lambda),
\qquad g\in\mathcal S_+(\mathbb R).
\]
Iterating gives
\begin{equation}
\label{3.6.16}
\mathcal F_{d,u}(L_u^Ng)(\lambda)=\lambda^N\mathcal F_{d,u}g(\lambda),
\qquad N\ge0.
\end{equation}
Taking \(g=f\), using \eqref{3.6.14} with \(g=L_u^Nf\), and differentiating
\eqref{3.6.16} in \(\lambda\), we obtain rapid decay of every derivative of
\(\widetilde f\). More explicitly, for \(k=0\),
\[
\sup_{\lambda\ge0}
\lambda^N|\widetilde f(\lambda)|
=
\sup_{\lambda\ge0}
|\mathcal F_{d,u}(L_u^Nf)(\lambda)|
<\infty .
\]
For \(k\ge1\), assume by induction that all derivatives of order \(<k\) are rapidly
decreasing. Differentiating
\[
\lambda^N\widetilde f(\lambda)=\mathcal F_{d,u}(L_u^Nf)(\lambda)
\]
\(k\) times gives
\[
\lambda^N\partial_\lambda^k\widetilde f(\lambda)
=
\partial_\lambda^k\mathcal F_{d,u}(L_u^Nf)(\lambda)
-
\sum_{r=0}^{k-1}
c_{N,k,r}\lambda^{N-k+r}\partial_\lambda^r\widetilde f(\lambda).
\]
The first term is bounded by \eqref{3.6.14}. Taking \(N\ge k\), the remaining powers of
\(\lambda\) are nonnegative, and the remaining terms are bounded by the induction
hypothesis. Since \(N\) can be chosen arbitrarily large, we obtain
\begin{equation}
\label{3.6.17}
\sup_{\lambda>0}
\langle\lambda\rangle^M
|\partial_\lambda^k\widetilde f(\lambda)|
<\infty
\qquad\text{for all }M,k\ge0.
\end{equation}

It remains to check the behavior at \(\lambda=0\). Since \(f\in\mathcal S_+(\mathbb R)\),
its ordinary Fourier transform belongs to \(\mathcal S_*\), and hence
\begin{equation}
\label{3.6.18}
\partial_\lambda^k\widehat f(0)=0,
\qquad k\ge0.
\end{equation}
Letting \(\lambda\to0_{+}\) in \eqref{3.6.13}, and using \eqref{3.6.11}, gives
\begin{equation}
\label{3.6.19}
\partial_\lambda^k\widetilde f(0_{+})
=
\int_{\mathbb R}(-ix)^kf(x)\,dx
=
\partial_\lambda^k\widehat f(0)
=
0 .
\end{equation}
Therefore the extension of \(\widetilde f\) by zero to \((-\infty,0]\) is a Schwartz
function supported in \([0,\infty)\). Hence
\[
\mathcal F_{d,u}f\in\mathcal S_*.
\]
The estimates above also show that
\[
\mathcal F_{d,u}:\mathcal S_+(\mathbb R)\to\mathcal S_*
\]
is continuous.

We now prove surjectivity on \(\mathcal S_*\) and continuity of the inverse. Let
\(\varphi\in\mathcal S_*\), and define
\begin{equation}
\label{3.6.20}
G\varphi(x)
:=
\frac1{2\pi}
\int_0^\infty e^{i\lambda x}a(x,\lambda)\varphi(\lambda)\,d\lambda
=
\frac1{2\pi}
\int_0^\infty m_-(x,\lambda)\varphi(\lambda)\,d\lambda .
\end{equation}
We first show that \(G\varphi\in\mathcal S(\mathbb R)\). Differentiating in \(x\), we get
\begin{equation}
\label{3.6.21}
\partial_x^nG\varphi(x)
=
\frac1{2\pi}
\sum_{j=0}^n
\binom nj
\int_0^\infty
(i\lambda)^j e^{i\lambda x}
\partial_x^{n-j}a(x,\lambda)\varphi(\lambda)\,d\lambda .
\end{equation}
The integral is absolutely convergent because \(\varphi\) is rapidly decreasing and
\eqref{3.6.4} gives uniform-in-\(\lambda\) bounds for the \(x\)-derivatives of \(a\).

To gain powers of \(x\), use
\[
x^M e^{i\lambda x}=i^{-M}\partial_\lambda^M(e^{i\lambda x})
\]
and integrate by parts \(M\) times in \(\lambda\). There are no boundary terms: at
\(\lambda=+\infty\) this follows from the rapid decay of \(\varphi\), while at
\(\lambda=0\) it follows from the fact that \(\varphi\in\mathcal S_*\), so every
derivative of \(\varphi\) vanishes at \(0\). Thus
\[
\begin{aligned}
|x|^M|\partial_x^nG\varphi(x)|
&\le
C_{M,n}
\sum_{j=0}^n
\int_0^\infty
\left|
\partial_\lambda^M
\left[
\lambda^j\partial_x^{n-j}a(x,\lambda)\varphi(\lambda)
\right]
\right|\,d\lambda .
\end{aligned}
\]
Using \eqref{3.6.4} and the rapid decay of \(\varphi\), we obtain
\begin{equation}
\label{3.6.22}
\sup_{x\in\mathbb R}
|x|^M|\partial_x^nG\varphi(x)|
\le
C_{M,n}
\sum_{r+s\le N_{M,n}}
\sup_{\lambda>0}
\langle\lambda\rangle^r
|\partial_\lambda^s\varphi(\lambda)|.
\end{equation}
Therefore \(G\varphi\in\mathcal S(\mathbb R)\), and
\[
G:\mathcal S_*\to\mathcal S(\mathbb R)
\]
is continuous.

We next show that \(G\varphi\in L^2_+(\mathbb R)\). By Proposition
\ref{prop:generalized_eigenfunctions}, for every \(\lambda>0\),
\[
m_-(\cdot,\lambda)\in L^\infty_+(\mathbb R).
\]
Hence, if \(q\in\mathcal S(\mathbb R)\) has Fourier support contained in
\((-\infty,0)\), then
\[
\int_{\mathbb R}m_-(x,\lambda)\overline{q(x)}\,dx=0
\qquad\text{for every }\lambda>0.
\]
Using the absolute convergence of the integral defining \(G\varphi\), we obtain
\[
\int_{\mathbb R}G\varphi(x)\overline{q(x)}\,dx=0.
\]
Thus the Fourier transform of \(G\varphi\) is supported in \([0,\infty)\). Since
\(G\varphi\in\mathcal S(\mathbb R)\), we have
\begin{equation}
\label{3.6.23}
G\varphi\in\mathcal S_+(\mathbb R).
\end{equation}

It remains to identify \(G\) with the inverse distorted Fourier transform. By Proposition
\ref{prop 2.6} and Lemma \ref{lemma 2.15}, the normalized transform
\[
\mathcal U:=\frac1{\sqrt{2\pi}}\mathcal F_{d,u}
\]
is unitary from \(L^2_+(\mathbb R)\) onto \(L^2(0,\infty)\). Therefore
\(\mathcal F_{d,u}^{-1}=(2\pi)^{-1}\mathcal F_{d,u}^*\) in the \(L^2\)-sense.

For \(h\in\mathcal S_+(\mathbb R)\), Fubini's theorem gives
\begin{equation}
\label{3.6.24}
\begin{aligned}
\langle h,G\varphi\rangle_{L^2(\mathbb R)}
&=
\frac1{2\pi}
\int_0^\infty
\left(
\int_{\mathbb R}h(x)\overline{m_-(x,\lambda)}\,dx
\right)
\overline{\varphi(\lambda)}\,d\lambda \\
&=
\frac1{2\pi}
\int_0^\infty
(\mathcal F_{d,u}h)(\lambda)\overline{\varphi(\lambda)}\,d\lambda .
\end{aligned}
\end{equation}
The right-hand side is precisely
\[
\langle h,\mathcal F_{d,u}^{-1}\varphi\rangle_{L^2(\mathbb R)}.
\]
Since \(\mathcal S_+(\mathbb R)\) is dense in \(L^2_+(\mathbb R)\), and since both
\(G\varphi\) and \(\mathcal F_{d,u}^{-1}\varphi\) belong to \(L^2_+(\mathbb R)\), we conclude
that
\begin{equation}
\label{3.6.25}
G\varphi=\mathcal F_{d,u}^{-1}\varphi.
\end{equation}
In particular,
\[
\mathcal F_{d,u}(G\varphi)=\varphi
\quad\text{in }L^2(0,\infty).
\]
Since both sides are smooth functions on \((0,\infty)\), the equality holds pointwise.

Finally, for \(f\in\mathcal S_+(\mathbb R)\), the \(L^2\)-unitarity gives
\[
G(\mathcal F_{d,u}f)=f.
\]
Therefore
\[
\mathcal F_{d,u}:\mathcal S_+(\mathbb R)\to\mathcal S_*
\]
is a bijection. The continuity of \(\mathcal F_{d,u}\) was proved above, and the continuity of
\(G:\mathcal S_*\to\mathcal S_+(\mathbb R)\) follows from \eqref{3.6.22}. This completes
the proof.
\end{proof}
Since \(\mathcal F_{d,u}\) is an isomorphism between the test-function spaces
\(\mathcal S_+\) and \(\mathcal S_*\), it has a canonical transpose acting on the
corresponding distribution spaces. This distributional extension will be used when the
objects under consideration are not in \(L^2\).
\begin{corollary}
\label{corollary 3.7}
Let \(u\in\mathcal S_+(\mathbb R)\). The distorted Fourier transform
\[
\mathcal F_{d,u}:\mathcal S_+(\mathbb R)\to\mathcal S_*
\]
extends canonically to a continuous linear isomorphism
\[
\mathcal F_{d,u}:\mathcal S_+'(\mathbb R)\to\mathcal S_*'
\]
by
\[
\langle \mathcal F_{d,u} F,\phi\rangle_{\mathcal S_*',\mathcal S_*}
:=
2\pi\,
\langle F,\mathcal F_{d,u}^{-1}\phi\rangle_{\mathcal S_+',\mathcal S_+},
\qquad
F\in\mathcal S_+'(\mathbb R),\quad \phi\in\mathcal S_*.
\tag{3.7.1}
\]
Here the duality pairings are taken with the same sesquilinear convention as the
\(L^2\)-inner product.
\end{corollary}

\begin{proof}
By Proposition \ref{proposition 3.6},
\[
\mathcal F_{d,u}:\mathcal S_+(\mathbb R)\to\mathcal S_*
\]
is a continuous bijection with continuous inverse
\[
\mathcal F_{d,u}^{-1}:\mathcal S_*\to\mathcal S_+(\mathbb R).
\]
Thus, for every \(F\in\mathcal S_+'(\mathbb R)\), the formula
\[
\phi\mapsto
2\pi\,
\langle F,\mathcal F_{d,u}^{-1}\phi\rangle_{\mathcal S_+',\mathcal S_+}
\]
defines a continuous linear functional on \(\mathcal S_*\). Hence
\[
\mathcal F_{d,u}F\in\mathcal S_*',
\]
and \((3.7.1)\) defines a continuous linear map
\[
\mathcal F_{d,u}:\mathcal S_+'(\mathbb R)\to\mathcal S_*'.
\]
This map is an isomorphism. Indeed, its inverse is given by
\[
\langle \mathcal F_{d,u}^{-1}G,\psi\rangle_{\mathcal S_+',\mathcal S_+}
:=
\frac1{2\pi}
\langle G,\mathcal F_{d,u}\psi\rangle_{\mathcal S_*',\mathcal S_*},
\qquad
G\in\mathcal S_*',\quad \psi\in\mathcal S_+(\mathbb R).
\tag{3.7.2}
\]
The continuity of \((3.7.2)\) follows from the continuity of
\[
\mathcal F_{d,u}:\mathcal S_+(\mathbb R)\to\mathcal S_*.
\]
A direct computation shows that the maps defined by \((3.7.1)\) and \((3.7.2)\) are
inverses of each other.

Finally, this extension agrees with the original distorted Fourier transform on regular
distributions. Indeed, if \(f\in\mathcal S_+(\mathbb R)\) and \(\phi\in\mathcal S_*\),
then, using the inverse formula from Proposition \ref{proposition 3.6},
\[
\mathcal F_{d,u}^{-1}\phi(x)
=
\frac1{2\pi}
\int_0^\infty m_-(x,\lambda)\phi(\lambda)\,d\lambda,
\]
we obtain
\[
\begin{aligned}
2\pi\,
\langle f,\mathcal F_{d,u}^{-1}\phi\rangle_{L^2(\mathbb R)}
&=
2\pi
\int_{\mathbb R}
f(x)\overline{\mathcal F_{d,u}^{-1}\phi(x)}\,dx  \\
&=
\int_0^\infty
\left(
\int_{\mathbb R}f(x)\overline{m_-(x,\lambda)}\,dx
\right)
\overline{\phi(\lambda)}\,d\lambda  \\
&=
\int_0^\infty
(\mathcal F_{d,u}f)(\lambda)\overline{\phi(\lambda)}\,d\lambda  \\
&=
\langle \mathcal F_{d,u}f,\phi\rangle_{L^2(0,\infty)}.
\end{aligned}
\]
Thus the distributional definition extends the classical one. This proves the corollary.
\end{proof}
\noindent
Finally, we need a stability result with respect to the potential. This will allow us to
derive identities first for smooth data and then pass to rough initial data by
approximation.
\begin{proposition}[Continuity of the distorted Fourier transform with respect to the potential]
\label{prop:continuity_Fd}
Let
\[
u_0\in L^2_+(\mathbb R),
\]
and let \((u_n)_{n\ge1}\subset \mathcal S_+(\mathbb R)\) satisfy
\[
u_n\to u_0
\qquad\text{in }L^2_+(\mathbb R).
\]
For \(n\ge0\), with the convention that \(u_0\) is the limiting potential, let
\[
m_n(x,\lambda):=m_-(x,\lambda;u_n),
\qquad
\mathcal F_n:=\mathcal F_{d,u_n}.
\]
Then the following statements hold.

\begin{enumerate}
\item We have
\[
\sup_{\lambda>0}
\|m_n(\cdot,\lambda)-m_0(\cdot,\lambda)\|_{L^\infty(\mathbb R)}
\longrightarrow0.
\]

\item For every \(f\in L^1(\mathbb R)\cap L^2_+(\mathbb R)\),
\[
\sup_{\lambda>0}
|(\mathcal F_n f)(\lambda)-(\mathcal F_0 f)(\lambda)|
\longrightarrow0.
\]

\item For every \(f\in L^2_+(\mathbb R)\),
\[
\mathcal F_n f\to \mathcal F_0f
\qquad\text{in }L^2(0,\infty).
\]

\item If
\[
\widetilde u_n:=\mathcal F_nu_n,
\qquad
\widetilde u_0:=\mathcal F_0u_0,
\]
then
\[
\widetilde u_n\to \widetilde u_0
\qquad\text{in }L^2(0,\infty).
\]

\end{enumerate}
\end{proposition}

\begin{proof}
We divide the proof into several steps.

\medskip
\noindent\textbf{Step 1: continuity of the Fredholm operators and uniform invertibility.}
For \(u\in L^2_+(\mathbb R)\) and \(\lambda>0\), define
\[
K_{\lambda,u}:L^\infty(\mathbb R)\to L^\infty(\mathbb R)
\]
by
\[
K_{\lambda,u}a(x)
:=
-i\int_{-\infty}^x
e^{-i\lambda y}u(y)
\Pi\bigl(\overline u\,e^{i\lambda\cdot}a\bigr)(y)\,dy .
\]
For the potential \(u_n\), write
\[
K_{\lambda,n}:=K_{\lambda,u_n}.
\]
The generalized eigenfunction has the form
\[
m_n(x,\lambda)=e^{i\lambda x}a_n(x,\lambda),
\]
where
\[
(I-K_{\lambda,n})a_n(\cdot,\lambda)=1.
\]
We first prove that
\begin{equation}
\label{2.9.1}
\delta_n:=
\sup_{\lambda>0}
\|K_{\lambda,n}-K_{\lambda,0}\|_{\mathcal L(L^\infty,L^\infty)}
\longrightarrow0 .
\end{equation}
Indeed, let \(\|a\|_{L^\infty}\le1\). Then
\[
\begin{aligned}
(K_{\lambda,n}-K_{\lambda,0})a(x)
&=
-i\int_{-\infty}^x
e^{-i\lambda y}(u_n-u_0)(y)
\Pi\bigl(\overline{u_n}\,e^{i\lambda\cdot}a\bigr)(y)\,dy  \\
&\quad
-i\int_{-\infty}^x
e^{-i\lambda y}u_0(y)
\Pi\bigl((\overline{u_n}-\overline{u_0})e^{i\lambda\cdot}a\bigr)(y)\,dy .
\end{aligned}
\]
By the Cauchy--Schwarz inequality and the \(L^2\)-boundedness of \(\Pi\),
\[
\begin{aligned}
\|(K_{\lambda,n}-K_{\lambda,0})a\|_{L^\infty}
&\le
\|u_n-u_0\|_{L^2}
\|\Pi(\overline{u_n}e^{i\lambda\cdot}a)\|_{L^2}  \\
&\quad
+\|u_0\|_{L^2}
\|\Pi((\overline{u_n}-\overline{u_0})e^{i\lambda\cdot}a)\|_{L^2}  \\
&\le
\|u_n-u_0\|_{L^2}
\bigl(\|u_n\|_{L^2}+\|u_0\|_{L^2}\bigr).
\end{aligned}
\]
The right-hand side is independent of \(\lambda\) and tends to \(0\). This proves
\eqref{2.9.1}.

By Proposition \ref{prop:uniform_m}, more precisely by \eqref{2.3}, we have
\begin{equation}
\label{2.9.2}
M_0:=
\sup_{\lambda>0}
\|(I-K_{\lambda,0})^{-1}\|_{\mathcal L(L^\infty,L^\infty)}
<\infty .
\end{equation}
Choose \(N\) so large that
\begin{equation}
\label{2.9.3}
\delta_nM_0\le \frac12
\qquad\text{for all }n\ge N.
\end{equation}
For \(n\ge N\), we factor
\[
I-K_{\lambda,n}
=
\Bigl[
I-(K_{\lambda,n}-K_{\lambda,0})(I-K_{\lambda,0})^{-1}
\Bigr](I-K_{\lambda,0}).
\]
By \eqref{2.9.2} and \eqref{2.9.3}, the first factor is invertible by the Neumann
series, uniformly in \(\lambda>0\). Hence
\begin{equation}
\label{2.9.4}
\sup_{n\ge N}\sup_{\lambda>0}
\|(I-K_{\lambda,n})^{-1}\|_{\mathcal L(L^\infty,L^\infty)}
\le 2M_0 .
\end{equation}
Now the second resolvent identity gives
\[
(I-K_{\lambda,n})^{-1}-(I-K_{\lambda,0})^{-1}
=
(I-K_{\lambda,n})^{-1}
(K_{\lambda,n}-K_{\lambda,0})
(I-K_{\lambda,0})^{-1}.
\]
Applying this identity to \(1\), we get
\[
a_n(\cdot,\lambda)-a_0(\cdot,\lambda)
=
(I-K_{\lambda,n})^{-1}
(K_{\lambda,n}-K_{\lambda,0})
a_0(\cdot,\lambda).
\]
Therefore, by \eqref{2.9.1}, \eqref{2.9.2}, and \eqref{2.9.4},
\begin{equation}
\label{2.9.5}
\begin{aligned}
\sup_{\lambda>0}
\|a_n(\cdot,\lambda)-a_0(\cdot,\lambda)\|_{L^\infty}
&\le
2M_0\,\delta_n\,
\sup_{\lambda>0}\|a_0(\cdot,\lambda)\|_{L^\infty}
\longrightarrow0 .
\end{aligned}
\end{equation}
Since
\[
m_n(x,\lambda)-m_0(x,\lambda)
=
e^{i\lambda x}
\bigl(a_n(x,\lambda)-a_0(x,\lambda)\bigr),
\]
we obtain
\[
\sup_{\lambda>0}
\|m_n(\cdot,\lambda)-m_0(\cdot,\lambda)\|_{L^\infty}
\to0.
\]
This proves part \((1)\).

\medskip
\noindent\textbf{Step 2: convergence for \(L^1\cap L^2_+\) data.}
Let \(f\in L^1(\mathbb R)\cap L^2_+(\mathbb R)\). Then
\[
\begin{aligned}
|(\mathcal F_nf)(\lambda)-(\mathcal F_0f)(\lambda)|
&=
\left|
\int_{\mathbb R}
f(x)\bigl(\overline{m_n(x,\lambda)}-\overline{m_0(x,\lambda)}\bigr)\,dx
\right|  \\
&\le
\|f\|_{L^1}
\|m_n(\cdot,\lambda)-m_0(\cdot,\lambda)\|_{L^\infty}.
\end{aligned}
\]
Taking the supremum over \(\lambda>0\) and using part \((1)\), we obtain
\[
\sup_{\lambda>0}
|(\mathcal F_nf)(\lambda)-(\mathcal F_0f)(\lambda)|
\to0.
\]
This proves part \((2)\).

\medskip
\noindent\textbf{Step 3: strong convergence in \(L^2(0,\infty)\).}
Fix \(f\in L^2_+(\mathbb R)\). We first prove local \(L^2\)-convergence. Let
\(0<\delta<R<\infty\), and let \(\varepsilon>0\). Since
\(L^1(\mathbb R)\cap L^2_+(\mathbb R)\) is dense in \(L^2_+(\mathbb R)\), choose
\(g\in L^1(\mathbb R)\cap L^2_+(\mathbb R)\) such that
\[
\|f-g\|_{L^2}<\varepsilon.
\]
By the \(L^2\)-isometry of the distorted Fourier transform,
\[
\|\mathcal F_n(f-g)\|_{L^2(0,\infty)}
=
\sqrt{2\pi}\|f-g\|_{L^2}
<\sqrt{2\pi}\varepsilon,
\]
and the same estimate holds for \(\mathcal F_0(f-g)\). Hence
\[
\begin{aligned}
\|\mathcal F_nf-\mathcal F_0f\|_{L^2([\delta,R])}
&\le
\|\mathcal F_n(f-g)\|_{L^2([\delta,R])}
+
\|\mathcal F_ng-\mathcal F_0g\|_{L^2([\delta,R])}  \\
&\quad
+\|\mathcal F_0(g-f)\|_{L^2([\delta,R])}  \\
&\le
2\sqrt{2\pi}\varepsilon
+
|R-\delta|^{1/2}
\sup_{\lambda\in[\delta,R]}
|(\mathcal F_ng)(\lambda)-(\mathcal F_0g)(\lambda)|.
\end{aligned}
\]
By part \((2)\), the last term tends to \(0\) as \(n\to\infty\). Therefore
\begin{equation}
\label{2.9.6}
\mathcal F_nf\to\mathcal F_0f
\qquad\text{in }L^2([\delta,R])
\end{equation}
for every compact interval \([\delta,R]\subset(0,\infty)\).

We now upgrade local convergence to global strong \(L^2\)-convergence. Set
\[
H_n:=\mathcal F_nf,
\qquad
H:=\mathcal F_0f.
\]
The isometry gives
\begin{equation}
\label{2.9.7}
\|H_n\|_{L^2(0,\infty)}
=
\|H\|_{L^2(0,\infty)}
=
\sqrt{2\pi}\|f\|_{L^2}.
\end{equation}
We claim that \(H_n\rightharpoonup H\) weakly in \(L^2(0,\infty)\). Indeed, let
\(\phi\in L^2(0,\infty)\). Choose \(\phi_c\in L^2(0,\infty)\) with compact support in
\((0,\infty)\) such that
\[
\|\phi-\phi_c\|_{L^2}<\varepsilon.
\]
Using \eqref{2.9.6} on a compact interval containing the support of \(\phi_c\), we obtain
\[
\langle H_n-H,\phi_c\rangle_{L^2(0,\infty)}\to0.
\]
Moreover, by \eqref{2.9.7},
\[
|\langle H_n-H,\phi-\phi_c\rangle|
\le
\bigl(\|H_n\|_{L^2}+\|H\|_{L^2}\bigr)\|\phi-\phi_c\|_{L^2}
\le C_f\varepsilon.
\]
Since \(\varepsilon>0\) is arbitrary, \(H_n\rightharpoonup H\) weakly in
\(L^2(0,\infty)\).

Finally, weak convergence together with equality of the norms in \eqref{2.9.7} gives
\[
\|H_n-H\|_{L^2(0,\infty)}\to0.
\]
Thus
\[
\mathcal F_nf\to \mathcal F_0f
\qquad\text{in }L^2(0,\infty).
\]
This proves part \((3)\).

\medskip
\noindent\textbf{Step 4: convergence of \(\widetilde u_n\) in \(L^2\).}
We write
\[
\widetilde u_n-\widetilde u_0
=
\mathcal F_n(u_n-u_0)+(\mathcal F_n-\mathcal F_0)u_0.
\]
For the first term, the isometry gives
\[
\|\mathcal F_n(u_n-u_0)\|_{L^2(0,\infty)}
=
\sqrt{2\pi}\|u_n-u_0\|_{L^2(\mathbb R)}
\to0 \quad \text{ as } n \to \infty.
\]
For the second term, part \((3)\), applied to \(f=u_0\), gives
\[
(\mathcal F_n-\mathcal F_0)u_0\to0
\qquad\text{in }L^2(0,\infty) \quad \text{ as } n \to \infty.
\]
Therefore
\[
\widetilde u_n\to\widetilde u_0
\qquad\text{in }L^2(0,\infty) \quad \text{ as } n \to \infty.
\]
This proves part \((4)\).

\end{proof}
\noindent
We close this section with the zero-energy normalization of the generalized
eigenfunction. This elementary limit will be used later when differentiating the
distorted Fourier representation at the spectral endpoint \(\lambda=0\).
\begin{lemma}
\label{lemma 2.5}
Let \(u\in L^{2}_+(\mathbb R)\). For
\(\lambda>0\), let
\[
m_-(x,\lambda)=e^{i\lambda x}a(x,\lambda)
\]
be the generalized eigenfunction normalized by
\[
\ell_{m_-}(-\infty)=1.
\]
Then, for every fixed \(x\in\mathbb R\),
\[
\lim_{\lambda\to0_{+}}m_-(x,\lambda)=1.
\]
\end{lemma}

\begin{proof}
Recall that \(a_\lambda=a(\cdot,\lambda)\) solves
\[
(I-K_\lambda)a_\lambda=1,
\]
where
\[
K_\lambda b(x)
:=
-i\int_{-\infty}^x
e^{-i\lambda y}u(y)
\Pi\bigl(\overline u\,e^{i\lambda\cdot}b\bigr)(y)\,dy .
\]
At \(\lambda=0\), define
\[
K_0 b(x)
:=
-i\int_{-\infty}^x
u(y)\Pi(\overline u\,b)(y)\,dy .
\]
We first prove that
\begin{equation}
\label{2.5.1}
\|K_\lambda-K_0\|_{\mathcal L(L^\infty,L^\infty)}
\longrightarrow0
\qquad\text{as }\lambda\to0_{+}.
\end{equation}
Indeed, let \(\|b\|_{L^\infty}\le1\). Then
\[
\begin{aligned}
(K_\lambda-K_0)b(x)
&=
-i\int_{-\infty}^x
\bigl(e^{-i\lambda y}-1\bigr)u(y)
\Pi\bigl(\overline u\,e^{i\lambda\cdot}b\bigr)(y)\,dy  \\
&\quad
-i\int_{-\infty}^x
u(y)
\Pi\bigl(\overline u\,(e^{i\lambda\cdot}-1)b\bigr)(y)\,dy .
\end{aligned}
\]
By the Cauchy--Schwarz inequality and the \(L^2\)-boundedness of \(\Pi\),
\[
\begin{aligned}
\|(K_\lambda-K_0)b\|_{L^\infty}
&\le
\|(e^{-i\lambda\cdot}-1)u\|_{L^2}
\left\|
\Pi\bigl(\overline u\,e^{i\lambda\cdot}b\bigr)
\right\|_{L^2}  
+\|u\|_{L^2}
\left\|
\Pi\bigl(\overline u\,(e^{i\lambda\cdot}-1)b\bigr)
\right\|_{L^2}  \\
&\le
\|(e^{-i\lambda\cdot}-1)u\|_{L^2}\|u\|_{L^2}
+
\|u\|_{L^2}\|(e^{i\lambda\cdot}-1)u\|_{L^2}.
\end{aligned}
\]
Since \(u\in L^2(\mathbb R)\), dominated convergence gives
\[
\|(e^{\pm i\lambda\cdot}-1)u\|_{L^2}
\to0
\qquad\text{as }\lambda\to0_{+}.
\]
Taking the supremum over all \(\|b\|_{L^\infty}\le1\), we obtain \eqref{2.5.1}.

By Remark \ref{remark 2.7}, the operator
\[
I-K_0:L^\infty(\mathbb R)\to L^\infty(\mathbb R)
\]
is invertible, and the unique solution of
\[
(I-K_0)a=1
\]
is
\[
a_0=1.
\]
Set
\[
R_0:=(I-K_0)^{-1}.
\]
From \eqref{2.5.1}, for all sufficiently small \(\lambda>0\),
\[
\|(K_\lambda-K_0)R_0\|_{\mathcal L(L^\infty,L^\infty)}
\le \frac12 .
\]
From \eqref{2.3} we infer
\begin{equation}
\label{2.5.2}
\sup_{\lambda>0}
\|(I-K_\lambda)^{-1}\|_{\mathcal L(L^\infty,L^\infty)}
<\infty .
\end{equation}
By the second resolvent identity,
\[
(I-K_\lambda)^{-1}-(I-K_0)^{-1}
=
(I-K_\lambda)^{-1}(K_\lambda-K_0)(I-K_0)^{-1}.
\]
Applying this identity to \(1\), and using
\[
a_\lambda=(I-K_\lambda)^{-1}1,
\qquad
a_0=(I-K_0)^{-1}1=1,
\]
we obtain
\[
a_\lambda-1
=
(I-K_\lambda)^{-1}(K_\lambda-K_0)1.
\]
Hence, by \eqref{2.5.1} and \eqref{2.5.2},
\begin{equation}
\label{2.5.3}
\|a_\lambda-1\|_{L^\infty}
\le
\|(I-K_\lambda)^{-1}\|_{\mathcal L(L^\infty,L^\infty)}
\|K_\lambda-K_0\|_{\mathcal L(L^\infty,L^\infty)}
\|1\|_{L^\infty}  
\longrightarrow0
\qquad\text{as }\lambda\to0_{+}.
\end{equation}
Finally, for each fixed \(x\in\mathbb R\),
\[
m_-(x,\lambda)-1
=
e^{i\lambda x}\bigl(a(x,\lambda)-1\bigr)
+
\bigl(e^{i\lambda x}-1\bigr).
\]
The first term tends to \(0\) by \eqref{2.5.3}, and the second term tends to \(0\)
because \(x\) is fixed. Therefore
\[
\lim_{\lambda\to0_{+}}m_-(x,\lambda)=1.
\]
This proves the lemma.
\end{proof}
The preceding results provide the spectral framework needed in the rest of the paper:
the generalized eigenfunctions are uniformly controlled, the distorted Fourier transform
is unitary and onto, and both its distributional extension and its stability under
approximation are available. We now turn to the explicit formula in these distorted
Fourier variables.
\section{Distorted Fourier representation of the explicit formula}
\label{section 2.3}

We have developed the distorted Fourier transform in the previous section. We now use it
to rewrite the explicit formula in distorted Fourier variables. The main point is to
identify the action of \(X^*\) after applying \(\mathcal F_{d,u}\). This produces a
spectral correction term involving the half-line Hilbert transform.

\subsection{The action of \(X^*\) in distorted Fourier variables}

The goal of this subsection is to prove a distorted Fourier representation for \(X^*\).
The first step is to understand $e^{i\lambda x}\partial_\lambda a(x,\lambda)$. This term is not square-integrable in \(x\), so we treat it as a
positive-frequency distribution and use the distributional extension of the distorted
Fourier transform.

We first analyze the \(\lambda\)-derivative of the generalized eigenfunction. Recall that
the problem
\[
(L_u-\lambda)m_-=0,
\qquad
\ell_{m_-}(-\infty)=1,
\]
is equivalent to writing
\[
m_-(x,\lambda)=e^{i\lambda x}a(x,\lambda),
\]
where
\[
a_\lambda:=a(\cdot,\lambda)
\]
solves
\[
a_\lambda=1+K_\lambda a_\lambda,
\]
with
\[
K_\lambda b(x)
:=
-i\int_{-\infty}^x
e^{-i\lambda y}u(y)
\Pi\bigl(\overline u\,e^{i\lambda\cdot}b\bigr)(y)\,dy .
\]
In this subsection we assume
\[
u\in\mathcal S_+(\mathbb R).
\]
Then the map
\[
\lambda\mapsto a_\lambda
\]
is \(C^\infty\) as an \(L^\infty(\mathbb R)\)-valued function. Differentiating
\[
(I-K_\lambda)a_\lambda=1
\]
with respect to \(\lambda\), we obtain
\[
(I-K_\lambda)\partial_\lambda a_\lambda
=
(\partial_\lambda K_\lambda)a_\lambda .
\]
We now compute the right-hand side. Set
\[
g_\lambda(x):=\overline u(x)e^{i\lambda x}a(x,\lambda).
\]
Using the commutator identity
\[
\Pi(xg)-x\Pi g
=
\frac1{2i\pi}\widehat g(0),
\]
we get
\[
(\partial_\lambda K_\lambda)a_\lambda(x)
=
\frac1{2i\pi}\widehat g_\lambda(0)
\int_{-\infty}^x e^{-i\lambda y}u(y)\,dy .
\]
Moreover,
\[
\widehat g_\lambda(0)
=
\int_{\mathbb R}
\overline u(y)e^{i\lambda y}a(y,\lambda)\,dy
=
\overline{\widetilde u(\lambda)}.
\]
Therefore
\begin{equation}
\label{eq:dlambda-a-equation}
\partial_\lambda a(x,\lambda)
=
\frac1{2i\pi}
\overline{\widetilde u(\lambda)}
\int_{-\infty}^x e^{-i\lambda y}u(y)\,dy
+
K_\lambda(\partial_\lambda a_\lambda)(x).
\end{equation}
Define
\[
n(x,\lambda):=e^{i\lambda x}\partial_\lambda a(x,\lambda).
\]
Then \(n(\cdot,\lambda)\in L^\infty_+(\mathbb R)\). Indeed, after multiplying
\eqref{eq:dlambda-a-equation} by \(e^{i\lambda x}\), every term is obtained by applying
the retarded free resolvent kernel
\[
F\mapsto \int_{-\infty}^x e^{i\lambda(x-y)}F(y)\,dy
\]
to a positive-frequency source. The first source is \(u\in\mathcal S_+(\mathbb R)\), and
the second source is
\[
u\,\Pi\bigl(\overline u\,e^{i\lambda\cdot}\partial_\lambda a_\lambda\bigr),
\]
which has nonnegative Fourier support. Hence \(n(\cdot,\lambda)\) has nonnegative Fourier
support. The same argument used in Proposition \ref{prop:generalized_eigenfunctions} for \(m\) gives boundedness.

Moreover, \eqref{eq:dlambda-a-equation} is exactly the equation associated with
the inhomogeneous eigenvalue equation
\begin{equation}
\label{eq:n-inhomogeneous-equation}
(L_u-\lambda)n(\cdot,\lambda)
=
-\frac1{2\pi}
\overline{\widetilde u(\lambda)}\,u
\end{equation}
in the distributional sense.

Since
\[
n(\cdot,\lambda)\in L^\infty_+(\mathbb R)
\subset \mathcal S_+'(\mathbb R),
\]
Corollary \ref{corollary 3.7} allows us to define the distorted Fourier transform of
\(n(\cdot,\lambda)\) as a distribution in \(\mathcal S_*'\). We denote it by
\[
\widetilde n(\mu,\lambda):=\mathcal F_{d,u}[n(\cdot,\lambda)](\mu).
\]
Applying the distorted Fourier transform to \eqref{eq:n-inhomogeneous-equation} in the
distributional sense, and using
\[
\mathcal F_{d,u}(L_uF)(\mu)
=
\mu\,\mathcal F_{d,u}F(\mu),
\qquad F\in\mathcal S_+'(\mathbb R),
\]
we obtain
\begin{equation}
\label{eq:n-spectral-equation}
(\mu-\lambda)\widetilde n(\mu,\lambda)
=
-\frac1{2\pi}
\widetilde u(\mu)\overline{\widetilde u(\lambda)}.
\end{equation}
Hence, as a distribution in the spectral variable \(\mu\),
\begin{equation}
\label{eq:n-spectral-distribution}
\widetilde n(\mu,\lambda)
=
-\frac1{2\pi}
\operatorname{p.v.}
\frac{\widetilde u(\mu)\overline{\widetilde u(\lambda)}}{\mu-\lambda}
+
2\pi c(\lambda)\delta(\mu-\lambda),
\end{equation}
where \(c(\lambda)\) is to be determined from the normalization at \(x=-\infty\).

Equivalently, by the inverse distorted Fourier transform,
\begin{equation}
\label{3.51}
n(x,\lambda)
=
n_1(x,\lambda)+c(\lambda)m_-(x,\lambda),
\end{equation}
where
\begin{equation}
\label{eq:n1-definition}
n_1(x,\lambda)
:=
-\frac1{4\pi^2}
\operatorname{p.v.}\int_0^\infty
\frac{\widetilde u(\mu)\overline{\widetilde u(\lambda)}}{\mu-\lambda}
m_-(x,\mu)\,d\mu .
\end{equation}
The scalar \(c(\lambda)\) will be determined by the asymptotic behavior at
\(x=-\infty\).

It remains to determine the scalar \(c(\lambda)\). Since the normalization of
\(n(\cdot,\lambda)\) is imposed at \(x=-\infty\), it is enough to understand the
asymptotic behavior of \(e^{-i\lambda x}n_1(x,\lambda)\) as \(x\to-\infty\). This is the content of the next lemma.
\begin{lemma}
\label{lemma 3.8}
Let \(u\in\mathcal S_+(\mathbb R)\). For \(\lambda>0\), define
\[
n_1(x,\lambda)
:=
-\frac{1}{4\pi^2}\,
\mathrm{p.v.}\int_0^\infty
\frac{\widetilde u(\mu)\overline{\widetilde u(\lambda)}}{\mu-\lambda}
\,m_-(x,\mu)\,d\mu .
\]
Then, for every fixed \(\lambda>0\),
\[
e^{-i\lambda x}n_1(x,\lambda)
\longrightarrow
\frac{i}{4\pi}|\widetilde u(\lambda)|^2
\qquad\text{as }x\to-\infty .
\]
\end{lemma}
\begin{proof}
Fix \(\lambda>0\). Write
\[
m_-(x,\mu)=e^{i\mu x}a(x,\mu),
\qquad
b(x,\mu):=a(x,\mu)-1.
\]
Then
\[
m_-(x,\mu)=e^{i\mu x}\bigl(1+b(x,\mu)\bigr),
\]
and hence
\[
e^{-i\lambda x}n_1(x,\lambda)
=
I_1(x,\lambda)+I_2(x,\lambda),
\]
where
\[
I_1(x,\lambda)
:=
-\frac{e^{-i\lambda x}}{4\pi^2}
\,\mathrm{p.v.}\int_0^\infty
\frac{\widetilde u(\mu)\overline{\widetilde u(\lambda)}}{\mu-\lambda}
e^{i\mu x}\,d\mu ,
\]
and
\[
I_2(x,\lambda)
:=
-\frac{e^{-i\lambda x}}{4\pi^2}
\,\mathrm{p.v.}\int_0^\infty
\frac{\widetilde u(\mu)\overline{\widetilde u(\lambda)}}{\mu-\lambda}
e^{i\mu x}b(x,\mu)\,d\mu .
\]
We first compute the limit of \(I_1\). Since \(u\in\mathcal S_+(\mathbb R)\), by
Proposition \ref{proposition 3.6} we have
\[
\widetilde u\in\mathcal S_* .
\]
Choose \(\chi\in C_c^\infty((0,\infty))\) such that
\[
\chi(\mu)=1
\]
in a neighborhood of \(\lambda\). Define
\[
g_\lambda(\mu)
:=
\frac{\widetilde u(\mu)-\widetilde u(\lambda)\chi(\mu)}{\mu-\lambda},
\]
where the value at \(\mu=\lambda\) is defined by continuity. Since \(\chi=1\) near
\(\lambda\) and \(\widetilde u\in\mathcal S_*\), we have
\[
g_\lambda\in L^1(0,\infty).
\]
Therefore, by the Riemann--Lebesgue lemma,
\begin{equation}
\label{3.8.1}
e^{-i\lambda x}
\int_0^\infty g_\lambda(\mu)e^{i\mu x}\,d\mu
\longrightarrow0
\qquad\text{as }x\to-\infty .
\end{equation}
It remains to evaluate the localized constant part. With the change of variables
\(s=\mu-\lambda\), we get
\[
e^{-i\lambda x}
\,\mathrm{p.v.}\int_0^\infty
\frac{\chi(\mu)e^{i\mu x}}{\mu-\lambda}\,d\mu
=
\mathrm{p.v.}\int_{-\lambda}^{\infty}
\frac{\chi(\lambda+s)e^{isx}}{s}\,ds .
\]
Choose \(\delta>0\) such that \(\delta<\lambda\) and
\[
\chi(\lambda+s)=1
\qquad\text{for } |s|<\delta .
\]
Then
\[
\mathrm{p.v.}\int_{-\lambda}^{\infty}
\frac{\chi(\lambda+s)e^{isx}}{s}\,ds
=
\mathrm{p.v.}\int_{-\delta}^{\delta}
\frac{e^{isx}}{s}\,ds
+
\int_{\{|s|\ge\delta\}\cap[-\lambda,\infty)}
\frac{\chi(\lambda+s)e^{isx}}{s}\,ds .
\]
The second term tends to \(0\) as \(x\to-\infty\) by the Riemann--Lebesgue lemma. For the
first term,
\[
\mathrm{p.v.}\int_{-\delta}^{\delta}
\frac{e^{isx}}{s}\,ds
=
2i\int_0^\delta\frac{\sin(sx)}{s}\,ds
\longrightarrow -i\pi
\qquad\text{as }x\to-\infty .
\]
Hence
\begin{equation}
\label{3.8.2}
\mathrm{p.v.}\int_{-\lambda}^{\infty}
\frac{\chi(\lambda+s)e^{isx}}{s}\,ds
\longrightarrow -i\pi
\qquad\text{as }x\to-\infty .
\end{equation}
Combining \eqref{3.8.1} and \eqref{3.8.2}, we obtain
\begin{equation}
\label{3.8.3}
I_1(x,\lambda)
\longrightarrow
-\frac{1}{4\pi^2}|\widetilde u(\lambda)|^2(-i\pi)
=
\frac{i}{4\pi}|\widetilde u(\lambda)|^2
\qquad\text{as }x\to-\infty .
\end{equation}
It remains to show that
\begin{equation}
\label{3.8.4}
I_2(x,\lambda)\to0
\qquad\text{as }x\to-\infty .
\end{equation}
We first record two tail estimates for \(b(x,\mu)\). Since
\[
b(x,\mu)=K_\mu a_\mu(x)
=
-i\int_{-\infty}^x
e^{-i\mu y}u(y)
\Pi\bigl(\overline u\,e^{i\mu\cdot}a_\mu\bigr)(y)\,dy ,
\]
Proposition \ref{prop:uniform_m} and the Cauchy--Schwarz inequality give
\[
\begin{aligned}
|b(x,\mu)|
&\le
\left(\int_{-\infty}^x |u(y)|^2\,dy\right)^{1/2}
\left\|\Pi\bigl(\overline u\,e^{i\mu\cdot}a_\mu\bigr)\right\|_{L^2} \\
&\le
C_u
\left(\int_{-\infty}^x |u(y)|^2\,dy\right)^{1/2}.
\end{aligned}
\]
Hence
\begin{equation}
\label{3.8.5}
\sup_{\mu\ge0}|b(x,\mu)|\to0
\qquad\text{as }x\to-\infty .
\end{equation}
We also need a Lipschitz estimate in \(\mu\) on compact intervals. We claim that
\begin{equation}
\label{3.8.6}
\sup_{0\le\mu\le 2\lambda}
|\partial_\mu b(x,\mu)|
\longrightarrow0
\qquad\text{as }x\to-\infty .
\end{equation}
Indeed, the operator-norm differentiability argument used in the proof of Proposition
\ref{proposition 3.6}, together with the uniform inverse bound from Proposition
\ref{prop:uniform_m}, gives
\begin{equation}
\label{3.8.7}
\sup_{0\le\mu\le2\lambda}
\|\partial_\mu a_\mu\|_{L^\infty}<\infty .
\end{equation}
Set
\[
P_\mu:=\Pi\bigl(\overline u\,e^{i\mu\cdot}a_\mu\bigr).
\]
Then
\[
b(x,\mu)
=
-i\int_{-\infty}^x e^{-i\mu y}u(y)P_\mu(y)\,dy.
\]
Differentiating in \(\mu\), we get
\[
\partial_\mu b(x,\mu)
=
-i\int_{-\infty}^x
(-iy)e^{-i\mu y}u(y)P_\mu(y)\,dy
-i\int_{-\infty}^x
e^{-i\mu y}u(y)\partial_\mu P_\mu(y)\,dy .
\]
By \eqref{3.8.7} and \(u\in\mathcal S_{+}(\mathbb R)\), and by the \(L^2\)-boundedness of
\(\Pi\), both
\[
\sup_{0\le\mu\le2\lambda}\|P_\mu\|_{L^2}
\quad\text{and}\quad
\sup_{0\le\mu\le2\lambda}\|\partial_\mu P_\mu\|_{L^2}
\]
are finite. Thus
\[
\begin{aligned}
|\partial_\mu b(x,\mu)|
&\le
\left(\int_{-\infty}^x |y u(y)|^2\,dy\right)^{1/2}
\|P_\mu\|_{L^2}  \\
&\quad+
\left(\int_{-\infty}^x |u(y)|^2\,dy\right)^{1/2}
\|\partial_\mu P_\mu\|_{L^2}.
\end{aligned}
\]
Taking the supremum over \(0\le\mu\le2\lambda\) and using \(u\in\mathcal S_{+}(\mathbb R)\),
we obtain \eqref{3.8.6}.

We now estimate \(I_2\). Split
\[
I_2=I_{2,\mathrm{near}}+I_{2,\mathrm{far}},
\]
where \(I_{2,\mathrm{near}}\) corresponds to \(0\le\mu\le2\lambda\), and
\(I_{2,\mathrm{far}}\) corresponds to \(\mu\ge2\lambda\).

For the near part, set
\[
H_x(\mu):=\widetilde u(\mu)b(x,\mu).
\]
Then
\begin{equation}
\label{3.8.8}
\begin{aligned}
I_{2,\mathrm{near}}(x,\lambda)
&=
-\frac{\overline{\widetilde u(\lambda)}}{4\pi^2}
\,\mathrm{p.v.}\int_0^{2\lambda}
\frac{e^{i(\mu-\lambda)x}H_x(\mu)}{\mu-\lambda}\,d\mu  \\
&=
-\frac{\overline{\widetilde u(\lambda)}}{4\pi^2}
H_x(\lambda)
\,\mathrm{p.v.}\int_0^{2\lambda}
\frac{e^{i(\mu-\lambda)x}}{\mu-\lambda}\,d\mu  \\
&\quad
-\frac{\overline{\widetilde u(\lambda)}}{4\pi^2}
\int_0^{2\lambda}
e^{i(\mu-\lambda)x}
\frac{H_x(\mu)-H_x(\lambda)}{\mu-\lambda}\,d\mu .
\end{aligned}
\end{equation}
The principal value integral in the first term is uniformly bounded in \(x\), since
\[
\mathrm{p.v.}\int_0^{2\lambda}
\frac{e^{i(\mu-\lambda)x}}{\mu-\lambda}\,d\mu
=
\mathrm{p.v.}\int_{-\lambda}^{\lambda}
\frac{e^{isx}}{s}\,ds
=
i\int_{-\lambda}^{\lambda}
\frac{\sin(sx)}{s}\,ds .
\]
By \eqref{3.8.5},
\[
|H_x(\lambda)|\le |\widetilde u(\lambda)|\,|b(x,\lambda)|\longrightarrow0 \qquad\text{as }x\to-\infty.
\]
For the second term in \eqref{3.8.8}, the mean value theorem gives
\[
\left|
\frac{H_x(\mu)-H_x(\lambda)}{\mu-\lambda}
\right|
\le
\sup_{0\le\nu\le2\lambda}|H_x'(\nu)|.
\]
Since
\[
H_x'(\nu)=\widetilde u'(\nu)b(x,\nu)+\widetilde u(\nu)\partial_\nu b(x,\nu),
\]
\eqref{3.8.5} and \eqref{3.8.6} imply
\[
\sup_{0\le\nu\le2\lambda}|H_x'(\nu)|\to0
\qquad\text{as }x\to-\infty .
\]
Therefore
\begin{equation}
\label{3.8.9}
I_{2,\mathrm{near}}(x,\lambda)\to0
\qquad\text{as }x\to-\infty .
\end{equation}
For the far part, using \eqref{3.8.5}, we have
\[
\begin{aligned}
|I_{2,\mathrm{far}}(x,\lambda)|
&\le
\frac{|\widetilde u(\lambda)|}{4\pi^2}
\sup_{\mu\ge0}|b(x,\mu)|
\int_{2\lambda}^{\infty}
\frac{|\widetilde u(\mu)|}{\mu-\lambda}\,d\mu .
\end{aligned}
\]
The integral is finite because \(\widetilde u\in\mathcal S_*\), and the supremum tends to
\(0\) by \eqref{3.8.5}. Thus
\begin{equation}
\label{3.8.10}
I_{2,\mathrm{far}}(x,\lambda)\to0
\qquad\text{as }x\to-\infty .
\end{equation}
Combining \eqref{3.8.9} and \eqref{3.8.10}, we obtain \eqref{3.8.4}.

Finally, \eqref{3.8.3} and \eqref{3.8.4} give
\[
e^{-i\lambda x}n_1(x,\lambda)
\longrightarrow
\frac{i}{4\pi}|\widetilde u(\lambda)|^2
\qquad\text{as }x\to-\infty .
\]
This proves the lemma.
\end{proof}

\noindent
Lemma \ref{lemma 3.8} fixes the normalization constant in the distributional spectral
formula for \(n\). We therefore obtain the following explicit expression for
\(n\).

\begin{corollary}
\label{corollary 2.13}
Let \(u\in\mathcal S_+(\mathbb R)\). For every \(\lambda>0\), we have
\begin{equation}
\label{3.9}
n(x,\lambda)
=
n_1(x,\lambda)
-
\frac{i}{4\pi}|\widetilde u(\lambda)|^2m_-(x,\lambda).
\end{equation}
\end{corollary}

\begin{proof}
Recall from \eqref{3.51} that
\[
n(x,\lambda)=n_1(x,\lambda)+c(\lambda)m_-(x,\lambda),
\]
where \(c(\lambda)\) is determined by the normalization at \(-\infty\). More precisely,
since
\[
m_-(x,\lambda)=e^{i\lambda x}a(x,\lambda),
\qquad
a(x,\lambda)\to1
\quad\text{as }x\to-\infty,
\]
and since the chosen normalization of \(n\) is
\[
e^{-i\lambda x}n(x,\lambda)\to0
\qquad\text{as }x\to-\infty,
\]
we get
\[
0
=
\lim_{x\to-\infty}e^{-i\lambda x}n(x,\lambda)
=
\lim_{x\to-\infty}e^{-i\lambda x}n_1(x,\lambda)+c(\lambda).
\]
By Lemma \ref{lemma 3.8},
\[
\lim_{x\to-\infty}e^{-i\lambda x}n_1(x,\lambda)
=
\frac{i}{4\pi}|\widetilde u(\lambda)|^2.
\]
Therefore
\[
c(\lambda)
=
-\frac{i}{4\pi}|\widetilde u(\lambda)|^2,
\]
which proves \eqref{3.9}.
\end{proof}
We now insert this formula into the inverse distorted Fourier representation. The
principal-value term gives rise to a half-line Hilbert transform in the spectral variable,
while the Lax operator term produces a multiplication operator. This yields the precise
spectral form of \(X^*\).
\begin{proposition}[The action of \(X^*\) under the distorted Fourier transform]
\label{prop:Xstar_under_Fd}
Let \(u\in\mathcal S_+(\mathbb R)\), and set
\[
\mathfrak u:=\widetilde u:=\mathcal F_{d,u}u.
\]
Let
\[
\widetilde f\in H^1(0,\infty),
\qquad
f:=\mathcal F_{d,u}^{-1}\widetilde f.
\]
Then
\[
f\in\operatorname{Dom}(X^*),
\qquad
I_+(f)=J_+(\widetilde f):=\widetilde f(0_{+}),
\]
and
\[
\mathcal F_{d,u}(X^*f)
=
i\partial_\lambda\widetilde f
+
B(\mathfrak u)\widetilde f,
\]
where
\[
B(\mathfrak u)g
:=
-\frac1{4\pi}
\left(
i\mathfrak u\,H(\overline{\mathfrak u}g)
-
|\mathfrak u|^2g
\right).
\]
Here \(H\) denotes the Hilbert transform on the positive half-line
\[
H h(\mu)
:=
\frac1\pi\operatorname{p.v.}
\int_0^\infty\frac{h(\lambda)}{\mu-\lambda}\,d\lambda,\quad \mu >0.
\]
Moreover, for every $t\in\mathbb R$, the closed realization of
$X^*+2tL_u$ appearing in Theorem \ref{theorem 1.3} is given by
\[
X^*+2tL_u
=e^{itD^2}\left(X^*+2t e^{-itD^2}uT_{\bar u}e^{itD^2}\right)e^{-itD^2},
\]
with
\[
\operatorname{Dom}(X^*+2tL_u)
=e^{itD^2}\operatorname{Dom}(X^*)
=\left\{f\in L^2_+(\mathbb R):e^{-it\xi^2}\widehat f\in H^1(0,\infty)\right\}.
\]
Under the distorted Fourier transform,
\[
\mathcal F_{d,u}(X^*+2tL_u)\mathcal F_{d,u}^{-1}
=i\partial_\lambda+B(\mathfrak u)+2t\lambda
\]
as closed operators, where
\[
\operatorname{Dom}(i\partial_\lambda+B(\mathfrak u)+2t\lambda)
=\left\{h\in L^2(0,\infty):e^{-it\lambda^2}h\in H^1(0,\infty)\right\}.
\]
For $h$ in this domain, the right trace exists and we retain the definition
\[
J_+(h):=h(0_+).
\]
Then
\[
\operatorname{Dom}(X^*+2tL_u)\subset\operatorname{Dom}(I_+),
\qquad
I_+(f)=J_+(\mathcal F_{d,u}f),
\quad f\in\operatorname{Dom}(X^*+2tL_u).
\]
\end{proposition}

\begin{proof}
We first prove the identity for
\[
\widetilde f\in C_c^\infty([0,\infty)),
\]
and then pass to general \(\widetilde f\in H^1(0,\infty)\) by density. Let
\[
f:=\mathcal F_{d,u}^{-1}\widetilde f,
\qquad
\chi_\varepsilon(\lambda):=e^{-\varepsilon\lambda^2}.
\]
All identities involving \(xf\) below are understood in the sense of tempered
distributions in the \(x\)-variable until the \(L^2\)-conclusion is obtained.

Using the inverse distorted Fourier formula for \(L^1\cap L^2\) spectral data, obtained
by density from Proposition \ref{proposition 3.6}, we have
\[
2\pi\mathcal F_{d,u}^{-1}(i\partial_\lambda\widetilde f)
=
i\lim_{\varepsilon\to0_{+}}
\int_0^\infty
\partial_\lambda\widetilde f(\lambda)
\chi_\varepsilon(\lambda)m_-(x,\lambda)\,d\lambda .
\]
Integrating by parts in \(\lambda\), and using Lemma \ref{lemma 2.5}, namely
\[
m_-(x,\lambda)\to1
\qquad\text{as }\lambda\to 0_+,
\]
we obtain
\begin{equation}
\label{eq:Xstar-ibp}
\begin{aligned}
2\pi\mathcal F_{d,u}^{-1}(i\partial_\lambda\widetilde f)
&=
-i\widetilde f(0_+)
-i\lim_{\varepsilon\to0_+}
\int_0^\infty
\widetilde f(\lambda)\chi_\varepsilon'(\lambda)m_-(x,\lambda)\,d\lambda  \\
&\quad
-i\lim_{\varepsilon\to0_+}
\int_0^\infty
\widetilde f(\lambda)\chi_\varepsilon(\lambda)
\partial_\lambda m_-(x,\lambda)\,d\lambda .
\end{aligned}
\end{equation}
Since
\[
\|\chi_\varepsilon'\widetilde f\|_{L^1(0,\infty)}\to0 \text{ as } \varepsilon \to 0_+,
\]
the middle term in \eqref{eq:Xstar-ibp} tends to \(0\). Moreover,
\[
\partial_\lambda m_-(x,\lambda)=ixm_-(x,\lambda)+n(x,\lambda).
\]
Therefore
\begin{equation}
\label{eq:Xstar-after-n}
\begin{aligned}
2\pi\mathcal F_{d,u}^{-1}(i\partial_\lambda\widetilde f)
=-i\widetilde f(0_+)
+
2\pi xf 
-i\lim_{\varepsilon\to0_+}
\int_0^\infty
\widetilde f(\lambda)\chi_\varepsilon(\lambda)n(x,\lambda)\,d\lambda .
\end{aligned}
\end{equation}
By Corollary \ref{corollary 2.13},
\[
n(x,\lambda)
=
n_1(x,\lambda)
-
\frac{i}{4\pi}|\mathfrak u(\lambda)|^2m_-(x,\lambda).
\]
The contribution of the second term is
\[
-\frac1{4\pi}
\mathcal F_{d,u}^{-1}\left(|\mathfrak u|^2\widetilde f\right),
\]
because \(\chi_\varepsilon|\mathfrak u|^2\widetilde f\to
|\mathfrak u|^2\widetilde f\) in \(L^2(0,\infty)\). It remains to justify and compute the
\(n_1\)-contribution.

We do not compute this term by applying \(\mathcal F_{d,u}\) directly to the singular
principal-value kernel. Instead, we first truncate the principal value. For
\(\delta>0\), define
\[
n_{1,\delta}(x,\lambda)
:=
-\frac1{4\pi^2}
\int_{\substack{0<\mu<\infty\\ |\mu-\lambda|>\delta}}
\frac{\mathfrak u(\mu)\overline{\mathfrak u(\lambda)}}{\mu-\lambda}
m_-(x,\mu)\,d\mu ,
\]
and set
\begin{equation}
\label{eq:Xstar-r3-eps-delta}
r_{3,\varepsilon,\delta}(x)
:=
-\frac{i}{2\pi}
\int_0^\infty
\widetilde f(\lambda)\chi_\varepsilon(\lambda)
n_{1,\delta}(x,\lambda)\,d\lambda .
\end{equation}
Since \(\widetilde f\) is compactly supported and \(\mathfrak u\in\mathcal S_*\), Fubini's
theorem is applicable for fixed \(\varepsilon,\delta>0\). Thus
\begin{equation}
\label{eq:Xstar-r3-eps-delta-inverse}
r_{3,\varepsilon,\delta}
=
\mathcal F_{d,u}^{-1}g_{\varepsilon,\delta},
\end{equation}
where
\begin{equation}
\label{eq:Xstar-g-eps-delta}
g_{\varepsilon,\delta}(\mu)
:=
\frac{i}{4\pi^2}\mathfrak u(\mu)
\int_{\substack{0<\lambda<\infty\\ |\mu-\lambda|>\delta}}
\frac{\overline{\mathfrak u(\lambda)}
\widetilde f(\lambda)\chi_\varepsilon(\lambda)}
{\mu-\lambda}\,d\lambda .
\end{equation}
Let
\[
h_\varepsilon(\lambda)
:=
\overline{\mathfrak u(\lambda)}
\widetilde f(\lambda)\chi_\varepsilon(\lambda).
\]
With the half-line Hilbert transform convention
\[
Hg(\mu)
:=
\frac1\pi\operatorname{p.v.}
\int_0^\infty\frac{g(\lambda)}{\mu-\lambda}\,d\lambda,
\]
we have
\[
g_{\varepsilon,\delta}
=
\frac{i}{4\pi}\mathfrak u\,H_\delta h_\varepsilon,
\]
where \(H_\delta\) denotes the corresponding truncated Hilbert transform. Since the
truncated Hilbert transforms converge to \(H\) in \(L^2(0,\infty)\), and since \(H\) is
bounded on \(L^2(0,\infty)\), we obtain
\[
g_{\varepsilon,\delta}
\longrightarrow
g_\varepsilon
:=
\frac{i}{4\pi}\mathfrak u\,Hh_\varepsilon
\qquad\text{in }L^2(0,\infty)
\]
as \(\delta\to0_+\). Hence, by the \(L^2\)-unitarity of the normalized distorted Fourier
transform, as \(\delta\to0_+\),
\[
r_{3,\varepsilon,\delta}
\longrightarrow
r_{3,\varepsilon}
:=
\mathcal F_{d,u}^{-1}g_\varepsilon
\qquad\text{in }L^2_+(\mathbb R).
\]
Next, since
\[
h_\varepsilon
\to
\overline{\mathfrak u}\,\widetilde f
\qquad\text{in }L^2(0,\infty)
\]
as \(\varepsilon\to0_{+}\), another use of the \(L^2\)-boundedness of \(H\) gives
\[
g_\varepsilon
\longrightarrow
g
:=
\frac{i}{4\pi}
\mathfrak u\,H(\overline{\mathfrak u}\widetilde f)
\qquad\text{in }L^2(0,\infty).
\]
Therefore the \(n_1\)-contribution has the \(L^2\)-limit
\[
r_3:=\mathcal F_{d,u}^{-1}g,
\]
and its distorted Fourier transform is
\begin{equation}
\label{eq:Xstar-r3-transform}
\widetilde r_3(\mu)
=
\frac{i}{4\pi}
\mathfrak u(\mu)
H(\overline{\mathfrak u}\widetilde f)(\mu).
\end{equation}
Combining \eqref{eq:Xstar-after-n} with the preceding computation gives
\begin{equation}
\label{eq:Xstar-after-cor}
\mathcal F_{d,u}^{-1}(i\partial_\lambda\widetilde f)
=
-\frac{i}{2\pi}\widetilde f(0_{+})
+
xf
+
r_3
-
\frac1{4\pi}
\mathcal F_{d,u}^{-1}\left(|\mathfrak u|^2\widetilde f\right).
\end{equation}
By \eqref{eq:Xstar-r3-transform},
\begin{equation}
\label{eq:Xstar-r3-combined}
r_3
-
\frac1{4\pi}
\mathcal F_{d,u}^{-1}\left(|\mathfrak u|^2\widetilde f\right)
=
\frac1{4\pi}
\mathcal F_{d,u}^{-1}
\left(
i\mathfrak u\,H(\overline{\mathfrak u}\widetilde f)
-
|\mathfrak u|^2\widetilde f
\right).
\end{equation}
The right-hand side of \eqref{eq:Xstar-r3-combined} belongs to \(L^2_+(\mathbb R)\),
because \(\mathfrak u\in\mathcal S_*\), \(\widetilde f\in L^2(0,\infty)\), and \(H\) is
bounded on \(L^2(0,\infty)\). Also
\[
\mathcal F_{d,u}^{-1}(i\partial_\lambda\widetilde f)\in L^2_+(\mathbb R).
\]
It follows from \eqref{eq:Xstar-after-cor} and \eqref{eq:Xstar-r3-combined} that
\begin{equation}
\label{eq:Xstar-domain-correction}
xf-\frac{i}{2\pi}\widetilde f(0_+)
=
xf+\frac1{2i\pi}\widetilde f(0_+)
\in L^2(\mathbb R).
\end{equation}
We now identify the boundary functional. By \eqref{1.51}, the existence of a constant
\[
\lambda_f:=\frac1{2i\pi}\widetilde f(0_+)
\]
such that \(xf+\lambda_f\in L^2(\mathbb R)\) implies
\[
f\in\operatorname{Dom}(X^*).
\]
By the definition \eqref{1.121} and the identity \eqref{1.9}, we have
\[
X^*f
=
xf+\frac1{2i\pi}I_+(f).
\]
On the other hand, \eqref{eq:Xstar-domain-correction} shows that
\[
xf+\frac1{2i\pi}\widetilde f(0_+)\in L^2(\mathbb R).
\]
The correction constant in the definition \eqref{1.51} is unique, since a nonzero
constant function does not belong to \(L^2(\mathbb R)\). Therefore
\begin{equation}
\label{eq:Xstar-boundary-test}
I_+(f)=\widetilde f(0_+).
\end{equation}
Using \eqref{eq:Xstar-after-cor}, \eqref{eq:Xstar-r3-combined}, and
\eqref{eq:Xstar-boundary-test}, we get
\[
\mathcal F_{d,u}^{-1}(i\partial_\lambda\widetilde f)
=
X^*f
+
\frac1{4\pi}
\mathcal F_{d,u}^{-1}
\left(
i\mathfrak u\,H(\overline{\mathfrak u}\widetilde f)
-
|\mathfrak u|^2\widetilde f
\right).
\]
Applying \(\mathcal F_{d,u}\), we obtain
\begin{equation}
\label{eq:Xstar-final-test}
\mathcal F_{d,u}(X^*f)
=
i\partial_\lambda\widetilde f
-
\frac1{4\pi}
\left(
i\mathfrak u\,H(\overline{\mathfrak u}\widetilde f)
-
|\mathfrak u|^2\widetilde f
\right).
\end{equation}
Equivalently,
\[
\mathcal F_{d,u}(X^*f)
=
i\partial_\lambda\widetilde f
+
B(\mathfrak u)\widetilde f.
\]
This proves the result for \(\widetilde f\in C_c^\infty([0,\infty))\).

It remains to pass to general \(\widetilde f\in H^1(0,\infty)\). Choose
\[
\widetilde f_j\in C_c^\infty([0,\infty))
\]
such that
\[
\widetilde f_j\to\widetilde f
\qquad\text{in }H^1(0,\infty).
\]
Set
\[
f_j:=\mathcal F_{d,u}^{-1}\widetilde f_j.
\]
By the \(L^2\)-unitarity of the normalized distorted Fourier transform,
\[
f_j\to f
\qquad\text{in }L^2_+(\mathbb R).
\]
Moreover, since \(B(\mathfrak u)\) is bounded on \(L^2(0,\infty)\),
\[
i\partial_\lambda\widetilde f_j+B(\mathfrak u)\widetilde f_j
\longrightarrow
i\partial_\lambda\widetilde f+B(\mathfrak u)\widetilde f
\qquad\text{in }L^2(0,\infty).
\]
For each \(j\), the result already proved gives
\[
\mathcal F_{d,u}(X^*f_j)
=
i\partial_\lambda\widetilde f_j+B(\mathfrak u)\widetilde f_j.
\]
Hence \(X^*f_j\) converges in \(L^2_+(\mathbb R)\). Since \(X^*\) is closed, we conclude
that
\[
f\in\operatorname{Dom}(X^*)
\]
and
\begin{equation}
\label{eq:Xstar-final-general}
\mathcal F_{d,u}(X^*f)
=
i\partial_\lambda\widetilde f+B(\mathfrak u)\widetilde f.
\end{equation}
Finally, the trace map
\[
H^1(0,\infty)\ni g\mapsto g(0_+)
\]
is continuous, so
\[
\widetilde f_j(0_+)\to\widetilde f(0_+).
\]
On the other hand, the boundary functional \(I_+\) is continuous for the graph norm of
\(X^*\). Indeed, for \(h\in\operatorname{Dom}(X^*)\),
\[
|I_+(h)|^2
\le
4\pi\|X^*h\|_{L^2}\|h\|_{L^2}.
\]
Applying this to \(h=f_j-f\), and using the graph convergence \(f_j\to f\),
\(X^*f_j\to X^*f\), gives
\[
I_+(f_j)\to I_+(f).
\]
Since
\[
I_+(f_j)=\widetilde f_j(0_+)
\]
for every \(j\), we obtain
\[
I_+(f)=\widetilde f(0_+).
\]
It remains only to consider $X^*+2tL_u$. Since
$uT_{\bar u}=T_{\bar u}^*T_{\bar u}$ is bounded and self-adjoint, the identity
\[
e^{itD^2}X^*e^{-itD^2}=X^*+2tD
\]
and $L_u=D+uT_{\bar u}$ give the physical conjugation formula in the statement.
The operator in parentheses has domain $\operatorname{Dom}(X^*)$, and its
$-i$-multiple is maximally dissipative by the bounded perturbation theorem.
Thus the conjugated operator is closed and has the stated domain. In ordinary
Fourier variables, the algebraic sum is contained in this operator and agrees
with it on $C_c^\infty([0,\infty))$, which is a core; hence this operator is its
closure. Since the realization in Theorem \ref{theorem 1.3} extends the same
algebraic sum and its $-i$-multiple is dissipative, maximality identifies the two.

Likewise, $H^*=-H$ shows that $B(\mathfrak u)$ is bounded and self-adjoint, and
\[
i\partial_\lambda+B(\mathfrak u)+2t\lambda
=e^{it\lambda^2}\left(i\partial_\lambda+
e^{-it\lambda^2}B(\mathfrak u)e^{it\lambda^2}\right)e^{-it\lambda^2}.
\]
Hence this spectral operator has the stated domain and its $-i$-multiple is
maximally dissipative. On $C_c^\infty([0,\infty))$, the first part of the
proposition and Corollary \ref{coro 2.14} give the required intertwining.
Since this is a core, maximality on both sides extends the identity to the full
domains.

Finally, if $f\in\operatorname{Dom}(X^*+2tL_u)$, then
$e^{-it\xi^2}\widehat f\in H^1(0,\infty)$, and we have
\[
I_+(f)=\widehat f(0_+).
\]
Since both quadratic phases equal one at the origin, the two conjugation formulas
and the $H^1$ trace estimate show that $I_+$ and $J_+(h)=h(0_+)$ are continuous
for the corresponding graph norms. Since
\[
I_+(\mathcal F_{d,u}^{-1}h)=h(0_+)=J_+(h)
\]
on $C_c^\infty([0,\infty))$, graph-norm density gives the same identity on the
full domain.
This completes the proof.
\end{proof}
The next elementary lemma shows the corresponding free resolvent formula.
It will be used to separate the free Schrödinger evolution from the defect term in the
spectral representation.
\begin{lemma}[The free resolvent and its boundary value]
\label{lemma:free_resolvent_boundary}
Let \(t\in\mathbb R\), \(z\in\mathbb C_+\), and \(F\in L^2(0,\infty)\). Then
the operator $i\partial_\lambda+2t\lambda$ is understood on the closed domain
\[
\left\{h\in L^2(0,\infty):e^{-it\lambda^2}h\in H^1(0,\infty)\right\},
\]
and
\begin{equation}
\label{3.18}
(i\partial_\lambda+2t\lambda-z)^{-1}F
=
e^{it\lambda^2}(i\partial_\lambda-z)^{-1}e^{-it\lambda^2}F,
\end{equation}
and explicitly
\begin{equation}
\label{3.19}
\bigl((i\partial_\lambda+2t\lambda-z)^{-1}F\bigr)(\lambda)
=
i e^{it\lambda^2-iz\lambda}
\int_\lambda^\infty
e^{-it\eta^2+iz\eta}F(\eta)\,d\eta .
\end{equation}
Moreover, if \(v\in L^2_+(\mathbb R)\) is defined by
\[
\widehat v(\xi)=F(\xi)\mathbf 1_{\xi>0},
\]
then
\begin{equation}
\label{3.20}
\frac1{2i\pi}
J_+\left((i\partial_\lambda+2t\lambda-z)^{-1}F\right)
=
e^{it\partial_x^2}v(z).
\end{equation}
\end{lemma}

\begin{proof}
Let \(U_t\) denote multiplication by \(e^{it\lambda^2}\). On the domain specified
above, we have
\[
U_t(i\partial_\lambda-z)U_t^{-1}
=
i\partial_\lambda+2t\lambda-z.
\]
This gives \eqref{3.18}. The resolvent formula for \(z\in\mathbb C_+\) is
\[
(i\partial_\lambda-z)^{-1}G(\lambda)
=
i e^{-iz\lambda}
\int_\lambda^\infty e^{iz\eta}G(\eta)\,d\eta ,
\qquad G\in L^2(0,\infty).
\]
Taking
\[
G(\eta)=e^{-it\eta^2}F(\eta)
\]
and using \eqref{3.18} gives \eqref{3.19}.

Since \(\Im z>0\), the integral in \eqref{3.19} is absolutely convergent at
\(\lambda=0\) by the Cauchy--Schwarz inequality. Evaluating \eqref{3.19} at
\(\lambda=0_{+}\), we obtain
\[
J_+\left((i\partial_\lambda+2t\lambda-z)^{-1}F\right)
=
i\int_0^\infty e^{-it\eta^2+iz\eta}F(\eta)\,d\eta .
\]
Therefore
\[
\frac1{2i\pi}
J_+\left((i\partial_\lambda+2t\lambda-z)^{-1}F\right)
=
\frac1{2\pi}
\int_0^\infty e^{-it\eta^2+iz\eta}F(\eta)\,d\eta .
\]
On the other hand, since
\[
\widehat v(\xi)=F(\xi)\mathbf 1_{\xi>0},
\]
the positive-frequency point-evaluation formula gives
\[
e^{it\partial_x^2}v(z)
=
\frac1{2\pi}
\int_0^\infty e^{iz\eta}e^{-it\eta^2}F(\eta)\,d\eta .
\]
Thus \eqref{3.20} follows.
\end{proof}
We have now identified both the distorted action of \(X^*\) and the free boundary term.
We can therefore return to Gérard-type explicit formula \eqref{1.8} and rewrite it in distorted Fourier variables.

\subsection{Distorted representation}

In this subsection we derive the spectral representation of the defect term
\[
u(t,z)-e^{it\partial_x^2}v_0^-(z).
\]
We first prove the formula for smooth initial data in $\mathcal S_+(\mathbb R)$, where all distorted Fourier
operations are classical. The rough case will then follow by approximation and the
boundary-functional construction from Lemma \ref{coro a.7}.
\begin{proposition}[Distorted formula]
\label{prop:smooth_defect_formula}
Let \(u_0\in\mathcal S_+(\mathbb R)\), and set
\[
\mathfrak u_0:=\widetilde u_0:=\mathcal F_{d,u_0}u_0.
\]
Let \(u(t)\) be the solution of \eqref{3.1} with initial data \(u_0\), and let
\(v_0^-\in L^2_+(\mathbb R)\) be defined by
\[
\widehat{v_0^-}(\xi)=\mathfrak u_0(\xi)\mathbf 1_{\xi>0}.
\]
Then, for every \(t\in\mathbb R\) and \(z\in\mathbb C_+\),
\begin{equation}
\label{3.21}
u(t,z)
=
\frac1{2i\pi}
J_+\left[
(i\partial_\lambda+B(\mathfrak u_0)+2t\lambda-z)^{-1}\mathfrak u_0
\right],
\end{equation}
and
\begin{equation}
\label{3.22}
\begin{aligned}
u(t,z)-e^{it\partial_x^2}v_0^-(z)
&=
-\frac1{2i\pi}
J_+\left[
(i\partial_\lambda+B(\mathfrak u_0)+2t\lambda-z)^{-1}
B(\mathfrak u_0)
(i\partial_\lambda+2t\lambda-z)^{-1}
\mathfrak u_0
\right] \\
&=
-\frac1{2i\pi}
J_+\left[
A(\mathfrak u_0,t,z)
e^{-it\lambda^2}
B(\mathfrak u_0)
e^{it\lambda^2}
(i\partial_\lambda-z)^{-1}
e^{-it\lambda^2}\mathfrak u_0
\right],
\end{aligned}
\end{equation}
where, in the smooth case,
\begin{equation}
\label{3.23}
A(\mathfrak u_0,t,z)
:=
\left(
i\partial_\lambda
+
e^{-it\lambda^2}B(\mathfrak u_0)e^{it\lambda^2}
-z
\right)^{-1}.
\end{equation}
\end{proposition}

\begin{proof}
By Theorem \ref{theorem 1.3},
\begin{equation}
\label{3.24}
u(t,z)
=
\frac1{2i\pi}
I_+\left[
(X^*+2tL_{u_0}-z\operatorname{Id})^{-1}u_0
\right].
\end{equation}
By Proposition \ref{prop:Xstar_under_Fd} and the standard intertwining identity
\[
\mathcal F_{d,u_0}(L_{u_0}f)(\lambda)
=
\lambda\mathcal F_{d,u_0}f(\lambda),
\]
we have, as an identity of closed operators under the distorted Fourier transform,
\begin{equation}
\label{3.25}
\mathcal F_{d,u_0}
(X^*+2tL_{u_0})
\mathcal F_{d,u_0}^{-1}
=
i\partial_\lambda+B(\mathfrak u_0)+2t\lambda .
\end{equation}
Moreover, for the resolvent vector
\[
g:=(X^*+2tL_{u_0}-z\operatorname{Id})^{-1}u_0,
\]
Proposition \ref{prop:Xstar_under_Fd} gives
\[
I_+(g)=J_+(\mathcal F_{d,u_0}g).
\]
Applying \(\mathcal F_{d,u_0}\) to \eqref{3.24}, and using \eqref{3.25}, gives
\eqref{3.21}.

Let
\[
R_0(t,z):=(i\partial_\lambda+2t\lambda-z)^{-1},
\]
and
\[
R_B(t,z):=(i\partial_\lambda+B(\mathfrak u_0)+2t\lambda-z)^{-1}.
\]
The second resolvent identity gives
\begin{equation}
\label{3.26}
R_B(t,z)-R_0(t,z)
=
-R_B(t,z)B(\mathfrak u_0)R_0(t,z).
\end{equation}
By Lemma \ref{lemma:free_resolvent_boundary},
\begin{equation}
\label{3.27}
\frac1{2i\pi}J_+\bigl(R_0(t,z)\mathfrak u_0\bigr)
=
e^{it\partial_x^2}v_0^-(z).
\end{equation}
Subtracting \eqref{3.27} from \eqref{3.21}, and using \eqref{3.26}, gives the first
identity in \eqref{3.22}.

Finally, conjugating by \(e^{it\lambda^2}\), we obtain
\[
R_B(t,z)
=
e^{it\lambda^2}
A(\mathfrak u_0,t,z)
e^{-it\lambda^2},
\]
and, by \eqref{3.18},
\[
R_0(t,z)
=
e^{it\lambda^2}
(i\partial_\lambda-z)^{-1}
e^{-it\lambda^2}.
\]
Substituting these two identities into the first identity in \eqref{3.22} gives
\[
J_+\left[
R_B(t,z)B(\mathfrak u_0)R_0(t,z)\mathfrak u_0
\right] =
J_+\left[
e^{it\lambda^2}
A(\mathfrak u_0,t,z)
e^{-it\lambda^2}
B(\mathfrak u_0)
e^{it\lambda^2}
(i\partial_\lambda-z)^{-1}
e^{-it\lambda^2}\mathfrak u_0
\right].
\]
Since \(e^{it\lambda^2}=1\) at \(\lambda=0\), we have
\[
J_+\bigl(e^{it\lambda^2}h\bigr)=J_+(h)
\]
whenever the boundary value is defined. Therefore
\[
J_+\left[
R_B(t,z)B(\mathfrak u_0)R_0(t,z)\mathfrak u_0
\right]  =
J_+\left[
A(\mathfrak u_0,t,z)
e^{-it\lambda^2}
B(\mathfrak u_0)
e^{it\lambda^2}
(i\partial_\lambda-z)^{-1}
e^{-it\lambda^2}\mathfrak u_0
\right].
\]
This proves the second identity in \eqref{3.22}.
\end{proof}
The smooth formula is the identity we need, but the main theorem allows rough data in
\(L_{+}^{2}(\mathbb R)\). We now extend the formula by approximating the initial datum with
Schwartz positive-frequency data and by using the stability of the distorted Fourier
transform and of the boundary functional.

\begin{proposition}[Extension of the formula to \(L^2\) data]
\label{prop:rough_defect_formula}
Let
\[
u_0\in L^2_+(\mathbb R),
\]
and let
\[
u\in C(\mathbb R;L^2_+(\mathbb R))
\]
be the corresponding global solution of \eqref{3.1}. Set
\[
\mathfrak u_0:=\widetilde u_0:=\mathcal F_{d,u_0}u_0.
\]
Then
\[
\mathfrak u_0\in L^2(0,\infty)
\]
by Proposition \(\ref{prop 2.6}\). Define \(v_0^-\in L^2_+(\mathbb R)\) by
\[
\widehat{v_0^-}(\xi)=\mathfrak u_0(\xi)\mathbf 1_{\xi>0}.
\]
Then, for every \(t\in\mathbb R\) and \(z\in\mathbb C_+\),
\begin{equation}
\label{3.28}
\begin{aligned}
u(t,z)-e^{it\partial_x^2}v_0^-(z)
&=
-\frac1{2i\pi}
J_+\left[
A(\mathfrak u_0,t,z)
e^{-it\lambda^2}
B(\mathfrak u_0)
e^{it\lambda^2}
(i\partial_\lambda-z)^{-1}
e^{-it\lambda^2}\mathfrak u_0
\right].
\end{aligned}
\end{equation}
Here the whole boundary value
\[
J_+\bigl[A(\mathfrak u_0,t,z)h\bigr],
\qquad h\in L^1(0,\infty),
\]
is understood in the sense of Lemma \(\ref{coro a.7}\).
\end{proposition}

\begin{proof}
Choose
\[
u_0^n\in\mathcal S_+(\mathbb R)
\]
such that
\[
u_0^n\to u_0
\qquad\text{in }L^2_+(\mathbb R).
\]
Set
\[
\mathfrak u_n:=\mathcal F_{d,u_0^n}u_0^n.
\]
By Proposition \(\ref{proposition 3.6}\), since \(u_0^n\in\mathcal S_+(\mathbb R)\),
we have
\[
\mathfrak u_n\in\mathcal S_*.
\]
By Proposition \(\ref{prop:continuity_Fd}\), in its \(L^2\)-continuity form,
\begin{equation}
\label{3.29}
\mathfrak u_n\to\mathfrak u_0
\qquad\text{in }L^2(0,\infty).
\end{equation}
Let \(u_n(t)\) be the solution with initial data \(u_0^n\), and define
\(v_{0,n}^-\in L^2_+(\mathbb R)\) by
\[
\widehat{v_{0,n}^-}(\xi)=\mathfrak u_n(\xi)\mathbf 1_{\xi>0}.
\]
By Proposition \ref{prop:smooth_defect_formula},
\begin{equation}
\label{3.30}
\begin{aligned}
u_n(t,z)-e^{it\partial_x^2}v_{0,n}^-(z)
&=
-\frac1{2i\pi}
J_+\left[
A(\mathfrak u_n,t,z)
e^{-it\lambda^2}
B(\mathfrak u_n)
e^{it\lambda^2}
(i\partial_\lambda-z)^{-1}
e^{-it\lambda^2}\mathfrak u_n
\right].
\end{aligned}
\end{equation}
We pass to the limit in \eqref{3.30}.

First, by the \(L^2\)-continuity of the flow and by the point-evaluation estimate
in \(\mathbb C_+\),
\begin{equation}
\label{3.31}
u_n(t,z)\to u(t,z).
\end{equation}
Also, since \(\mathfrak u_n\to\mathfrak u_0\) in \(L^2(0,\infty)\), we have
\[
v_{0,n}^-\to v_0^-
\qquad\text{in }L^2_+(\mathbb R).
\]
Therefore, again by the point-evaluation estimate,
\begin{equation}
\label{3.32}
e^{it\partial_x^2}v_{0,n}^-(z)
\to
e^{it\partial_x^2}v_0^-(z).
\end{equation}
It remains to pass to the limit on the right-hand side of \eqref{3.30}. Let
\[
R_z:=(i\partial_\lambda-z)^{-1},
\]
and define
\[
r_n
:=
e^{it\lambda^2}R_z e^{-it\lambda^2}\mathfrak u_n,
\qquad
r_0
:=
e^{it\lambda^2}R_z e^{-it\lambda^2}\mathfrak u_0.
\]
By Proposition \ref{prop 2},
\begin{equation}
\label{3.33}
\|r_n-r_0\|_{L^\infty(0,\infty)}
\lesssim_z
\|\mathfrak u_n-\mathfrak u_0\|_{L^2(0,\infty)}
\to0.
\end{equation}
For \(v\in L^2(0,\infty)\) and \(r\in L^\infty(0,\infty)\), the expression
\[
B(v)r
=
-\frac1{4\pi}
\left(
ivH(\overline vr)-|v|^2r
\right)
\]
belongs to \(L^1(0,\infty)\), and
\begin{equation}
\label{3.34}
\|B(v)r\|_{L^1}
\lesssim
\|v\|_{L^2}^2\|r\|_{L^\infty}.
\end{equation}
Moreover, for \(v_1,v_2\in L^2(0,\infty)\) and
\(r_1,r_2\in L^\infty(0,\infty)\),
\begin{equation}
\label{3.35}
\begin{aligned}
\|B(v_1)r_1-B(v_2)r_2\|_{L^1}
&\lesssim
\|v_1-v_2\|_{L^2}
\bigl(\|v_1\|_{L^2}+\|v_2\|_{L^2}\bigr)
\|r_1\|_{L^\infty}  \\
&\quad
+
\|v_2\|_{L^2}^2\|r_1-r_2\|_{L^\infty}.
\end{aligned}
\end{equation}
Indeed, the Hilbert-transform part is estimated by
\[
\begin{aligned}
&\|v_1H(\overline{v_1}r_1)-v_2H(\overline{v_2}r_2)\|_{L^1} \\
&\quad\le
\|(v_1-v_2)H(\overline{v_1}r_1)\|_{L^1}
+
\|v_2H((\overline{v_1}-\overline{v_2})r_1)\|_{L^1} \\
&\qquad
+
\|v_2H(\overline{v_2}(r_1-r_2))\|_{L^1}  \\
&\quad\lesssim
\|v_1-v_2\|_{L^2}
\bigl(\|v_1\|_{L^2}+\|v_2\|_{L^2}\bigr)
\|r_1\|_{L^\infty}
+
\|v_2\|_{L^2}^2
\|r_1-r_2\|_{L^\infty},
\end{aligned}
\]
and the multiplication term
\[
|v_1|^2r_1-|v_2|^2r_2
\]
is estimated in the same way.

Using \eqref{3.29}, \eqref{3.33}, and \eqref{3.35}, we get
\begin{equation}
\label{3.36}
e^{-it\lambda^2}B(\mathfrak u_n)r_n
\to
e^{-it\lambda^2}B(\mathfrak u_0)r_0
\qquad\text{in }L^1(0,\infty).
\end{equation}
By Lemma \ref{coro a.7}, applied with
\[
h_n:=e^{-it\lambda^2}B(\mathfrak u_n)r_n,
\qquad
h_0:=e^{-it\lambda^2}B(\mathfrak u_0)r_0,
\]
and using \eqref{3.29} and \eqref{3.36}, we obtain
\begin{equation}
\label{3.37}
J_+\left[
A(\mathfrak u_n,t,z)
e^{-it\lambda^2}B(\mathfrak u_n)r_n
\right] \longrightarrow
J_+\left[
A(\mathfrak u_0,t,z)
e^{-it\lambda^2}B(\mathfrak u_0)r_0
\right] \text{ as } n \to \infty.
\end{equation}
Taking the limit in \eqref{3.30}, and using \eqref{3.31}, \eqref{3.32}, and
\eqref{3.37}, gives \eqref{3.28}. This completes the proof.
\end{proof}
This completes the distorted Fourier representation for the class of data
used in the main theorem. In the next section we combine this formula with the
oscillatory estimates from the appendix to prove the weak radiation limit required for scattering.

\section{Proof of the main theorem}
\label{section 3}
In this section, we aim to prove Theorem
\ref{theorem 1.5}. The proof has two main steps. First, we show
that a weak radiation limit for the defect term implies full \(L^2\)-scattering.
Second, we verify this weak radiation limit by using the distorted Fourier representation
 obtained in the previous section, together with the oscillatory estimates
proved in the appendix.
\subsection{A sufficient condition for scattering}
The next lemma reduces the proof of the theorem to a weak convergence statement along
the rays \(z+2t\eta\). This criterion is useful because the distorted formula gives precise control in this regime.
\begin{lemma}[A sufficient condition for \(L^2\)-scattering]
\label{lemma 2.29}
Let
\[
u_0\in L_+^{2}(\mathbb R)
\]
and let
\[
u\in C(\mathbb R;L_+^2(\mathbb R))
\]
be the corresponding global solution of \eqref{3.1}. 

Set
\[
\widetilde u_0:=\mathcal F_{d,u_0}u_0,
\]
and define \(v_0^-\in L_+^2(\mathbb R)\) by
\[
\widehat{v_0^-}(\xi)
=
\widetilde u_0(\xi)\mathbf 1_{\xi>0}.
\]
Let
\[
w(\widetilde{u}_0,t,z)
:=
u(t,z)-e^{it\partial_x^2}v_0^-(z),
\qquad z\in\mathbb C_+.
\]
Assume that, for every
\(z\in\mathbb C_+\),
\begin{equation}
\label{eq:l2-weak-radiation-assumption}
|t|^{1/2}e^{-it\eta^2}
w(\widetilde{u}_0,t,z+2t\eta)
\rightharpoonup0
\quad\text{weakly in }L^2(0,\infty)_\eta
\end{equation}
as \(t\to-\infty\). Then
\[
u(t)-e^{it\partial_x^2}v_0^-
\longrightarrow0
\qquad\text{in }L^2(\mathbb R)
\]
as \(t\to-\infty\). Consequently, by Remark \ref{remark 1.101}, Theorem \ref{theorem 1.5} follows.
\end{lemma}
A similar sufficient condition for scattering in $H^1$ was first introduced for the Benjamin–Ono equation in \cite[Lemma 3.2]{0}. Lemma \ref{lemma 2.29} subsequently extends this sufficient condition to $L^2$.
\begin{proof}
For \(z\in\mathbb C_+\), define
\[
F_{t,z}(\eta)
:=
(2|t|)^{1/2}e^{-it\eta^2}w(\widetilde{u}_0, t,z+2t\eta),
\qquad \eta>0,\quad t<0.
\]
From the assumption \eqref{eq:l2-weak-radiation-assumption}, we know that
\begin{equation}
\label{eq:l2-ray-limit-positive-height}
F_{t,z} \rightharpoonup0
\quad\text{weakly in }L^2(0,\infty)_\eta
\end{equation}
as \(t\to-\infty\).
We now prove weak convergence in \(L^2(\mathbb R)\) of the profiles
\[
h_t:=e^{-it\partial_x^2} w = e^{-it\partial_x^2}u(t) - v_0^{-}.
\]
It is enough to test against functions \(f\in L_+^2(\mathbb R)\) such that
\[
\widehat f\in C_c^\infty(0,\infty),
\]
because this class is dense in \(L_+^2(\mathbb R)\), and the family
\(\{h_t:t<0\}\) is bounded in \(L^2\).

Fix such an \(f\). Choose \(y>0\), and define \(f^{(y)}\in L_+^2(\mathbb R)\)
by
\[
\widehat{f^{(y)}}(\xi)
:=
e^{y\xi}\widehat f(\xi).
\]
This is again a smooth positive-frequency function, because \(\widehat f\) is compactly
supported in \((0,\infty)\).

We claim that
\begin{equation}
\label{eq:horizontal-line-identity}
\left\langle
w(t),e^{it\partial_x^2}f
\right\rangle_{L^2(\mathbb R)}
=
\left\langle
w(t,\cdot+iy),e^{it\partial_x^2}f^{(y)}
\right\rangle_{L^2(\mathbb R)} .
\end{equation}
Indeed, by Plancherel's identity,
\[
\begin{aligned}
\left\langle
w(t,\cdot+iy),e^{it\partial_x^2}f^{(y)}
\right\rangle
&=
\frac1{2\pi}
\int_0^\infty
e^{-y\xi}\widehat w(t,\xi)
\overline{e^{-it\xi^2}\widehat{f^{(y)}}(\xi)}
\,d\xi \\
&=
\frac1{2\pi}
\int_0^\infty
\widehat w(t,\xi)
e^{it\xi^2}
\overline{\widehat f(\xi)}
\,d\xi  \\
&=
\left\langle
w(t),e^{it\partial_x^2}f
\right\rangle .
\end{aligned}
\]
Thus,
\[
\left\langle h_t,f\right\rangle
=
\left\langle
w(t,\cdot+iy),e^{it\partial_x^2}f^{(y)}
\right\rangle .
\]
It is well known that, for every $g \in L^2(\mathbb{R})$,
\[
\mathrm{e}^{\mathrm{i} t \partial_x^2} g=\mathrm{e}^{\mathrm{i} \frac{x^2}{4 t}} \frac{\mathrm{e}^{\mathrm{i} \pi / 4}}{\sqrt{4 \pi |t|}} \widehat{g}\left(\frac{x}{2 t}\right)+o(1) \quad \text{ in } L^2
\]
as $t \to -\infty$, this gives
\[
\left\langle
w(t,\cdot+iy),e^{it\partial_x^2}f^{(y)}
\right\rangle
=
\frac{e^{-i\pi/4}}{\sqrt{2\pi}}
\int_0^\infty
F_{t,iy}(\eta)
\overline{\widehat{f^{(y)}}(\eta)}
\,d\eta
+
o(1)
\]
as \(t\to-\infty\). Using
\eqref{eq:l2-ray-limit-positive-height} with \(z=iy\), we get
\[
\lim_{t\to-\infty}
\left\langle h_t,f\right\rangle =0.
\]
Hence
\[
e^{-it\partial_x^2}u(t)
\rightharpoonup
v_0^-
\qquad\text{weakly in }L^2(\mathbb R)
\]
as \(t\to-\infty\).

It remains only to upgrade weak convergence to strong convergence. By conservation of
mass,
\[
\left\|e^{-it\partial_x^2}u(t)\right\|_{L^2}
=
\|u_0\|_{L^2}.
\]
On the other hand, by Plancherel's identity for the distorted Fourier transform,
\[
\|v_0^-\|_{L^2}^2
=
\frac1{2\pi}
\|\widetilde u_0\|_{L^2(0,\infty)}^2
=
\|u_0\|_{L^2}^2.
\]
Therefore weak convergence together with convergence of the norms implies
\[
e^{-it\partial_x^2}u(t)
\longrightarrow
v_0^-
\qquad\text{strongly in }L^2(\mathbb R),
\]
equivalently,
\[
u(t)-e^{it\partial_x^2}v_0^-
\longrightarrow0
\qquad\text{strongly in }L^2(\mathbb R)
\]
as \(t\to-\infty\). This proves the lemma.
\end{proof}
The remaining task is therefore to verify the weak convergence assumption
\eqref{eq:l2-weak-radiation-assumption}. We do this by first proving a compact-uniform decay estimate for smooth
spectral cutoff data, and then approximating the rough spectral data by smooth ones.

\subsection{Completion of the proof}
In this section, we aim to finish the proof of Theorem \ref{theorem 1.5}. The following lemma is the compact-uniform smooth decay estimate. The compactness assumption on
\(\Lambda\) keeps the parameter away from the endpoint \(\eta=0\), which is
exactly the regime required by the oscillatory integral estimate in Proposition
\ref{prop a.3}.
\begin{lemma}[Compact-uniform smooth decay]
\label{lemma:smooth_compact_uniform_decay}
Let
\[
\mathfrak u_0\in L^2(0,\infty).
\]
Let
\[
\mathfrak u,\mathfrak v\in C_c^\infty(0,\infty),
\qquad
z=x_0+iy_0\in\mathbb C_+,
\qquad
\Lambda\Subset(0,\infty).
\]
For \(t<0\) and \(\zeta\in\mathbb C_+\), define
\begin{equation}
\label{eq:smooth-compact-W-def}
W_{\mathfrak u,\mathfrak v}(t,\zeta)
:=
-\frac1{2i\pi}
J_+\left[
A(\mathfrak u_0,t,\zeta)e^{-it\lambda^2}
B(\mathfrak u)e^{it\lambda^2}
(i\partial_\lambda-\zeta)^{-1}
e^{-it\lambda^2}\mathfrak v
\right].
\end{equation}
Here the whole boundary value
\[
J_+\bigl[A(\mathfrak u_0,t,\zeta)h\bigr],
\qquad h\in L^1(0,\infty),
\]
is understood in the sense of Lemma \(\ref{coro a.7}\). Then
\begin{equation}
\label{eq:smooth-compact-uniform-decay}
\sup_{\eta\in\Lambda}
|t|^{1/2}
\left|
W_{\mathfrak u,\mathfrak v}(t,z+2t\eta)
\right|
\longrightarrow0
\qquad\text{as }t\to-\infty.
\end{equation}
\end{lemma}
\begin{proof}
Set
\[
\zeta=z+2t\eta,
\qquad
R_\zeta:=(i\partial_\lambda-\zeta)^{-1}.
\]
Since
\[
\Im(z+2t\eta)=\Im z=y_0>0,
\]
Lemma \ref{coro a.7}, with \(\gamma=y_0\), gives
\begin{equation}
\label{eq:smooth-compact-boundary-estimate}
\left|
J_+\bigl[A(\mathfrak u_0,t,z+2t\eta)h\bigr]
\right|
\le
C_{\mathfrak u_0,y_0}\|h\|_{L^1(0,\infty)}
\end{equation}
for all \(t<0\), all \(\eta\in\Lambda\), and all \(h\in L^1(0,\infty)\).

It remains to estimate the \(L^1\)-input
\[
h_{\mathfrak u,\mathfrak v}(t,\eta,\lambda)
:=
e^{-it\lambda^2}
B(\mathfrak u)e^{it\lambda^2}
R_\zeta e^{-it\lambda^2}\mathfrak v .
\]
We write
\[
r_{t,\eta}(\lambda)
:=
e^{it\lambda^2}
R_\zeta e^{-it\lambda^2}\mathfrak v(\lambda).
\]
Using the explicit formula for the free resolvent, we get
\begin{equation}
\label{eq:smooth-compact-r-formula}
r_{t,\eta}(\lambda)
=
i e^{it(\lambda-\eta)^2}e^{-iz\lambda}
\int_\lambda^\infty
e^{-it(s-\eta)^2}e^{izs}\mathfrak v(s)\,ds .
\end{equation}
Since
\[
B(\mathfrak u)r
=
-\frac1{4\pi}
\left(
i\mathfrak u\,H(\overline{\mathfrak u}r)
-
|\mathfrak u|^2r
\right),
\]
and
\[
Hf(\lambda)
=
\frac1\pi\operatorname{p.v.}
\int_0^\infty\frac{f(\mu)}{\lambda-\mu}\,d\mu ,
\]
a direct computation gives
\begin{equation}
\label{eq:smooth-compact-h-identity}
h_{\mathfrak u,\mathfrak v}(t,\eta,\lambda)
=
\frac{i}{4\pi}
e^{-it\lambda^2}\mathfrak u(\lambda)
\left[
J(t,\lambda,\eta)-\frac{i}{\pi}I(t,\lambda,\eta)
\right],
\end{equation}
where \(I(t,\lambda,\eta)\) and \(J(t,\lambda,\eta)\) are exactly the oscillatory
integrals in Proposition \ref{prop a.3}, with the same functions
\(\mathfrak u,\mathfrak v\) and the same \(z\).

Let
\[
K:=\operatorname{supp}\mathfrak u .
\]
Since \(\mathfrak u\in C_c^\infty(0,\infty)\), we have
\[
K\Subset(0,\infty).
\]
Moreover, by \eqref{eq:smooth-compact-h-identity},
\[
h_{\mathfrak u,\mathfrak v}(t,\eta,\lambda)=0
\qquad\text{for }\lambda\notin K.
\]
Hence
\begin{equation}
\label{eq:smooth-compact-input-L1-bound}
\begin{aligned}
\|h_{\mathfrak u,\mathfrak v}(t,\eta,\cdot)\|_{L^1(0,\infty)}
&\le
\frac1{4\pi}
\|\mathfrak u\|_{L^1}
\sup_{\lambda\in K}
\left|
J(t,\lambda,\eta)-\frac{i}{\pi}I(t,\lambda,\eta)
\right|.
\end{aligned}
\end{equation}
Since
\[
\Lambda,K\Subset(0,\infty),
\]
Proposition \ref{prop a.3} applies and gives
\begin{equation}
\label{eq:smooth-compact-input-decay}
\sup_{\eta\in\Lambda}
|t|^{1/2}
\|h_{\mathfrak u,\mathfrak v}(t,\eta,\cdot)\|_{L^1(0,\infty)}
\longrightarrow0
\qquad\text{as }t\to-\infty.
\end{equation}
Finally, by \eqref{eq:smooth-compact-W-def},
\[
W_{\mathfrak u,\mathfrak v}(t,z+2t\eta)
=
-\frac1{2i\pi}
J_+\bigl[
A(\mathfrak u_0,t,z+2t\eta)
h_{\mathfrak u,\mathfrak v}(t,\eta,\cdot)
\bigr].
\]
Using \eqref{eq:smooth-compact-boundary-estimate} and
\eqref{eq:smooth-compact-input-decay}, we obtain
\[
\sup_{\eta\in\Lambda}
|t|^{1/2}
\left|
W_{\mathfrak u,\mathfrak v}(t,z+2t\eta)
\right|
\le
C_{\mathfrak u_0,y_0}
\sup_{\eta\in\Lambda}
|t|^{1/2}
\|h_{\mathfrak u,\mathfrak v}(t,\eta,\cdot)\|_{L^1}
\longrightarrow0.
\]
This proves \eqref{eq:smooth-compact-uniform-decay}.
\end{proof}

The smooth decay estimate will now be combined with two approximation arguments. The
first approximation replaces the \(B(\mathfrak u_0)\)-factor by \(B(\mathfrak u_\varepsilon)\)
with smooth compactly supported spectral data. The second approximation replaces the
input \(\mathfrak u_0\) by a smooth compactly supported function. The corresponding error
terms are controlled by Corollary \ref{corollary a.4} and Lemma \ref{lemma a.8},
respectively.

We now prove Theorem \ref{theorem 1.5}.
\begin{proof}[Proof of Theorem \ref{theorem 1.5}]
Let
\[
u_0\in L_+^2(\mathbb R),
\]
and set
\[
\mathfrak u_0:=\widetilde u_0:=\mathcal F_{d,u_0}u_0.
\]
By Proposition \(\ref{prop 2.6}\),
\[
\mathfrak u_0\in L^2(0,\infty).
\]
Let \(v_0^-\in L^2_+(\mathbb R)\) be defined by
\[
\widehat{v_0^-}(\xi)=\mathfrak u_0(\xi)\mathbf 1_{\xi>0}.
\]
By Lemma \ref{lemma 2.29}, it suffices to prove that, for every
\(z\in\mathbb C_+\),
\begin{equation}
\label{eq:thm15-weak-goal}
|t|^{1/2}e^{-it\eta^2}
w(\mathfrak u_0,t,z+2t\eta)
\rightharpoonup0
\quad\text{weakly in }L^2(0,\infty)_\eta
\end{equation}
as \(t\to-\infty\), where
\[
w(\mathfrak u_0,t,\zeta)
:=
u(t,\zeta)-e^{it\partial_x^2}v_0^-(\zeta).
\]
Fix
\[
z=x_0+iy_0\in\mathbb C_+.
\]
We first note that the family
\[
W_t(\eta)
:=
|t|^{1/2}e^{-it\eta^2}
w(\mathfrak u_0,t,z+2t\eta)
\]
is bounded in \(L^2(0,\infty)_\eta\), uniformly for \(t<0\). Indeed, since both
\(u(t)\) and \(e^{it\partial_x^2}v_0^-\) belong to \(L^2_+(\mathbb R)\), we infer
\[
\begin{aligned}
|t|\int_0^\infty
|u(t,z+2t\eta)|^2\,d\eta
&=
\frac12
\int_{-\infty}^{x_0}
|u(t,x+iy_0)|^2\,dx  \\
&\le
\frac12
\|u(t)\|_{L^2(\mathbb R)}^2,
\end{aligned}
\]
and similarly
\[
|t|\int_0^\infty
|e^{it\partial_x^2}v_0^-(z+2t\eta)|^2\,d\eta
\le
\frac12
\|v_0^-\|_{L^2(\mathbb R)}^2.
\]
Thus
\begin{equation}
\label{eq:thm15-uniform-L2-bound}
\sup_{t<0}\|W_t\|_{L^2(0,\infty)}<\infty.
\end{equation}
Consequently, by density, it is enough to prove
\begin{equation}
\label{eq:thm15-test-goal}
\int_0^\infty
|t|^{1/2}e^{-it\eta^2}
w(\mathfrak u_0,t,z+2t\eta)
\overline{\phi(\eta)}\,d\eta
\longrightarrow0 \qquad\text{as }t\to-\infty
\end{equation}
for every
\[
\phi\in C_c^\infty(0,\infty).
\]
Fix such a test function \(\phi\), and set
\[
\Lambda:=\operatorname{supp}\phi\Subset(0,\infty).
\]
Let \(\varepsilon>0\). Choose
\[
\mathfrak v_\varepsilon\in C_c^\infty(0,\infty)
\]
such that
\begin{equation}
\label{eq:thm15-v-eps}
\|\mathfrak v_\varepsilon-\mathfrak u_0\|_{L^2(0,\infty)}<\varepsilon.
\end{equation}
For
\[
\mathfrak u,\mathfrak v\in C_c^\infty(0,\infty),
\]
define
\begin{equation}
\label{eq:thm15-smooth-w-def}
w_{\mathfrak u,\mathfrak v}(t,\zeta)
:=
-\frac1{2i\pi}
J_+\left[
A(\mathfrak u_0,t,\zeta)e^{-it\lambda^2}
B(\mathfrak u)e^{it\lambda^2}
(i\partial_\lambda-\zeta)^{-1}
e^{-it\lambda^2}\mathfrak v
\right].
\end{equation}
By Lemma \ref{lemma:smooth_compact_uniform_decay}, for every fixed
\(\mathfrak u\in C_c^\infty(0,\infty)\),
\begin{equation}
\label{eq:thm15-smooth-decay}
\sup_{\eta\in\Lambda}
|t|^{1/2}
\left|
w_{\mathfrak u,\mathfrak v_\varepsilon}(t,z+2t\eta)
\right|
\longrightarrow0
\qquad\text{as }t\to-\infty.
\end{equation}
In particular,
\begin{equation}
\label{eq:thm15-smooth-test-decay}
\int_0^\infty
|t|^{1/2}e^{-it\eta^2}
w_{\mathfrak u,\mathfrak v_\varepsilon}(t,z+2t\eta)
\overline{\phi(\eta)}\,d\eta
\longrightarrow0 \qquad\text{as }t\to-\infty.
\end{equation}
We now compare
\[
w_{\mathfrak u,\mathfrak v_\varepsilon}(t,z+2t\eta)
\]
with
\[
w(\mathfrak u_0,t,z+2t\eta).
\]
By Proposition \ref{prop:rough_defect_formula},
\begin{equation}
\label{eq:thm15-rough-defect-formula}
w(\mathfrak u_0,t,\zeta)
=
-\frac1{2i\pi}
J_+\left[
A(\mathfrak u_0,t,\zeta)e^{-it\lambda^2}
B(\mathfrak u_0)e^{it\lambda^2}
(i\partial_\lambda-\zeta)^{-1}
e^{-it\lambda^2}\mathfrak u_0
\right].
\end{equation}
Therefore
\begin{equation}
\label{eq:thm15-difference-split}
w_{\mathfrak u,\mathfrak v_\varepsilon}(t,z+2t\eta)
-
w(\mathfrak u_0,t,z+2t\eta)
=
I_1(t,\eta)+I_2(t,\eta),
\end{equation}
where
\begin{equation}
\label{eq:thm15-I1-def}
\begin{aligned}
I_1(t,\eta)
&:=
-\frac1{2i\pi}
J_+\left[
A(\mathfrak u_0,t,z+2t\eta)e^{-it\lambda^2}
\bigl(B(\mathfrak u)-B(\mathfrak u_0)\bigr)e^{it\lambda^2}
\right. \\
&\hspace{4.3cm}\left.
\times
(i\partial_\lambda-z-2t\eta)^{-1}
e^{-it\lambda^2}\mathfrak v_\varepsilon
\right],
\end{aligned}
\end{equation}
and
\begin{equation}
\label{eq:thm15-I2-def}
\begin{aligned}
I_2(t,\eta)
&:=
-\frac1{2i\pi}
J_+\left[
A(\mathfrak u_0,t,z+2t\eta)e^{-it\lambda^2}
B(\mathfrak u_0)e^{it\lambda^2}
\right.\\
&\hspace{4.3cm}\left.
\times
(i\partial_\lambda-z-2t\eta)^{-1}
e^{-it\lambda^2}
(\mathfrak v_\varepsilon-\mathfrak u_0)
\right].
\end{aligned}
\end{equation}

We first estimate \(I_1\). Put
\[
R_{t,\eta,z}\mathfrak v_\varepsilon
:=
e^{it\lambda^2}
(i\partial_\lambda-z-2t\eta)^{-1}
e^{-it\lambda^2}\mathfrak v_\varepsilon.
\]
By Corollary \ref{corollary a.4}, for all \(|t|\ge1\),
\begin{equation}
\label{eq:thm15-B-difference-estimate}
\sup_{\eta\in\Lambda}
|t|^{1/2}
\left\|
\bigl(B(\mathfrak u)-B(\mathfrak u_0)\bigr)
R_{t,\eta,z}\mathfrak v_\varepsilon
\right\|_{L^1(0,\infty)}
\le
C_{\Lambda,z,\mathfrak v_\varepsilon}
\|\mathfrak u-\mathfrak u_0\|_{L^2(0,\infty)}
\bigl(
\|\mathfrak u\|_{L^2(0,\infty)}
+
\|\mathfrak u_0\|_{L^2(0,\infty)}
\bigr).
\end{equation}
By Lemma \ref{coro a.7}, applied with \(\gamma=y_0=\Im z\),
\begin{equation}
\label{eq:thm15-boundary-L1-estimate}
\left|
J_+\bigl[A(\mathfrak u_0,t,z+2t\eta)h\bigr]
\right|
\le
C_{\mathfrak u_0,y_0}\|h\|_{L^1(0,\infty)}
\end{equation}
uniformly for \(t<0\), \(\eta\in\Lambda\), and \(h\in L^1(0,\infty)\). Combining
\eqref{eq:thm15-B-difference-estimate} and \eqref{eq:thm15-boundary-L1-estimate}, we get
\begin{equation}
\label{eq:thm15-I1-bound}
\begin{aligned}
\sup_{\eta\in\Lambda}
|t|^{1/2}|I_1(t,\eta)|
&\le
C_{\Lambda,z,\mathfrak v_\varepsilon,\mathfrak u_0}
\|\mathfrak u-\mathfrak u_0\|_{L^2}
\bigl(
\|\mathfrak u\|_{L^2}+\|\mathfrak u_0\|_{L^2}
\bigr).
\end{aligned}
\end{equation}
Choose
\[
\mathfrak u=\mathfrak u_\varepsilon\in C_c^\infty(0,\infty)
\]
so close to \(\mathfrak u_0\) in \(L^2(0,\infty)\) that
\begin{equation}
\label{eq:thm15-I1-eps-bound}
\sup_{\eta\in\Lambda}|t|^{1/2}|I_1(t,\eta)|
\le\varepsilon
\end{equation}
for all \(|t|\ge1\). Therefore
\begin{equation}
\label{eq:thm15-I1-test-bound}
\left|
\int_0^\infty
|t|^{1/2}e^{-it\eta^2}
I_1(t,\eta)\overline{\phi(\eta)}\,d\eta
\right|
\le
\varepsilon\|\phi\|_{L^1(0,\infty)}.
\end{equation}

We now estimate \(I_2\). Set
\[
f_\varepsilon:=\mathfrak v_\varepsilon-\mathfrak u_0.
\]
For \(f\in L^2(0,\infty)\), define
\begin{equation}
\label{eq:thm15-Gf-def}
G_f(t,\zeta)
:=
-\frac1{2i\pi}
J_+\left[
A(\mathfrak u_0,t,\zeta)e^{-it\lambda^2}
B(\mathfrak u_0)e^{it\lambda^2}
(i\partial_\lambda-\zeta)^{-1}
e^{-it\lambda^2}f
\right].
\end{equation}
Then
\[
I_2(t,\eta)=G_{f_\varepsilon}(t,z+2t\eta).
\]
By Lemma \ref{lemma a.8}, there exists
\[
g_{f_\varepsilon}(t)\in L^2_+(\mathbb R)
\]
such that \(G_{f_\varepsilon}(t,\cdot)\) is the Hardy extension of
\(g_{f_\varepsilon}(t)\), and
\begin{equation}
\label{eq:thm15-G-boundary-L2-bound}
\|g_{f_\varepsilon}(t)\|_{L^2(\mathbb R)}
\le
\|f_\varepsilon\|_{L^2(0,\infty)}.
\end{equation}
Thus we infer
\begin{equation}
\label{eq:thm15-G-horizontal-bound}
\|G_{f_\varepsilon}(t,\cdot+iy_0)\|_{L^2_x(\mathbb R)}
\le
\|g_{f_\varepsilon}(t)\|_{L^2(\mathbb R)}
\le
\|f_\varepsilon\|_{L^2(0,\infty)}.
\end{equation}
Since \(t<0\), the change of variables
\[
x=x_0+2t\eta
\]
gives
\begin{equation}
\label{eq:thm15-I2-L2-bound}
\begin{aligned}
\bigl\||t|^{1/2}I_2(t,\eta)\bigr\|_{L^2_\eta(0,\infty)}^2
&=
|t|\int_0^\infty
|G_{f_\varepsilon}(t,z+2t\eta)|^2\,d\eta \\
&=
\frac12
\int_{-\infty}^{x_0}
|G_{f_\varepsilon}(t,x+iy_0)|^2\,dx \\
&\le
\frac12
\|G_{f_\varepsilon}(t,\cdot+iy_0)\|_{L^2_x(\mathbb R)}^2 \\
&\le
\frac12
\|f_\varepsilon\|_{L^2(0,\infty)}^2.
\end{aligned}
\end{equation}
By \eqref{eq:thm15-v-eps},
\begin{equation}
\label{eq:thm15-I2-eps-bound}
\bigl\||t|^{1/2}I_2(t,\eta)\bigr\|_{L^2_\eta(0,\infty)}
\le
\frac{\varepsilon}{\sqrt2}.
\end{equation}
Consequently,
\begin{equation}
\label{eq:thm15-I2-test-bound}
\left|
\int_0^\infty
|t|^{1/2}e^{-it\eta^2}
I_2(t,\eta)\overline{\phi(\eta)}\,d\eta
\right|
\le
\frac{\varepsilon}{\sqrt2}\|\phi\|_{L^2(0,\infty)}.
\end{equation}

Using
\[
w(\mathfrak u_0,t,z+2t\eta)
=
w_{\mathfrak u_\varepsilon,\mathfrak v_\varepsilon}(t,z+2t\eta)
-
I_1(t,\eta)
-
I_2(t,\eta),
\]
and combining \eqref{eq:thm15-smooth-test-decay},
\eqref{eq:thm15-I1-test-bound}, and \eqref{eq:thm15-I2-test-bound}, we obtain
\begin{equation}
\label{eq:thm15-limsup-estimate}
\begin{aligned}
&\limsup_{t\to-\infty}
\left|
\int_0^\infty
|t|^{1/2}e^{-it\eta^2}
w(\mathfrak u_0,t,z+2t\eta)
\overline{\phi(\eta)}\,d\eta
\right|
\\
&\qquad\le
\varepsilon
\left(
\|\phi\|_{L^1(0,\infty)}
+
\frac1{\sqrt2}\|\phi\|_{L^2(0,\infty)}
\right).
\end{aligned}
\end{equation}
Since \(\varepsilon>0\) is arbitrary, we obtain \eqref{eq:thm15-test-goal}. Together with
the uniform \(L^2\)-bound \eqref{eq:thm15-uniform-L2-bound}, this proves the weak
convergence \eqref{eq:thm15-weak-goal}.

Therefore the hypothesis of Lemma \ref{lemma 2.29} is satisfied. Hence
\[
u(t)-e^{it\partial_x^2}v_0^-
\to0
\qquad\text{in }L^2(\mathbb R)
\]
as \(t\to-\infty\). The positive-time scattering statement follows from Remark
\ref{remark 1.101}. This completes the proof of Theorem \ref{theorem 1.5}.
\end{proof}

\begin{appendix}
\section{Appendix}
\label{appendix}
This appendix collects the technical estimates used in the proof of the main theorem. In the first part, we prove the oscillatory
integral bounds needed for the compact-uniform decay of the smooth defect. The second part provides resolvent estimates, constructs the rough boundary functional, and proves the bound used for the defect term.

\subsection{Oscillatory integral estimates}

The main estimate in this subsection compares two oscillatory integrals arising from the
distorted formula \eqref{3.22}. The compactness assumptions on
\(\Lambda\) and \(K\) keep the variables away from the endpoint \(0\), where the phase
analysis would require a different treatment.
\begin{proposition}
\label{prop a.3}
Let \(\Lambda,K\Subset(0,\infty)\), let \(z\in\mathbb C\) satisfy
\(\Im z>0\), and let
\[
\mathfrak u,\mathfrak v\in\mathcal{S}_{*}:=\{\varphi=g|_{(0,\infty)}:\ g\in \mathcal{S}(\mathbb R),\
\operatorname{supp}g\subset[0,\infty)\}.
\]
For \(t<0\), \(\eta\in\Lambda\), and \(\lambda\in K\), define
\[
I(t,\lambda,\eta)
:=
\operatorname{p.v.}\int_0^\infty
\frac{\overline{\mathfrak u(\mu)}e^{it(\mu-\eta)^2}e^{-iz\mu}
\displaystyle\int_\mu^\infty e^{-it(s-\eta)^2}e^{izs}\mathfrak v(s)\,ds}
{\lambda-\mu}\,d\mu,
\]
and
\[
J(t,\lambda,\eta)
:=
\overline{\mathfrak u(\lambda)}e^{it(\lambda-\eta)^2}e^{-iz\lambda}
\int_\lambda^\infty e^{-it(s-\eta)^2}e^{izs}\mathfrak v(s)\,ds.
\]
Then
\[
\sup_{\eta\in\Lambda,\ \lambda\in K}
|t|^{1/2}\left|J(t,\lambda,\eta)-\frac{i}{\pi}I(t,\lambda,\eta)\right|
\longrightarrow 0
\qquad\text{as }t\to-\infty.
\]
More precisely, for \(t\le -1\),
\[
\sup_{\eta\in\Lambda,\ \lambda\in K}
\left|J(t,\lambda,\eta)-\frac{i}{\pi}I(t,\lambda,\eta)\right|
\lesssim_{\Lambda,K,\mathfrak u,\mathfrak v,z} |t|^{-1}.
\]
\end{proposition}

\begin{proof}
We prove the compact-uniform version of the estimate. Thus we use that
\[
\Lambda,K\Subset(0,\infty).
\]
Set
\[
\tau:=-t>0.
\]
Then \(\tau\to+\infty\) as \(t\to-\infty\).

\medskip
\noindent\textbf{Step 1: rewriting and separation of the singularity.}
Changing variables \(s=\lambda+v\) in the definition of \(J\), we get
\begin{equation}
\label{eq:propa3-J-rewrite}
J(t,\lambda,\eta)
=
\overline{\mathfrak u(\lambda)}
\int_0^\infty
e^{i\tau(v^2-2\eta v+2\lambda v)}
e^{izv}
\mathfrak v(\lambda+v)\,dv .
\end{equation}
Similarly, changing variables \(s=\mu+v\) in the inner integral defining \(I\), we obtain
\begin{equation}
\label{eq:propa3-I-rewrite}
I(t,\lambda,\eta)
=
\operatorname{p.v.}\int_0^\infty
\frac{1}{\lambda-\mu}
\left(
\int_0^\infty
e^{i\tau(v^2-2\eta v+2\mu v)}
e^{izv}
\overline{\mathfrak u(\mu)}\mathfrak v(\mu+v)\,dv
\right)d\mu .
\end{equation}
For \(v\ge0\), set
\[
f_v(\mu):=\overline{\mathfrak u(\mu)}\mathfrak v(\mu+v).
\]
Choose \(\chi\in C_c^\infty((0,\infty))\) such that \(\chi=1\) in a neighborhood of
\(\Lambda\cup K\). For \(\lambda\in K\), define
\[
G_v(\mu;\lambda)
:=
\frac{f_v(\mu)-\chi(\mu)f_v(\lambda)}{\mu-\lambda},
\]
with the smooth extension at \(\mu=\lambda\), and define
\[
H_v(\mu;\lambda)
:=
\frac{(1-\chi(\mu))f_v(\lambda)}{\mu-\lambda}.
\]
Then
\[
\frac{f_v(\mu)-f_v(\lambda)}{\mu-\lambda}
=
G_v(\mu;\lambda)-H_v(\mu;\lambda).
\]
Using the identity
\[
\operatorname{p.v.}\int_0^\infty
\frac{e^{i\kappa\mu}}{\lambda-\mu}\,d\mu
=
-i\pi e^{i\kappa\lambda}
+
\int_{-\infty}^0
\frac{e^{i\kappa\mu}}{\mu-\lambda}\,d\mu,
\qquad \kappa>0,\ \lambda>0,
\]
and first working with truncated principal values, we obtain
\begin{equation}
\label{eq:propa3-error-decomposition}
I(t,\lambda,\eta)
=
-i\pi J(t,\lambda,\eta)
+
E_1(\tau,\lambda,\eta)
-
E_2(\tau,\lambda,\eta)
+
E_0(\tau,\lambda,\eta),
\end{equation}
where
\begin{equation}
\label{eq:propa3-E1-def}
E_1(\tau,\lambda,\eta)
=
\int_0^\infty\int_{-\infty}^0
e^{i\tau\Psi_\eta(v,\mu)}
e^{izv}
\frac{f_v(\lambda)}{\mu-\lambda}\,d\mu\,dv,
\end{equation}
\begin{equation}
\label{eq:propa3-E2-def}
E_2(\tau,\lambda,\eta)
=
\int_0^\infty\int_0^\infty
e^{i\tau\Psi_\eta(v,\mu)}
e^{izv}
G_v(\mu;\lambda)\,d\mu\,dv,
\end{equation}
\begin{equation}
\label{eq:propa3-E0-def}
E_0(\tau,\lambda,\eta)
=
\int_0^\infty\int_0^\infty
e^{i\tau\Psi_\eta(v,\mu)}
e^{izv}
H_v(\mu;\lambda)\,d\mu\,dv,
\end{equation}
and
\[
\Psi_\eta(v,\mu):=v^2-2\eta v+2\mu v.
\]
The estimates below justify the improper limits and the passage from the truncated
principal values to \eqref{eq:propa3-error-decomposition}. From
\eqref{eq:propa3-error-decomposition},
\begin{equation}
\label{eq:propa3-main-reduction}
J(t,\lambda,\eta)-\frac{i}{\pi}I(t,\lambda,\eta)
=
-\frac{i}{\pi}
\bigl(
E_1(\tau,\lambda,\eta)
-
E_2(\tau,\lambda,\eta)
+
E_0(\tau,\lambda,\eta)
\bigr).
\end{equation}
It remains to prove that \(E_1,E_2,E_0=O(\tau^{-1})\), uniformly for
\((\eta,\lambda)\in\Lambda\times K\).

\medskip
\noindent\textbf{Step 2: amplitude bounds.}
Since \(\mathfrak u,\mathfrak v\in\mathcal S_*\), for every \(a,b,N\ge0\),
\begin{equation}
\label{eq:propa3-f-bound}
|\partial_v^a\partial_\mu^b f_v(\mu)|
\le
C_{a,b,N}(1+\mu+v)^{-N},
\qquad v\ge0,\quad \mu\ge0.
\end{equation}
Because \(\chi=1\) in a neighborhood of \(K\), the corrected difference quotient
\(G_v\) is smooth at \(\mu=\lambda\). Moreover, the cutoff removes the non-decaying
constant tail. Hence, for every \(a,b,N\ge0\),
\begin{equation}
\label{eq:propa3-G-bound}
|\partial_v^a\partial_\mu^bG_v(\mu;\lambda)|
\le
C_{a,b,N,K}(1+\mu+v)^{-N},
\qquad v\ge0,\quad \mu\ge0,\quad \lambda\in K.
\end{equation}
On the other hand, since \(1-\chi\) vanishes in a neighborhood of \(K\), the denominator
\(\mu-\lambda\) in \(H_v\) is bounded away from zero on the support of \(1-\chi\), for
\(\lambda\in K\). Thus, for every \(a,b,N\ge0\),
\begin{equation}
\label{eq:propa3-H-bound}
|\partial_v^a\partial_\mu^bH_v(\mu;\lambda)|
\le
C_{a,b,N,K}(1+v)^{-N}(1+\mu)^{-1-b},
\qquad v\ge0,\quad \mu\ge0,\quad \lambda\in K.
\end{equation}

\medskip
\noindent\textbf{Step 3: estimate of \(E_1\).}
We prove
\begin{equation}
\label{eq:propa3-E1-estimate}
\sup_{\eta\in\Lambda,\lambda\in K}
|E_1(\tau,\lambda,\eta)|
\le
C_{\Lambda,K,\mathfrak u,\mathfrak v,z}\tau^{-1},
\qquad \tau\ge1.
\end{equation}
Set
\[
Q_\lambda(v):=e^{izv}f_v(\lambda)
=
e^{izv}\overline{\mathfrak u(\lambda)}\mathfrak v(\lambda+v).
\]
Then, for every \(m,N\ge0\),
\begin{equation}
\label{eq:propa3-Q-bound}
\sup_{\lambda\in K,\ v\ge0}
(1+v)^N|\partial_v^mQ_\lambda(v)|
\le
C_{m,N,K,\mathfrak u,\mathfrak v,z}.
\end{equation}
Changing variables \(\mu=-\rho\) in \eqref{eq:propa3-E1-def}, we get
\[
E_1(\tau,\lambda,\eta)
=
-\int_0^\infty\int_0^\infty
e^{i\tau\Phi_{\eta,\rho}(v)}
\frac{Q_\lambda(v)}{\rho+\lambda}\,d\rho\,dv,
\]
where
\[
\Phi_{\eta,\rho}(v):=v^2-2(\eta+\rho)v.
\]
Let
\[
\eta_0:=\inf\Lambda>0,
\qquad
\lambda_0:=\inf K>0,
\qquad
\delta:=\frac{\eta_0}{4}.
\]
Split
\[
E_1=E_{1,\mathrm{s}}+E_{1,\mathrm{l}},
\]
where \(E_{1,\mathrm{s}}\) is the contribution of \(0\le v\le\delta\), and
\(E_{1,\mathrm{l}}\) is the contribution of \(v\ge\delta\).

For \(0\le v\le\delta\),
\[
|\partial_v\Phi_{\eta,\rho}(v)|
=
2(\eta+\rho-v)
\ge
c(\eta+\rho),
\]
uniformly for \(\eta\in\Lambda\) and \(\rho\ge0\). Integrating once by parts in \(v\), and
using \eqref{eq:propa3-Q-bound}, gives
\[
\left|
\int_0^\delta e^{i\tau\Phi_{\eta,\rho}(v)}Q_\lambda(v)\,dv
\right|
\le
C\tau^{-1}\frac1{\eta+\rho}.
\]
Therefore
\[
|E_{1,\mathrm{s}}|
\le
C\tau^{-1}
\int_0^\infty
\frac{d\rho}{(\rho+\lambda)(\rho+\eta)}
\le
C_{\Lambda,K,\mathfrak u,\mathfrak v,z}\tau^{-1}.
\]
For \(v\ge\delta\), integrate by parts in \(\rho\):
\[
\int_0^R\frac{e^{-i2\tau v\rho}}{\rho+\lambda}\,d\rho
=
\left[
\frac{e^{-i2\tau v\rho}}{-i2\tau v(\rho+\lambda)}
\right]_{\rho=0}^{\rho=R}
-
\frac1{i2\tau v}
\int_0^R
\frac{e^{-i2\tau v\rho}}{(\rho+\lambda)^2}\,d\rho .
\]
Thus
\[
\left|
\int_0^R\frac{e^{-i2\tau v\rho}}{\rho+\lambda}\,d\rho
\right|
\le
\frac{C_K}{\tau v},
\qquad v\ge\delta.
\]
Using \eqref{eq:propa3-Q-bound} and letting \(R\to\infty\), we obtain
\[
|E_{1,\mathrm{l}}|
\le
C_K\tau^{-1}
\int_\delta^\infty \frac{|Q_\lambda(v)|}{v}\,dv
\le
C_{\Lambda,K,\mathfrak u,\mathfrak v,z}\tau^{-1}.
\]
Combining the estimates for \(E_{1,\mathrm{s}}\) and \(E_{1,\mathrm{l}}\) proves
\eqref{eq:propa3-E1-estimate}.

\medskip
\noindent\textbf{Step 4: estimate of the cutoff-tail term \(E_0\).}
We prove
\begin{equation}
\label{eq:propa3-E0-estimate}
\sup_{\eta\in\Lambda,\lambda\in K}
|E_0(\tau,\lambda,\eta)|
\le
C_{\Lambda,K,\mathfrak u,\mathfrak v,z}\tau^{-1},
\qquad \tau\ge1.
\end{equation}
Since \(\Lambda\Subset(0,\infty)\), there exists
\(\delta_0>0\) such that, on the support of \(1-\chi\), for \(0\le v\le\delta_0\),
\begin{equation}
\label{eq:propa3-tail-smallv-phase}
|\partial_v\Psi_\eta(v,\mu)|
=
2|v+\mu-\eta|
\ge
c_\Lambda(1+\mu),
\qquad \eta\in\Lambda,\quad \mu\ge0.
\end{equation}
Let \(E_{0,\mathrm{s}}\) and \(E_{0,\mathrm{l}}\) be the contributions of
\(0\le v\le\delta_0\) and \(v\ge\delta_0\), respectively.

For the small \(v\) part, integrate by parts in \(v\). By
\eqref{eq:propa3-H-bound} and \eqref{eq:propa3-tail-smallv-phase},
\[
\left|
\int_0^{\delta_0}
e^{i\tau\Psi_\eta(v,\mu)}
e^{izv}H_v(\mu;\lambda)\,dv
\right|
\le
C\tau^{-1}(1+\mu)^{-2}.
\]
Integrating in \(\mu\), we get
\[
|E_{0,\mathrm{s}}|
\le
C_{\Lambda,K,\mathfrak u,\mathfrak v,z}\tau^{-1}.
\]
For \(v\ge\delta_0\), use
\[
\partial_\mu\Psi_\eta(v,\mu)=2v.
\]
Integrating once by parts in \(\mu\), and using \eqref{eq:propa3-H-bound}, gives
\[
\left|
\int_0^\infty
e^{i\tau\Psi_\eta(v,\mu)}
e^{izv}H_v(\mu;\lambda)\,d\mu
\right|
\le
\frac{C}{\tau v}
\left(
|H_v(0;\lambda)|
+
\int_0^\infty |\partial_\mu H_v(\mu;\lambda)|\,d\mu
\right).
\]
The expression in parentheses decays rapidly in \(v\), uniformly for \(\lambda\in K\).
Therefore
\[
|E_{0,\mathrm{l}}|
\le
C\tau^{-1}
\int_{\delta_0}^\infty
\frac{(1+v)^{-2}}{v}\,dv
\le
C_{\Lambda,K,\mathfrak u,\mathfrak v,z}\tau^{-1}.
\]
Combining the estimates for \(E_{0,\mathrm{s}}\) and \(E_{0,\mathrm{l}}\) proves
\eqref{eq:propa3-E0-estimate}.

\medskip
\noindent\textbf{Step 5: estimate of \(E_2\).}
We prove
\begin{equation}
\label{eq:propa3-E2-estimate}
\sup_{\eta\in\Lambda,\lambda\in K}
|E_2(\tau,\lambda,\eta)|
\le
C_{\Lambda,K,\mathfrak u,\mathfrak v,z}\tau^{-1},
\qquad \tau\ge1.
\end{equation}
The phase
\[
\Psi_\eta(v,\mu)=v^2-2\eta v+2\mu v
\]
has a unique critical point on
\[
D_2:=\{(v,\mu):v\ge0,\ \mu\ge0\},
\]
namely \((v,\mu)=(0,\eta)\).

Choose \(\chi_0\in C_c^\infty(\mathbb R^2)\) such that
\[
\chi_0\equiv1\quad\text{on }\{|x|\le1/2,\ |y|\le1/2\},
\qquad
\operatorname{supp}\chi_0\subset\{|x|\le1,\ |y|\le1\}.
\]
For \(\eta\in\Lambda\), define
\[
\chi_\eta(v,\mu)
:=
\chi_0\left(\frac{4v}{\eta_0},\frac{4(\mu-\eta)}{\eta_0}\right).
\]
Then
\[
\operatorname{supp}\chi_\eta
\subset
\left\{
0\le v\le\frac{\eta_0}{4},
\quad
|\mu-\eta|\le\frac{\eta_0}{4}
\right\}
\subset D_2,
\]
and all derivatives of \(\chi_\eta\) are bounded uniformly for \(\eta\in\Lambda\).

Decompose
\[
E_2=E_2^{\mathrm{near}}+E_2^{\mathrm{far}},
\]
where
\[
E_2^{\mathrm{near}}
:=
\iint_{D_2}
e^{i\tau\Psi_\eta(v,\mu)}
\chi_\eta(v,\mu)e^{izv}G_v(\mu;\lambda)\,d\mu\,dv,
\]
and
\[
E_2^{\mathrm{far}}
:=
\iint_{D_2}
e^{i\tau\Psi_\eta(v,\mu)}
(1-\chi_\eta(v,\mu))e^{izv}G_v(\mu;\lambda)\,d\mu\,dv.
\]

\smallskip
\noindent\emph{Far part.}
Choose \(\theta\in C^\infty([0,\infty))\) such that
\[
0\le\theta\le1,\qquad
\theta(v)=0\ \text{for }0\le v\le\frac{\eta_0}{32},
\qquad
\theta(v)=1\ \text{for }v\ge\frac{\eta_0}{16}.
\]
Write
\[
E_2^{\mathrm{far}}=F_1+F_2,
\]
where \(F_1\) has amplitude
\[
\theta(v)(1-\chi_\eta(v,\mu))e^{izv}G_v(\mu;\lambda),
\]
and \(F_2\) has amplitude
\[
(1-\theta(v))(1-\chi_\eta(v,\mu))e^{izv}G_v(\mu;\lambda).
\]
On the support of the amplitude of \(F_1\),
\[
v\ge\frac{\eta_0}{32},
\]
and hence
\[
|\partial_\mu\Psi_\eta(v,\mu)|=2v\ge\frac{\eta_0}{16}.
\]
Integrating once by parts in \(\mu\), and using \eqref{eq:propa3-G-bound}, gives
\[
\sup_{\eta\in\Lambda,\lambda\in K}|F_1|
\le
C_{\Lambda,K,\mathfrak u,\mathfrak v,z}\tau^{-1}.
\]
On the support of the amplitude of \(F_2\), we have
\[
v\le\frac{\eta_0}{16}.
\]
Moreover, since \(\chi_\eta\equiv1\) whenever
\[
v\le\frac{\eta_0}{8},
\qquad
|\mu-\eta|\le\frac{\eta_0}{8},
\]
the factor \(1-\chi_\eta\) implies
\[
|\mu-\eta|\ge\frac{\eta_0}{8}
\]
on the support of \(F_2\). Therefore
\[
|\partial_v\Psi_\eta(v,\mu)|
=
2|v+\mu-\eta|
\ge
\frac{\eta_0}{8}.
\]
Integrating once by parts in \(v\), keeping the boundary term at \(v=0\), and using
\eqref{eq:propa3-G-bound}, we obtain
\[
\sup_{\eta\in\Lambda,\lambda\in K}|F_2|
\le
C_{\Lambda,K,\mathfrak u,\mathfrak v,z}\tau^{-1}.
\]
Thus
\begin{equation}
\label{eq:propa3-E2-far-estimate}
\sup_{\eta\in\Lambda,\lambda\in K}
|E_2^{\mathrm{far}}(\tau,\lambda,\eta)|
\le
C_{\Lambda,K,\mathfrak u,\mathfrak v,z}\tau^{-1}.
\end{equation}

\smallskip
\noindent\emph{Near part.}
Inside the support of \(\chi_\eta\), introduce
\[
y:=\mu+v-\eta,
\qquad
\mu=y-v+\eta.
\]
Then
\[
\Psi_\eta(v,\mu)
=
v^2-2\eta v+2v(y-v+\eta)
=
2vy-v^2.
\]
Therefore
\begin{equation}
\label{eq:propa3-E2-near-rewrite}
E_2^{\mathrm{near}}(\tau,\lambda,\eta)
=
\int_0^\infty e^{-i\tau v^2}
\left(
\int_{\mathbb R}e^{i2\tau vy}b_{\eta,\lambda}(v,y)\,dy
\right)dv,
\end{equation}
where
\[
b_{\eta,\lambda}(v,y)
:=
\chi_\eta(v,y-v+\eta)e^{izv}G_v(y-v+\eta;\lambda).
\]
Because of \eqref{eq:propa3-G-bound}, the compact support of \(\chi_\eta\), and the
uniform bounds on the derivatives of \(\chi_\eta\), the function \(b_{\eta,\lambda}\) and
all its derivatives are bounded uniformly in \((\eta,\lambda)\in\Lambda\times K\), and are
supported in a fixed compact subset of the \((v,y)\)-plane.

For fixed \(v\ge0\), define
\[
B_{\eta,\lambda}(v,\tau)
:=
\int_{\mathbb R}e^{i2\tau vy}b_{\eta,\lambda}(v,y)\,dy.
\]
For every \(N\ge0\), integration by parts in \(y\) gives, for \(v>0\),
\[
|B_{\eta,\lambda}(v,\tau)|
\le
C_N(\tau v)^{-N}.
\]
On the other hand,
\[
|B_{\eta,\lambda}(v,\tau)|\le C_0.
\]
Hence
\begin{equation}
\label{eq:propa3-B-bound}
|B_{\eta,\lambda}(v,\tau)|
\le
C_N(1+\tau v)^{-N},
\qquad v\ge0,
\end{equation}
uniformly in \((\eta,\lambda)\in\Lambda\times K\). Taking \(N=2\), and using
\eqref{eq:propa3-E2-near-rewrite}, we obtain
\[
\begin{aligned}
|E_2^{\mathrm{near}}(\tau,\lambda,\eta)|
&\le
\int_0^\infty |B_{\eta,\lambda}(v,\tau)|\,dv  \\
&\le
\int_0^{\tau^{-1}} C\,dv
+
\int_{\tau^{-1}}^\infty C(\tau v)^{-2}\,dv
\le
C_{\Lambda,K,\mathfrak u,\mathfrak v,z}\tau^{-1}.
\end{aligned}
\]
Together with \eqref{eq:propa3-E2-far-estimate}, this proves
\eqref{eq:propa3-E2-estimate}.

\medskip
\noindent\textbf{Step 6: conclusion.}
Combining \eqref{eq:propa3-main-reduction}, \eqref{eq:propa3-E1-estimate},
\eqref{eq:propa3-E0-estimate}, and \eqref{eq:propa3-E2-estimate}, we get
\[
\sup_{\eta\in\Lambda,\lambda\in K}
\left|
J(t,\lambda,\eta)-\frac{i}{\pi}I(t,\lambda,\eta)
\right|
\le
C_{\Lambda,K,\mathfrak u,\mathfrak v,z}\tau^{-1}.
\]
Since \(\tau=|t|\), this proves the stronger estimate
\[
\sup_{\eta\in\Lambda,\lambda\in K}
\left|
J(t,\lambda,\eta)-\frac{i}{\pi}I(t,\lambda,\eta)
\right|
\le
C_{\Lambda,K,\mathfrak u,\mathfrak v,z}|t|^{-1},
\qquad t\le-1.
\]
Consequently,
\[
\sup_{\eta\in\Lambda,\lambda\in K}
|t|^{1/2}
\left|
J(t,\lambda,\eta)-\frac{i}{\pi}I(t,\lambda,\eta)
\right|
\le
C_{\Lambda,K,\mathfrak u,\mathfrak v,z}|t|^{-1/2}
\to0
\]
as \(t\to-\infty\). This completes the proof.
\end{proof}
The next estimate is a version of van der Corput lemma for the free resolvent formula. It is
uniform in the parameter without requiring compactness away from \(0\).
\begin{lemma}
\label{lemma a.3}
Let \(f\in C_c^\infty(0,\infty)\), let \(\Lambda\subset(0,\infty)\), and let
\(z\in\mathbb C_+\). For \(t\in\mathbb R\), \(\lambda>0\), and \(\eta\in\Lambda\), define
\[
F(t,\lambda,\eta)
:=
e^{it(\lambda-\eta)^2-iz\lambda}
\int_\lambda^\infty
e^{-it(\zeta-\eta)^2+iz\zeta}f(\zeta)\,d\zeta .
\]
Then there exists a constant \(C=C_{\Lambda,f,z}>0\) such that
\begin{equation}
\label{A.3.1}
\sup_{\eta\in\Lambda}\sup_{\lambda>0}
|F(t,\lambda,\eta)|
\le
C_{\Lambda,f,z}|t|^{-1/2}
\qquad\text{for all } |t|\ge1 .
\end{equation}
Equivalently,
\begin{equation}
\label{A.3.2}
\sup_{\eta\in\Lambda}
\left\|
e^{it\lambda^2}
(i\partial_\lambda-z-2t\eta)^{-1}
e^{-it\lambda^2}f
\right\|_{L^\infty(0,\infty)}
\le
C_{\Lambda,f,z}|t|^{-1/2}
\qquad\text{for all } |t|\ge1 .
\end{equation}
\end{lemma}

\begin{proof}
Let
\[
\operatorname{supp}f\subset [a,b]
\]
for some \(0<a<b<\infty\). If \(\lambda>b\), then \(f(\zeta)=0\) for every
\(\zeta\ge\lambda\), and hence
\[
F(t,\lambda,\eta)=0.
\]
Thus it suffices to consider \(0<\lambda\le b\).\\\\
We make the change of variables
\[
\zeta=\lambda+s,\qquad s\ge0.
\]
Since
\[
(\lambda+s-\eta)^2
=
(\lambda-\eta)^2+2(\lambda-\eta)s+s^2,
\]
we obtain
\[
F(t,\lambda,\eta)
=
\int_0^\infty
e^{-it(s^2+2(\lambda-\eta)s)+izs}
f(\lambda+s)\,ds .
\]
Because \(f\) is supported in \([a,b]\), the integrand vanishes for \(s>b-\lambda\).
Thus
\[
F(t,\lambda,\eta)
=
\int_0^{b-\lambda}
e^{it\Phi_{\lambda,\eta}(s)}a_\lambda(s)\,ds,
\]
where
\[
\Phi_{\lambda,\eta}(s)
:=
-\bigl(s^2+2(\lambda-\eta)s\bigr),
\qquad
a_\lambda(s):=e^{izs}f(\lambda+s).
\]
The key point is that
\[
\partial_s^2\Phi_{\lambda,\eta}(s)=-2
\]
for all \(s\), \(\lambda\), and \(\eta\). Hence the phase is uniformly non-degenerate,
uniformly for \(\eta\in\Lambda\).\\\\
We now estimate the amplitude. Since \(\Im z>0\),
\[
|e^{izs}|=e^{-\Im z\,s}\le1,
\]
and therefore
\[
\|a_\lambda\|_{L^\infty(0,\infty)}
\le
\|f\|_{L^\infty(0,\infty)} .
\]
Moreover,
\[
\partial_s a_\lambda(s)
=
e^{izs}\bigl(iz f(\lambda+s)+f'(\lambda+s)\bigr),
\]
so
\[
|\partial_s a_\lambda(s)|
\le
e^{-\Im z\,s}
\left(
|z|\,|f(\lambda+s)|+|f'(\lambda+s)|
\right).
\]
Consequently,
\[
\begin{aligned}
\|\partial_s a_\lambda\|_{L^1(0,\infty)}
&\le
|z|\int_0^\infty e^{-\Im z\,s}|f(\lambda+s)|\,ds
+
\int_0^\infty e^{-\Im z\,s}|f'(\lambda+s)|\,ds  \\
&\le
|z|\|f\|_{L^1(0,\infty)}
+
\|f'\|_{L^1(0,\infty)} .
\end{aligned}
\]
Thus
\begin{equation}
\label{A.3.3}
\sup_{\lambda>0}
\left(
\|a_\lambda\|_{L^\infty(0,\infty)}
+
\|\partial_s a_\lambda\|_{L^1(0,\infty)}
\right)
<\infty .
\end{equation}
By the one-dimensional van der Corput lemma for phases with non-vanishing second
derivative, and since
\[
|\partial_s^2\Phi_{\lambda,\eta}|=2
\]
uniformly in \(\lambda\) and \(\eta\), we have
\[
|F(t,\lambda,\eta)|
\le
C|t|^{-1/2}
\left(
\|a_\lambda\|_{L^\infty(0,\infty)}
+
\|\partial_s a_\lambda\|_{L^1(0,\infty)}
\right)
\]
for all \(|t|\ge1\). Combining this with \eqref{A.3.3}, we obtain
\[
\sup_{\eta\in\Lambda}\sup_{\lambda>0}
|F(t,\lambda,\eta)|
\le
C_{\Lambda,f,z}|t|^{-1/2}.
\]
This proves \eqref{A.3.1}.\\\\
Finally, using the explicit resolvent formula
\[
(i\partial_\lambda-z-2t\eta)^{-1}g(\lambda)
=
i e^{-i(z+2t\eta)\lambda}
\int_\lambda^\infty
e^{i(z+2t\eta)\zeta}g(\zeta)\,d\zeta,
\]
with \(g(\lambda)=e^{-it\lambda^2}f(\lambda)\), we get
\[
e^{it\lambda^2}
(i\partial_\lambda-z-2t\eta)^{-1}
e^{-it\lambda^2}f(\lambda)
=
iF(t,\lambda,\eta).
\]
Therefore \eqref{A.3.2} follows from \eqref{A.3.1}.
\end{proof}
Using Lemma \ref{lemma a.3}, we obtain the compact-uniform estimate used to control the
approximation error in the proof of the main theorem.
\begin{corollary}[Compact-uniform \(I_1\) estimate]
\label{corollary a.4}
Let \(z\in\mathbb C_+\), let \(\Lambda\subset(0,\infty)\), and let
\(v\in C_c^\infty(0,\infty)\). For \(u_1,u_2\in L^2(0,\infty)\), set
\[
G_{t,\eta}v
:=
e^{it\lambda^2}
(i\partial_\lambda-z-2t\eta)^{-1}
e^{-it\lambda^2}v .
\]
Then there exists a constant \(C=C_{\Lambda,z,v}>0\) such that, for all \(|t|\ge1\),
\begin{equation}
\label{A.3.4}
\sup_{\eta\in\Lambda}
|t|^{1/2}
\left\|
\bigl(B(u_1)-B(u_2)\bigr)G_{t,\eta}v
\right\|_{L^1(0,\infty)}
\le
C_{\Lambda,z,v}
\|u_1-u_2\|_{L^2(0,\infty)}
\bigl(
\|u_1\|_{L^2(0,\infty)}
+
\|u_2\|_{L^2(0,\infty)}
\bigr).
\end{equation}
Here
\[
B(u)h
:=
-\frac1{4\pi}
\left(
iu\,H(\overline u\,h)-|u|^2h
\right),
\]
where \(H\) denotes the Hilbert transform on the positive half-line.
\end{corollary}

\begin{proof}
By Lemma \ref{lemma a.3}, we have
\begin{equation}
\label{A.3.5}
\sup_{\eta\in\Lambda}
\|G_{t,\eta}v\|_{L^\infty(0,\infty)}
\le
C_{\Lambda,z,v}|t|^{-1/2}
\qquad\text{for all } |t|\ge1 .
\end{equation}
Let
\[
g:=G_{t,\eta}v.
\]
Using the definition of \(B(u)\), we write
\begin{equation}
\label{A.3.6}
(B(u_1)-B(u_2))g
=
-\frac1{4\pi}
\Bigl[
i(u_1-u_2)H(\overline{u_1}g)
+
iu_2H\bigl((\overline{u_1}-\overline{u_2})g\bigr) 
-
\bigl(|u_1|^2-|u_2|^2\bigr)g
\Bigr].
\end{equation}
Since \(H\) is bounded on \(L^2(0,\infty)\), we have
\begin{equation}
\label{A.3.7}
\begin{aligned}
\|(u_1-u_2)H(\overline{u_1}g)\|_{L^1}
&\le
\|u_1-u_2\|_{L^2}
\|H(\overline{u_1}g)\|_{L^2}  \\
&\le
C
\|u_1-u_2\|_{L^2}
\|\overline{u_1}g\|_{L^2}  \\
&\le
C
\|u_1-u_2\|_{L^2}
\|u_1\|_{L^2}
\|g\|_{L^\infty}.
\end{aligned}
\end{equation}
Similarly,
\begin{equation}
\label{A.3.8}
\begin{aligned}
\|u_2H((\overline{u_1}-\overline{u_2})g)\|_{L^1}
&\le
\|u_2\|_{L^2}
\|H((\overline{u_1}-\overline{u_2})g)\|_{L^2}  \\
&\le
C
\|u_2\|_{L^2}
\|u_1-u_2\|_{L^2}
\|g\|_{L^\infty}.
\end{aligned}
\end{equation}
For the multiplication term, we use
\[
|u_1|^2-|u_2|^2
=
(u_1-u_2)\overline{u_1}
+
u_2(\overline{u_1}-\overline{u_2}),
\]
and obtain
\begin{equation}
\label{A.3.9}
\begin{aligned}
\bigl\|
(|u_1|^2-|u_2|^2)g
\bigr\|_{L^1}
&\le
\|g\|_{L^\infty}
\bigl\||u_1|^2-|u_2|^2\bigr\|_{L^1} \\
&\le
\|g\|_{L^\infty}
\|u_1-u_2\|_{L^2}
\bigl(
\|u_1\|_{L^2}+\|u_2\|_{L^2}
\bigr).
\end{aligned}
\end{equation}
Combining \eqref{A.3.6}, \eqref{A.3.7}, \eqref{A.3.8}, and \eqref{A.3.9}, we get
\[
\left\|
(B(u_1)-B(u_2))g
\right\|_{L^1}
\le
C
\|g\|_{L^\infty}
\|u_1-u_2\|_{L^2}
\bigl(
\|u_1\|_{L^2}+\|u_2\|_{L^2}
\bigr).
\]
Using \eqref{A.3.5}, and then taking the supremum over \(\eta\in\Lambda\), we conclude
\[
\sup_{\eta\in\Lambda}
|t|^{1/2}
\left\|
(B(u_1)-B(u_2))G_{t,\eta}v
\right\|_{L^1}
\le
C_{\Lambda,z,v}
\|u_1-u_2\|_{L^2}
\bigl(
\|u_1\|_{L^2}+\|u_2\|_{L^2}
\bigr).
\]
This proves \eqref{A.3.4}.
\end{proof}

\subsection{Resolvent estimates and the
construction of the rough boundary functional}
In this part, the auxiliary results concern boundary values of resolvents.
We first give elementary bounds for the free resolvent, and then use them to
construct the boundary functional needed for rough spectral data.\\\\
Now we give here the elementary estimates for the free resolvent
\((i\partial_\lambda-z)^{-1}\) on the positive half-line. These bounds are repeatedly
used to pass between \(L^1\), \(L^2\) and \(L^\infty\) spaces in the construction of the rough boundary functional.

\begin{proposition}
\label{prop 2}
Let $z\in\mathbb C_{+}$. We have
\begin{equation}\label{eq:L1-L2-bound 4}
\forall f \in L^2(0,\infty), \qquad \|(i\partial_\lambda-z)^{-1}f\|_{L^2(0,\infty)}
\leq \frac{1}{\Im z}\,\|f\|_{L^2(0,\infty)}.
\end{equation}
\begin{equation}\label{eq:L1-L2-bound 1}
\forall f \in L^1(0,\infty), \qquad \|(i\partial_\lambda-z)^{-1}f\|_{L^2(0,\infty)}
\leq \frac{1}{\sqrt{2\,\Im z}}\,\|f\|_{L^1(0,\infty)}.
\end{equation}
\begin{equation}\label{eq:L1-L2-bound 2}
\forall f \in L^2(0,\infty), \qquad \|(i\partial_\lambda-z)^{-1}f\|_{L^\infty(0,\infty)}
\leq \frac{1}{\sqrt{2\,\Im z}}\,\|f\|_{L^2(0,\infty)}.
\end{equation}
\begin{equation}\label{eq:L1-L2-bound 3}
\forall f \in L^1(0,\infty), \qquad\|(i\partial_\lambda-z)^{-1}f\|_{L^\infty(0,\infty)}
\leq \|f\|_{L^1(0,\infty)}.
\end{equation}
\end{proposition}
\begin{proof}
Consider the inhomogeneous first-order ODE
\[
(i\partial_\lambda-z)g=f,
\]
which is equivalent to
\[
g'(\lambda)+iz\,g(\lambda)=-i f(\lambda).
\]
Multiplying by the integrating factor $e^{iz\lambda}$ yields
\[
\frac{d}{d\lambda}\bigl(e^{iz\lambda}g(\lambda)\bigr)
=-i\,e^{iz\lambda}f(\lambda).
\]
Integrating from $\lambda$ to $+\infty$ and using the boundary condition $g(\infty)=0$ gives
\[
e^{iz\lambda}g(\lambda)
= i\int_\lambda^\infty e^{iz s} f(s)\,ds.
\]
Then we set $r=s-\lambda\ge 0$ to obtain
\[
g(\lambda)= i\int_0^\infty e^{iz r}\,f(\lambda+r)\,dr.
\]
Since $f \in L^1(0,\infty)$, we can easily infer \eqref{eq:L1-L2-bound 3}. Since $\Im z>0$,
\[
\|\mathrm{e}^{iz r}\|_{L^2(0,\infty)}^2
=\int_0^\infty \mathrm{e}^{-2(\Im z) r}\,dr
=\frac{1}{2\,\Im z}, \qquad \|\mathrm{e}^{iz r}\|_{L^1(0,\infty)}
=\int_0^\infty \mathrm{e}^{-\Im z r}\,dr
=\frac{1}{\Im z}.
\]
By Cauchy--Schwarz inequality we can deduce \eqref{eq:L1-L2-bound 2}. Finally, by Young's inequality for the (right-shift)
convolution on $(0,\infty)$,
\begin{align*}
&\|g\|_{L^2(0,\infty)}
\le \|\mathrm{e}^{iz r}\|_{L^1(0,\infty)}\,\|f\|_{L^2(0,\infty)}
= \frac{1}{\Im z}\,\|f\|_{L^2(0,\infty)} \\
&\|g\|_{L^2(0,\infty)}
\le \|\mathrm{e}^{iz r}\|_{L^2(0,\infty)}\,\|f\|_{L^1(0,\infty)}
= \frac{1}{\sqrt{2\,\Im z}}\,\|f\|_{L^1(0,\infty)}.
\end{align*}
This proves \eqref{eq:L1-L2-bound 4} and \eqref{eq:L1-L2-bound 1}.
\end{proof}
We now construct the following rough boundary functional.

\begin{lemma}[The rough boundary functional]
\label{coro a.7}
Let
\[
\mathfrak u\in L^2(0,\infty).
\]
For \(w\in L^2(0,\infty)\) and \(r\in L^\infty(0,\infty)\), define
\[
B(w)r
:=
-\frac1{4\pi}
\left(
iw\,H(\overline w r)-|w|^2r
\right),
\]
where \(H\) denotes the Hilbert transform on the positive half-line,
\[
Hh(\mu)
=
\frac1\pi\operatorname{p.v.}
\int_0^\infty
\frac{h(\lambda)}{\mu-\lambda}\,d\lambda .
\]
For smooth \(w\in\mathcal S_*\), set
\[
A(w,t,z)
:=
\left(
i\partial_\lambda
+
e^{-it\lambda^2}B(w)e^{it\lambda^2}
-z
\right)^{-1},
\qquad z\in\mathbb C_+ .
\]
Then there exists a family of bounded linear functionals
\[
\mathcal J_{\mathfrak u}(t,z):L^1(0,\infty)\to\mathbb C,
\qquad t\in\mathbb R,\quad z\in\mathbb C_+,
\]
with the following properties.

First, for every \(\gamma>0\), there exists \(C_{\mathfrak u,\gamma}>0\) such that
\begin{equation}
\label{A7-boundary-L1-bound}
|\mathcal J_{\mathfrak u}(t,z)[f]|
\le
C_{\mathfrak u,\gamma}\|f\|_{L^1(0,\infty)}
\end{equation}
for all \(t\in\mathbb R\), all \(z\in\mathbb C_+\) with
\[
\Im z\ge\gamma,
\]
and all \(f\in L^1(0,\infty)\). The constant is independent of \(t\) and
\(\Re z\).

Second, if
\[
\mathfrak u_n\in \mathcal S_*,
\qquad
\mathfrak u_n\to\mathfrak u
\quad\text{in }L^2(0,\infty),
\]
and if
\[
f_n\in L^1(0,\infty)\cap L^2(0,\infty),
\qquad
f_n\to f
\quad\text{in }L^1(0,\infty),
\]
then, for every fixed \(t\in\mathbb R\) and \(z\in\mathbb C_+\),
\begin{equation}
\label{A7-boundary-approx-conv}
J_+\bigl[A(\mathfrak u_n,t,z)f_n\bigr]
\longrightarrow
\mathcal J_{\mathfrak u}(t,z)[f],
\end{equation}
where \(J_+\) is defined in Proposition \(\ref{prop:Xstar_under_Fd}\).

We shall use the notation
\begin{equation}
\label{A7-boundary-notation}
J_+\bigl[A(\mathfrak u,t,z)f\bigr]
:=
\mathcal J_{\mathfrak u}(t,z)[f],
\qquad f\in L^1(0,\infty).
\end{equation}
\end{lemma}

\begin{proof}
We write
\[
R_z:=(i\partial_\lambda-z)^{-1},
\qquad
\mathcal C_H:=\frac1{4\pi}(\operatorname{Id}-iH).
\]
Thus
\[
B(w)=M_w\mathcal C_HM_{\overline w}
\]
whenever both sides are defined. Indeed,
\[
M_w\mathcal C_HM_{\overline w}g
=
\frac1{4\pi}
\left(
|w|^2g-iwH(\overline wg)
\right)
=
-\frac1{4\pi}
\left(
iwH(\overline wg)-|w|^2g
\right).
\]
For \(w\in L^2(0,\infty)\), set
\[
M_{w,t}:=M_{e^{-it\lambda^2}w},
\qquad
M_{\overline w,t}:=M_{\overline w\,e^{it\lambda^2}}.
\]
For smooth \(w\), we have
\[
e^{-it\lambda^2}B(w)e^{it\lambda^2}
=
M_{w,t}\mathcal C_HM_{\overline w,t}.
\]
We also introduce
\begin{equation}
\label{A7-Gamma-def}
\Gamma_w(t,z)
:=
\mathcal C_HM_{\overline w,t}R_zM_{w,t}.
\end{equation}

We first record that \(B(w)r\in L^1(0,\infty)\) whenever
\(w\in L^2(0,\infty)\) and \(r\in L^\infty(0,\infty)\). Since \(H\) is bounded
on \(L^2(0,\infty)\),
\[
\begin{aligned}
\|B(w)r\|_{L^1}
&\le
C\|wH(\overline wr)\|_{L^1}
+
C\||w|^2r\|_{L^1}       \\
&\le
C\|w\|_{L^2}\|H(\overline wr)\|_{L^2}
+
C\|w\|_{L^2}^2\|r\|_{L^\infty}  \\
&\le
C\|w\|_{L^2}^2\|r\|_{L^\infty}.
\end{aligned}
\]

\medskip
\noindent
\textbf{Step 1: Basic \(L^2\) and Hilbert--Schmidt estimates.}
For \(z=x+iy\in\mathbb C_+\), the free resolvent is given by
\[
(R_z f)(\lambda) := (i\partial_\lambda-z)^{-1} f
=
i\int_\lambda^\infty e^{iz(s-\lambda)}f(s)\,ds .
\]
By Proposition \ref{prop 2}, we have
\begin{equation}
\label{A7-Rz-basic-bounds}
\|R_z\|_{L^1\to L^\infty}\le1,
\qquad
\|R_z\|_{L^1\to L^2}\le (2y)^{-1/2},
\end{equation}
and
\begin{equation}
\label{A7-Rz-L2-bounds}
\|R_z\|_{L^2\to L^\infty}\le (2y)^{-1/2},
\qquad
\|R_z\|_{L^2\to L^2}\le y^{-1}.
\end{equation}
For \(a,b\in L^2(0,\infty)\), the operator
\[
M_{\overline a,t}R_zM_{b,t}
\]
is Hilbert--Schmidt on \(L^2(0,\infty)\), and
\begin{equation}
\label{A7-HS-mixed}
\begin{aligned}
\|M_{\overline a,t}R_zM_{b,t}\|_{\mathrm{HS}}^2
&=
\int_0^\infty\int_\lambda^\infty
|a(\lambda)|^2 e^{-2y(s-\lambda)}|b(s)|^2\,ds\,d\lambda  \\
&\le
\|a\|_{L^2}^2\|b\|_{L^2}^2 .
\end{aligned}
\end{equation}
In particular,
\[
\|M_{\overline a,t}R_zM_{b,t}\|_{\mathcal L(L^2,L^2)}
\le
\|a\|_{L^2}\|b\|_{L^2}.
\]
For \(w\in L^2(0,\infty)\), define
\[
\Theta_w(y)^2
:=
\int_0^\infty\int_\lambda^\infty
|w(\lambda)|^2 e^{-2y(s-\lambda)}|w(s)|^2\,ds\,d\lambda .
\]
Then
\[
\Theta_w(y)\to0
\qquad\text{as }y\to+\infty.
\]
Indeed, the integrand is dominated by
\[
|w(\lambda)|^2|w(s)|^2\mathbf 1_{\{s>\lambda\}},
\]
which is integrable on \((0,\infty)^2\), and it converges pointwise to \(0\) as
\(y\to+\infty\).

Since \(\mathcal C_H\) is bounded on \(L^2(0,\infty)\), \eqref{A7-HS-mixed}
implies
\begin{equation}
\label{A7-Gamma-HS-bound}
\|\Gamma_w(t,z)\|_{\mathcal L(L^2,L^2)}
\le
C\Theta_w(\Im z).
\end{equation}
Moreover, if \(w_n\to w\) in \(L^2(0,\infty)\), then for each fixed
\(z\in\mathbb C_+\),
\begin{equation}
\label{A7-Gamma-L2-continuity}
\Gamma_{w_n}(t,z)\to\Gamma_w(t,z)
\quad\text{in }\mathcal L(L^2,L^2),
\end{equation}
uniformly in \(t\) and \(\Re z\). Indeed,
\[
\Gamma_{w_n}-\Gamma_w
=
\mathcal C_HM_{\overline{w_n-w},t}R_zM_{w_n,t}
+
\mathcal C_HM_{\overline w,t}R_zM_{w_n-w,t},
\]
and \eqref{A7-HS-mixed} gives
\[
\|\Gamma_{w_n}-\Gamma_w\|_{\mathcal L(L^2,L^2)}
\le
C\|w_n-w\|_{L^2}
\bigl(\|w_n\|_{L^2}+\|w\|_{L^2}\bigr).
\]

\medskip
\noindent
\textbf{Step 2: Large-\(\Im z\) construction for \(L^1\)-inputs.}
Choose \(b>0\) so large that
\begin{equation}
\label{A7-large-im-smallness}
\|\Gamma_{\mathfrak u}(t,z)\|_{\mathcal L(L^2,L^2)}
\le
\frac14
\qquad
\text{whenever }\Im z\ge b.
\end{equation}
The choice of \(b\) is independent of \(t\) and \(\Re z\). If
\[
\mathfrak u_n\to\mathfrak u
\quad\text{in }L^2(0,\infty),
\]
then, after possibly increasing \(n\), the same estimate with \(1/2\) in place
of \(1/4\) holds for \(\mathfrak u_n\), uniformly for \(\Im z\ge b\).

For \(\Im z\ge b\) and \(f\in L^1(0,\infty)\), define
\begin{equation}
\label{A7-A1-large-def}
A^{(1)}(\mathfrak u,t,z)f
:=
R_zf
-
R_zM_{\mathfrak u,t}
\bigl(I+\Gamma_{\mathfrak u}(t,z)\bigr)^{-1}
\mathcal C_HM_{\overline{\mathfrak u},t}R_zf .
\end{equation}
Then
\[
A^{(1)}(\mathfrak u,t,z):L^1(0,\infty)\to L^2(0,\infty)
\]
is bounded. Indeed, by \eqref{A7-Rz-basic-bounds},
\[
\|R_zf\|_{L^2}\le (2\Im z)^{-1/2}\|f\|_{L^1},
\]
and
\[
\|M_{\overline{\mathfrak u},t}R_zf\|_{L^2}
\le
\|\mathfrak u\|_{L^2}\|R_zf\|_{L^\infty}
\le
\|\mathfrak u\|_{L^2}\|f\|_{L^1}.
\]
Thus
\[
h:=
\bigl(I+\Gamma_{\mathfrak u}(t,z)\bigr)^{-1}
\mathcal C_HM_{\overline{\mathfrak u},t}R_zf
\]
belongs to \(L^2\), with
\[
\|h\|_{L^2}\le C_{\mathfrak u}\|f\|_{L^1}.
\]
Since
\[
\|M_{\mathfrak u,t}h\|_{L^1}
\le
\|\mathfrak u\|_{L^2}\|h\|_{L^2},
\]
another use of \eqref{A7-Rz-basic-bounds} gives
\begin{equation}
\label{A7-A1-large-bound}
\|A^{(1)}(\mathfrak u,t,z)f\|_{L^2}
\le
C_{\mathfrak u,b}\|f\|_{L^1},
\qquad \Im z\ge b.
\end{equation}

For \(\Im z\ge b\), define the large-\(\Im z\) boundary functional by
\begin{equation}
\label{A7-J1-large-def}
\mathcal J^{(1)}_{\mathfrak u}(t,z)[f]
:=
J_+[R_zf]
-
J_+\left[
R_zM_{\mathfrak u,t}
\bigl(I+\Gamma_{\mathfrak u}(t,z)\bigr)^{-1}
\mathcal C_HM_{\overline{\mathfrak u},t}R_zf
\right].
\end{equation}
The first term is
\[
J_+[R_zf]
=
i\int_0^\infty e^{iz\lambda}f(\lambda)\,d\lambda,
\]
and hence
\begin{equation}
\label{A7-free-boundary-L1}
|J_+[R_zf]|
\le
\|f\|_{L^1}.
\end{equation}
For the second term, if \(h\in L^2(0,\infty)\), then
\begin{equation}
\label{A7-JRMu-bound}
|J_+[R_zM_{\mathfrak u,t}h]|
\le
\|M_{\mathfrak u,t}h\|_{L^1}
\le
\|\mathfrak u\|_{L^2}\|h\|_{L^2}.
\end{equation}
Therefore
\begin{equation}
\label{A7-J1-large-bound}
|\mathcal J^{(1)}_{\mathfrak u}(t,z)[f]|
\le
C_{\mathfrak u,b}\|f\|_{L^1},
\qquad \Im z\ge b.
\end{equation}
The constants above are independent of \(t\) and \(\Re z\).

For smooth \(w\in\mathcal S_*\), the Woodbury formula gives
\[
A(w,t,z)f
=
R_zf
-
R_zM_{w,t}
\bigl(I+\Gamma_w(t,z)\bigr)^{-1}
\mathcal C_HM_{\overline w,t}R_zf
\]
for \(f\in L^1(0,\infty)\cap L^2(0,\infty)\). Hence for \(f\in L^1(0,\infty)\cap L^2(0,\infty)\) and $\Im z \ge b$, 
\[
J_+[A(w,t,z)f]
=
\mathcal J^{(1)}_w(t,z)[f].
\]

\medskip
\noindent
\textbf{Step 3: The auxiliary \(L^2\)-input boundary functionals.}
For \(\Im z\ge b\) and \(g\in L^2(0,\infty)\), define
\begin{equation}
\label{A7-A2-large-def}
A^{(2)}(\mathfrak u,t,z)g
:=
R_zg
-
R_zM_{\mathfrak u,t}
\bigl(I+\Gamma_{\mathfrak u}(t,z)\bigr)^{-1}
\mathcal C_HM_{\overline{\mathfrak u},t}R_zg ,
\end{equation}
and
\begin{equation}
\label{A7-J2-large-def}
\mathcal J^{(2)}_{\mathfrak u}(t,z)[g]
:=
J_+[R_zg]
-
J_+\left[
R_zM_{\mathfrak u,t}
\bigl(I+\Gamma_{\mathfrak u}(t,z)\bigr)^{-1}
\mathcal C_HM_{\overline{\mathfrak u},t}R_zg
\right].
\end{equation}
By \eqref{A7-Rz-L2-bounds},
\[
\|M_{\overline{\mathfrak u},t}R_zg\|_{L^2}
\le
(2\Im z)^{-1/2}\|\mathfrak u\|_{L^2}\|g\|_{L^2},
\]
and hence
\[
A^{(2)}(\mathfrak u,t,z):L^2(0,\infty)\to L^2(0,\infty)
\]
is bounded for \(\Im z\ge b\). Moreover,
\[
|J_+[R_zg]|
\le
\left(\int_0^\infty e^{-2(\Im z)\lambda}\,d\lambda\right)^{1/2}
\|g\|_{L^2}
=
(2\Im z)^{-1/2}\|g\|_{L^2},
\]
and the second term in \(\mathcal J^{(2)}_{\mathfrak u}\) is controlled by
\eqref{A7-JRMu-bound}. Thus
\[
\mathcal J^{(2)}_{\mathfrak u}(t,z)\in (L^2(0,\infty))^*
\]
for \(\Im z\ge b\).

For smooth \(w\in\mathcal S_*\), the smooth theory gives that
\[
A(w,t,z)
=
\left(
i\partial_\lambda+
e^{-it\lambda^2}B(w)e^{it\lambda^2}
-z
\right)^{-1}
\]
is the resolvent of an operator whose $-i$-multiple is maximally dissipative. Hence
\begin{equation}
\label{A7-smooth-resolvent-bound}
\|A(w,t,z)\|_{\mathcal L(L^2,L^2)}
\le
\frac1{\Im z},
\qquad z\in\mathbb C_+,
\end{equation}
and the resolvent identity holds:
\begin{equation}
\label{A7-smooth-resolvent-identity}
A(w,t,z)
=
A(w,t,z_0)
\bigl[I-(z-z_0)A(w,t,z_0)\bigr]^{-1}
\end{equation}
In particular, if
\[
|z-z_0|<\frac12\Im z_0,
\]
then the inverse in \eqref{A7-smooth-resolvent-identity} is uniformly bounded.

In the large-\(\Im z\) region, the explicit formulae
\eqref{A7-A2-large-def} and \eqref{A7-J2-large-def}, together with
\eqref{A7-Gamma-L2-continuity}, imply that if
\[
w_n\to\mathfrak u
\qquad\text{in }L^2(0,\infty),
\]
then
\begin{equation}
\label{A7-A2-large-conv}
A(w_n,t,z)\to A^{(2)}(\mathfrak u,t,z)
\quad\text{in }\mathcal L(L^2,L^2),
\qquad \Im z\ge b,
\end{equation}
and
\begin{equation}
\label{A7-J2-large-conv}
J_+[A(w_n,t,z)\,\cdot]
\to
\mathcal J^{(2)}_{\mathfrak u}(t,z)
\quad\text{in }(L^2(0,\infty))^*,
\qquad \Im z\ge b.
\end{equation}

We now continue the \(L^2\)-input objects to all of \(\mathbb C_+\). Suppose that
at some \(z_0\in\mathbb C_+\) we have already constructed
\(A^{(2)}(\mathfrak u,t,z_0)\) and
\(\mathcal J^{(2)}_{\mathfrak u}(t,z_0)\), and that for every smooth
approximating sequence \(w_n\to\mathfrak u\) in \(L^2\),
\[
A(w_n,t,z_0)\to A^{(2)}(\mathfrak u,t,z_0)
\quad\text{in }\mathcal L(L^2,L^2),
\]
and
\[
J_+[A(w_n,t,z_0)\,\cdot]
\to
\mathcal J^{(2)}_{\mathfrak u}(t,z_0)
\quad\text{in }(L^2(0,\infty))^*.
\]
Since \eqref{A7-smooth-resolvent-bound} passes to the limit, we have
\[
\|A^{(2)}(\mathfrak u,t,z_0)\|_{\mathcal L(L^2,L^2)}
\le
\frac1{\Im z_0}.
\]
Thus, if
\[
|z-z_0|<\frac12\Im z_0,
\]
we may define
\begin{equation}
\label{A7-A2-continuation}
A^{(2)}(\mathfrak u,t,z)
:=
A^{(2)}(\mathfrak u,t,z_0)
\bigl[
I-(z-z_0)A^{(2)}(\mathfrak u,t,z_0)
\bigr]^{-1},
\end{equation}
and
\begin{equation}
\label{A7-J2-continuation}
\mathcal J^{(2)}_{\mathfrak u}(t,z)[g]
:=
\mathcal J^{(2)}_{\mathfrak u}(t,z_0)
\left[
\bigl(
I-(z-z_0)A^{(2)}(\mathfrak u,t,z_0)
\bigr)^{-1}g
\right].
\end{equation}
The smooth resolvent identity \eqref{A7-smooth-resolvent-identity} and the
convergence at \(z_0\) imply the corresponding convergence at \(z\). Therefore
the definition can be propagated along any finite chain of such discs.

Starting from the large-\(\Im z\) region and using a finite chain of discs, we
obtain
\[
\mathcal J^{(2)}_{\mathfrak u}(t,z)\in (L^2(0,\infty))^*
\]
for every \(z\in\mathbb C_+\). The definition is independent of the approximating
sequence and of the chain, because it is obtained as the norm limit of the smooth
functionals
\[
J_+[A(w_n,t,z)\,\cdot],
\]
and this limit is unique.

Moreover, for every \(\gamma>0\), the same finite-step continuation argument yields
\begin{equation}
\label{A7-J2-uniform-bound}
|\mathcal J^{(2)}_{\mathfrak u}(t,z)[g]|
\le
C^{(2)}_{\mathfrak u,\gamma}\|g\|_{L^2}
\end{equation}
for all \(t\in\mathbb R\), all \(z\in\mathbb C_+\) with
\(\Im z\ge\gamma\), and all \(g\in L^2(0,\infty)\). The constant is independent
of \(t\) and \(\Re z\).

\medskip
\noindent
\textbf{Step 4: Extension of the \(L^1\)-input boundary functional to all \(z\).}
Fix \(\gamma>0\). Choose \(Y\ge b\), and let
\[
z=x+iy\in\mathbb C_+,
\qquad y\ge\gamma.
\]
If \(y\ge Y\), we use the large-\(\Im z\) definition
\eqref{A7-J1-large-def}. If \(\gamma\le y<Y\), set
\[
z_Y:=x+iY.
\]
For \(f\in L^1(0,\infty)\), define
\begin{equation}
\label{A7-J1-lower-def}
\mathcal J_{\mathfrak u}(t,z)[f]
:=
\mathcal J^{(1)}_{\mathfrak u}(t,z_Y)[f]
+
(z-z_Y)
\mathcal J^{(2)}_{\mathfrak u}(t,z)
\left[
A^{(1)}(\mathfrak u,t,z_Y)f
\right].
\end{equation}
This is meaningful because
\[
A^{(1)}(\mathfrak u,t,z_Y)f\in L^2(0,\infty)
\]
by \eqref{A7-A1-large-bound}, and
\(\mathcal J^{(2)}_{\mathfrak u}(t,z)\) is bounded on \(L^2\) by
\eqref{A7-J2-uniform-bound}. Combining \eqref{A7-J1-large-bound},
\eqref{A7-A1-large-bound}, and \eqref{A7-J2-uniform-bound}, we obtain
\[
|\mathcal J_{\mathfrak u}(t,z)[f]|
\le
C_{\mathfrak u,\gamma}\|f\|_{L^1},
\]
uniformly for \(t\in\mathbb R\), \(\Re z\in\mathbb R\), and \(\Im z\ge\gamma\).
This proves \eqref{A7-boundary-L1-bound}.

The definition is independent of the auxiliary height \(Y\). Indeed, if
\(f\in L^1\cap L^2\) and \(w\in\mathcal S_*\), then the smooth resolvent identity
gives
\[
A(w,t,z)
=
A(w,t,z_Y)
+
(z-z_Y)A(w,t,z)A(w,t,z_Y),
\]
and therefore \eqref{A7-J1-lower-def} agrees with the usual smooth boundary
value
\[
J_+[A(w,t,z)f].
\]
Passing to smooth approximations \(w_n\to\mathfrak u\) in \(L^2\), we obtain the
same consistency for rough \(\mathfrak u\) and \(f\in L^1\cap L^2\). Since
\(L^1\cap L^2\) is dense in \(L^1\), and both possible definitions are bounded
\(L^1\)-functionals with the same type of bound, the definition is independent
of \(Y\) for all \(f\in L^1\).

\medskip
\noindent
\textbf{Step 5: Stability under approximation.}
Let
\[
\mathfrak u_n\in\mathcal S_*,
\qquad
\mathfrak u_n\to\mathfrak u
\quad\text{in }L^2(0,\infty),
\]
and let
\[
f_n\in L^1(0,\infty)\cap L^2(0,\infty),
\qquad
f_n\to f
\quad\text{in }L^1(0,\infty).
\]
Fix \(t\in\mathbb R\) and \(z=x+iy\in\mathbb C_+\). Choose \(Y\ge b\) so large
that the large-\(\Im z\) construction is valid for \(\mathfrak u\) and for all
sufficiently large \(n\), and set
\[
z_Y:=x+iY.
\]

In the large-\(\Im z\) region, the explicit formulae
\eqref{A7-A1-large-def} and \eqref{A7-J1-large-def}, together with
\[
\mathfrak u_n\to\mathfrak u
\quad\text{in }L^2,
\qquad
f_n\to f
\quad\text{in }L^1,
\]
give
\begin{equation}
\label{A7-A1-approx-conv}
A^{(1)}(\mathfrak u_n,t,z_Y)f_n
\longrightarrow
A^{(1)}(\mathfrak u,t,z_Y)f
\qquad\text{in }L^2(0,\infty),
\end{equation}
and
\begin{equation}
\label{A7-J1-approx-large}
\mathcal J^{(1)}_{\mathfrak u_n}(t,z_Y)[f_n]
\longrightarrow
\mathcal J^{(1)}_{\mathfrak u}(t,z_Y)[f].
\end{equation}
For instance,
\[
\|M_{\overline{\mathfrak u_n},t}R_{z_Y}f_n
-
M_{\overline{\mathfrak u},t}R_{z_Y}f\|_{L^2}
\le
\|\mathfrak u_n-\mathfrak u\|_{L^2}\|R_{z_Y}f\|_{L^\infty}
+
\|\mathfrak u_n\|_{L^2}\|R_{z_Y}(f_n-f)\|_{L^\infty},
\]
and the right-hand side tends to zero. The remaining terms are handled in the
same way, using \eqref{A7-Gamma-L2-continuity} and the convergence of the
inverses.

By Step 3,
\[
\mathcal J^{(2)}_{\mathfrak u_n}(t,z)
\to
\mathcal J^{(2)}_{\mathfrak u}(t,z)
\]
in the operator norm of functionals on \(L^2(0,\infty)\), and these functionals
are uniformly bounded. Combining this with \eqref{A7-A1-approx-conv}, we get
\[
\mathcal J^{(2)}_{\mathfrak u_n}(t,z)
\left[
A^{(1)}(\mathfrak u_n,t,z_Y)f_n
\right]
\to
\mathcal J^{(2)}_{\mathfrak u}(t,z)
\left[
A^{(1)}(\mathfrak u,t,z_Y)f
\right].
\]
Using \eqref{A7-J1-lower-def} for \(\mathfrak u_n\) and for \(\mathfrak u\), and
then using \eqref{A7-J1-approx-large}, we obtain
\[
\mathcal J_{\mathfrak u_n}(t,z)[f_n]
\to
\mathcal J_{\mathfrak u}(t,z)[f].
\]
For smooth \(\mathfrak u_n\), the left-hand side is exactly
\[
J_+\bigl[A(\mathfrak u_n,t,z)f_n\bigr].
\]
Thus \eqref{A7-boundary-approx-conv} follows, and the proof is complete.
\end{proof}

The last auxiliary estimate shows that the defect contribution produced by the rough
boundary functional is the Hardy extension of an \(L^2_+\)-function. This is the input
used in the final approximation step of the proof of Theorem \ref{theorem 1.5}.
\begin{lemma}
\label{lemma a.8}
Let
\[
u_0\in L^2_+(\mathbb R),
\]
and set
\[
\mathfrak u_0:=\widetilde u_0:=\mathcal F_{d,u_0}u_0 .
\]
Then
\[
\mathfrak u_0\in L^2(0,\infty)
\]
by Proposition \(\ref{prop 2.6}\). For \(f\in L^2(0,\infty)\), define
\begin{equation}
\label{A.8.1}
G_f(t,z)
:=
-\frac{1}{2i\pi}
J_+\left[
A(\mathfrak u_0,t,z)
e^{-it\lambda^2}
B(\mathfrak u_0)
e^{it\lambda^2}
(i\partial_\lambda-z)^{-1}
e^{-it\lambda^2}f
\right],
\qquad z\in\mathbb C_+ .
\end{equation}
Here the boundary value
\[
J_+\bigl[A(\mathfrak u_0,t,z)h\bigr],
\qquad h\in L^1(0,\infty),
\]
is understood in the sense of Lemma \(\ref{coro a.7}\).

Then there exists a unique function
\[
g_f(t)\in L^2_+(\mathbb R)
\]
such that \(G_f(t,\cdot)\) is the Hardy extension of \(g_f(t)\), namely
\begin{equation}
\label{A.8.2}
G_f(t,z)
=
\frac1{2\pi}
\int_0^\infty
e^{iz\xi}\widehat{g_f(t)}(\xi)\,d\xi,
\qquad z\in\mathbb C_+ .
\end{equation}
Moreover,
\begin{equation}
\label{A.8.3}
\|g_f(t)\|_{L^2(\mathbb R)}
\le
\|f\|_{L^2(0,\infty)}
\qquad\text{for every }t\in\mathbb R.
\end{equation}
\end{lemma}

\begin{proof}
We divide the proof into several steps.

\medskip
\noindent
\textbf{Step 1: The smooth case.}
Assume first that
\[
u_0\in\mathcal S_+(\mathbb R).
\]
Then
\[
\mathfrak u_0=\mathcal F_{d,u_0}u_0\in\mathcal{S}_{*}:=\{\varphi=g|_{(0,\infty)}:\ g\in \mathcal{S}(\mathbb R),\
\operatorname{supp}g\subset[0,\infty)\}.
\]
For \(f\in L^2(0,\infty)\), define
\[
v_f:=\mathcal F_{d,u_0}^{-1}f\in L^2_+(\mathbb R).
\]
We first introduce two auxiliary functions. Define
\begin{equation}
\label{A.8.4}
H_f(t,z)
:=
\frac{1}{2i\pi}
I_+\left[
\left(X^*+2tL_{u_0}-z\operatorname{Id}\right)^{-1}v_f
\right],
\qquad z\in\mathbb C_+,
\end{equation}
and
\begin{equation}
\label{A.8.5}
S_f(t,z)
:=
e^{it\partial_x^2}\mathcal F^{-1}f(z)
=
\frac{1}{2\pi}
\int_0^\infty e^{-it\lambda^2+iz\lambda}f(\lambda)\,d\lambda .
\end{equation}
Lemma \ref{lemma:lax-resolvent-hardy-bound} implies that \(H_f(t,\cdot)\) is the Hardy extension of a unique
function
\[
h_f(t)\in L^2_+(\mathbb R),
\]
and
\begin{equation}
\label{A.8.6}
\|h_f(t)\|_{L^2(\mathbb R)}
=
\|v_f\|_{L^2(\mathbb R)}.
\end{equation}
By the \(L^2\)-isometry of the distorted Fourier transform,
\begin{equation}
\label{A.8.7}
\|v_f\|_{L^2(\mathbb R)}
=
\frac{1}{\sqrt{2\pi}}\|f\|_{L^2(0,\infty)}.
\end{equation}
Similarly, \(S_f(t,\cdot)\) is the Hardy extension of
\[
s_f(t):=e^{it\partial_x^2}\mathcal F^{-1}f\in L^2_+(\mathbb R),
\]
and Plancherel's theorem for the ordinary Fourier transform gives
\begin{equation}
\label{A.8.8}
\|s_f(t)\|_{L^2(\mathbb R)}
=
\frac{1}{\sqrt{2\pi}}\|f\|_{L^2(0,\infty)}.
\end{equation}
For \(f\in\mathcal S_*\), we claim that
\begin{equation}
\label{A.8.9}
G_f(t,z)
=
H_f(t,z)-S_f(t,z),
\qquad z\in\mathbb C_+.
\end{equation}
Indeed, Proposition \ref{prop:Xstar_under_Fd} gives
\[
\mathcal F_{d,u_0}
\left(X^*+2tL_{u_0}\right)
\mathcal F_{d,u_0}^{-1}
=
i\partial_\lambda+B(\mathfrak u_0)+2t\lambda .
\]
Therefore, by the second resolvent identity,
\[
\begin{aligned}
&\left(i\partial_\lambda+B(\mathfrak u_0)+2t\lambda-z\right)^{-1}
-
\left(i\partial_\lambda+2t\lambda-z\right)^{-1}  \\
&\quad =
-
\left(i\partial_\lambda+B(\mathfrak u_0)+2t\lambda-z\right)^{-1}
B(\mathfrak u_0)
\left(i\partial_\lambda+2t\lambda-z\right)^{-1}.
\end{aligned}
\]
After conjugating by \(e^{it\lambda^2}\), the last term becomes exactly the expression in
\eqref{A.8.1}. The free term is precisely \eqref{A.8.5}. Hence \eqref{A.8.9} holds for
\(f\in\mathcal S_*\).

We now extend \eqref{A.8.9} to every \(f\in L^2(0,\infty)\). Let
\[
f_j\in\mathcal S_*,
\qquad
f_j\to f
\quad\text{in }L^2(0,\infty).
\]
By \eqref{A.8.6} and \eqref{A.8.7}, applied to \(f_j-f\), we have
\begin{equation}
\label{A.8.10}
h_{f_j}(t)\to h_f(t)
\qquad\text{in }L^2_+(\mathbb R).
\end{equation}
By \eqref{A.8.8},
\begin{equation}
\label{A.8.11}
s_{f_j}(t)\to s_f(t)
\qquad\text{in }L^2_+(\mathbb R).
\end{equation}
For each fixed \(z\in\mathbb C_+\), the point-evaluation functional
\[
g\mapsto
\frac1{2\pi}\int_0^\infty e^{iz\xi}\widehat g(\xi)\,d\xi
\]
is continuous on \(L^2_+(\mathbb R)\). Hence \eqref{A.8.10} and \eqref{A.8.11} imply
\[
H_{f_j}(t,z)-S_{f_j}(t,z)
\longrightarrow
H_f(t,z)-S_f(t,z)
\qquad\text{for every fixed }z\in\mathbb C_+.
\]
On the other hand, the free resolvent estimate
\[
\|(i\partial_\lambda-z)^{-1}g\|_{L^\infty}
\le
C_z\|g\|_{L^2}
\]
gives
\[
e^{it\lambda^2}(i\partial_\lambda-z)^{-1}e^{-it\lambda^2}f_j
\to
e^{it\lambda^2}(i\partial_\lambda-z)^{-1}e^{-it\lambda^2}f
\qquad\text{in }L^\infty(0,\infty).
\]
Since
\[
\|B(\mathfrak u_0)r\|_{L^1}
\le
C\|\mathfrak u_0\|_{L^2}^2\|r\|_{L^\infty},
\]
the corresponding \(L^1\)-inputs in \eqref{A.8.1} converge in \(L^1(0,\infty)\). By
Lemma \ref{coro a.7},
\[
G_{f_j}(t,z)\to G_f(t,z)
\qquad\text{for every fixed }z\in\mathbb C_+.
\]
Passing to the limit in \eqref{A.8.9}, we obtain
\begin{equation}
\label{A.8.12}
G_f(t,z)
=
H_f(t,z)-S_f(t,z),
\qquad z\in\mathbb C_+,
\end{equation}
for every \(f\in L^2(0,\infty)\).

Define
\begin{equation}
\label{A.8.13}
g_f(t):=h_f(t)-s_f(t)\in L^2_+(\mathbb R).
\end{equation}
Then \eqref{A.8.12} says precisely that \(G_f(t,\cdot)\) is the Hardy extension of
\(g_f(t)\). Moreover, by \eqref{A.8.6}, \eqref{A.8.7}, and \eqref{A.8.8},
\begin{equation}
\label{A.8.14}
\begin{aligned}
\|g_f(t)\|_{L^2(\mathbb R)}
&\le
\|h_f(t)\|_{L^2(\mathbb R)}
+
\|s_f(t)\|_{L^2(\mathbb R)}  \\
&\le
\sqrt{\frac{2}{\pi}}\,
\|f\|_{L^2(0,\infty)}
\le
\|f\|_{L^2(0,\infty)}.
\end{aligned}
\end{equation}
This proves the lemma in the smooth case.

\medskip
\noindent
\textbf{Step 2: Approximation of the potential.}
We now consider
\[
u_0\in L^2_+(\mathbb R).
\]
Choose
\[
u_n\in\mathcal S_+(\mathbb R)
\]
such that
\[
u_n\to u_0
\qquad\text{in }L^2_+(\mathbb R).
\]
Set
\[
\mathfrak u_n:=\mathcal F_{d,u_n}u_n,
\qquad
\mathfrak u_0:=\mathcal F_{d,u_0}u_0.
\]
By Proposition \(\ref{proposition 3.6}\), since \(u_n\in\mathcal S_+(\mathbb R)\),
we have
\[
\mathfrak u_n\in\mathcal S_*.
\]
By Proposition \ref{prop:continuity_Fd}, we infer
\begin{equation}
\label{A.8.15}
\mathfrak u_n\to\mathfrak u_0
\qquad\text{in }L^2(0,\infty).
\end{equation}
For each \(n\), define
\begin{equation}
\label{A.8.16}
G_{n,f}(t,z)
:=
-\frac{1}{2i\pi}
J_+\left[
A(\mathfrak u_n,t,z)
e^{-it\lambda^2}
B(\mathfrak u_n)
e^{it\lambda^2}
(i\partial_\lambda-z)^{-1}
e^{-it\lambda^2}f
\right].
\end{equation}
By the smooth case, there exists
\[
g_{n,f}(t)\in L^2_+(\mathbb R)
\]
whose Hardy extension is \(G_{n,f}(t,\cdot)\), and
\begin{equation}
\label{A.8.17}
\|g_{n,f}(t)\|_{L^2(\mathbb R)}
\le
\|f\|_{L^2(0,\infty)}
\qquad\text{for every }n.
\end{equation}
We next prove pointwise convergence. Fix \(z\in\mathbb C_+\), and define
\[
R_z:=(i\partial_\lambda-z)^{-1},
\qquad
r_z:=e^{it\lambda^2}R_ze^{-it\lambda^2}f.
\]
By Proposition \ref{prop 2},
\begin{equation}
\label{A.8.18}
\|r_z\|_{L^\infty(0,\infty)}
\le
C_z\|f\|_{L^2(0,\infty)}.
\end{equation}
For \(v\in L^2(0,\infty)\) and \(r\in L^\infty(0,\infty)\), the expression
\[
B(v)r
=
-\frac1{4\pi}
\left(
ivH(\overline v\,r)-|v|^2r
\right)
\]
belongs to \(L^1(0,\infty)\), and
\begin{equation}
\label{A.8.19}
\|B(v)r\|_{L^1}
\le
C\|v\|_{L^2}^2\|r\|_{L^\infty}.
\end{equation}
Moreover, if \(v_n\to v\) in \(L^2(0,\infty)\), then
\begin{equation}
\label{A.8.20}
\|B(v_n)r-B(v)r\|_{L^1}
\le
C
\|v_n-v\|_{L^2}
\bigl(\|v_n\|_{L^2}+\|v\|_{L^2}\bigr)
\|r\|_{L^\infty}.
\end{equation}
Indeed, the Hilbert-transform part is estimated by
\[
\begin{aligned}
&\|v_nH(\overline{v_n}r)-vH(\overline v r)\|_{L^1} \\
&\quad\le
\|(v_n-v)H(\overline{v_n}r)\|_{L^1}
+
\|vH((\overline{v_n}-\overline v)r)\|_{L^1} \\
&\quad\le
C\|v_n-v\|_{L^2}
\bigl(\|v_n\|_{L^2}+\|v\|_{L^2}\bigr)
\|r\|_{L^\infty},
\end{aligned}
\]
and the multiplication term is estimated similarly.

Applying \eqref{A.8.20} with
\[
v_n=\mathfrak u_n,
\qquad
v=\mathfrak u_0,
\qquad
r=r_z,
\]
we obtain
\begin{equation}
\label{A.8.21}
e^{-it\lambda^2}B(\mathfrak u_n)r_z
\longrightarrow
e^{-it\lambda^2}B(\mathfrak u_0)r_z
\qquad\text{in }L^1(0,\infty).
\end{equation}
By Lemma \ref{coro a.7}, applied with
\[
h_n:=e^{-it\lambda^2}B(\mathfrak u_n)r_z,
\qquad
h:=e^{-it\lambda^2}B(\mathfrak u_0)r_z,
\]
and using \eqref{A.8.15} and \eqref{A.8.21}, we get
\begin{equation}
\label{A.8.22}
G_{n,f}(t,z)\to G_f(t,z)
\qquad\text{for every fixed }z\in\mathbb C_+.
\end{equation}

\medskip
\noindent
\textbf{Step 3: Passage to the \(L^2_+\)-limit.}
The sequence \((g_{n,f}(t))_{n\ge1}\) is bounded in the Hilbert space
\(L^2_+(\mathbb R)\) by \eqref{A.8.17}. Hence every subsequence has a weakly convergent
subsequence in \(L^2_+(\mathbb R)\). Let
\[
g_{n_j,f}(t)\rightharpoonup g_f(t)
\qquad\text{weakly in }L^2_+(\mathbb R).
\]
For each fixed \(z\in\mathbb C_+\), the point-evaluation functional
\[
g\mapsto
\frac1{2\pi}
\int_0^\infty e^{iz\xi}\widehat g(\xi)\,d\xi
\]
is continuous on \(L^2_+(\mathbb R)\). Therefore the Hardy extensions of
\(g_{n_j,f}(t)\) converge pointwise to the Hardy extension of \(g_f(t)\). Since those
Hardy extensions are precisely \(G_{n_j,f}(t,\cdot)\), the pointwise convergence
\eqref{A.8.22} shows that
\[
G_f(t,z)
=
\frac1{2\pi}
\int_0^\infty e^{iz\xi}\widehat{g_f(t)}(\xi)\,d\xi,
\qquad z\in\mathbb C_+.
\]
Thus \(G_f(t,\cdot)\) is the Hardy extension of \(g_f(t)\). 

Finally, by weak lower semicontinuity and \eqref{A.8.17},
\[
\|g_f(t)\|_{L^2(\mathbb R)}
\le
\liminf_{j\to\infty}
\|g_{n_j,f}(t)\|_{L^2(\mathbb R)}
\le
\|f\|_{L^2(0,\infty)}.
\]
This proves \eqref{A.8.3}, and the proof is complete.
\end{proof}
\end{appendix}

\end{document}